\documentclass[12pt]{article}

\usepackage{amsmath,amsfonts,amssymb,mathabx,txfonts}
\usepackage{mathrsfs,upgreek,calligra}
\usepackage{dsfont}
\usepackage[utf8]{inputenc}
\usepackage[T1]{fontenc}
\usepackage{lmodern} \usepackage{ucs}
\usepackage{fullpage}

\usepackage{graphicx}
\usepackage{caption}
\usepackage{subcaption}
\usepackage[all]{xy}

\usepackage{array}
\usepackage{multicol}
\usepackage{multirow}
\usepackage{color}
\usepackage{authblk}
\usepackage{esint}
\usepackage{enumitem,hyperref}

\usepackage{accents}

\usepackage{tikz}
\usetikzlibrary{decorations}
\usetikzlibrary{decorations.pathreplacing}
\usetikzlibrary{arrows}

\newtheorem{lemma}{Lemma}[section]
\newtheorem{theo}[lemma]{Theorem}
\newtheorem{rmk}[lemma]{Remark}
\newtheorem{proposition}[lemma]{Proposition}
\newtheorem{defin}[lemma]{Definition}
\newtheorem{coro}[lemma]{Corollary}
\newtheorem{example}[lemma]{Example}

\DeclareMathAlphabet{\mathpzc}{OT1}{pzc}{m}{it}
\DeclareMathAlphabet{\mathcalligra}{T1}{calligra}{m}{n}
\newcommand*{\overbigdot}[1]{%
   \accentset{\mbox{\large\bfseries .}}{#1}}

\makeatletter
\def\namedlabel#1#2{\begingroup
    #2%
    \def\@currentlabel{#2}%
    \phantomsection\label{#1}\endgroup
}
\makeatother

\makeatletter
\renewcommand*{\eqref}[1]{%
  \hyperref[{#1}]{\textup{\tagform@{\ref*{#1}}}}%
}
\makeatother

\newcommand{\QED}{\mbox{}\hfill \raisebox{-0.2pt}{\rule{5.6pt}{6pt}\rule{0pt}{0pt}} \medskip\par}

\newcommand{\ds}{\displaystyle}
\newcommand{\ud}{\, {\mathrm{d}}}

\title{Effective Dissipation and Asymptotic Stability in a Finite-Level Quantum  System
 with Wave-Induced Memory
}
 
\author[1]{Thierry~Goudon\thanks{ {\tt thierry.goudon@univ-cotedazur.fr}}}
\date{}

\affil[1]{\small Universit\'e C\^ote d'Azur, CNRS, LJAD,

Parc Valrose, F-06108 Nice, France}

\begin{document}
\maketitle

\begin{abstract}
We study  quantum systems modeled by Liouville-Von Neumann equations, with a finite number of energy levels. The system is coupled to a wave equation, intended to reproduce environmental effects on the quantum system. The whole dynamic is energy conservative, but damping effects are expected to be embodied into the energy exchanges with the environment. We analyse how these effects can effectively damp quantum coherences. We exhibit key mechanisms and structural properties
for such a damping to arise, with a neat distinction between pure and mixed cases.
\end{abstract}

\vspace*{.5cm}
{\small
Kawashima-Shizuta condition. Hypocoercivity. 
\\[.3cm]

\noindent{\bf Math.~Subject Classification.} 
35Q40 
}

\section{Introduction}

The state of a quantum system is described by means of  a density matrix $ \rho$:
$\rho$ is a hermitian matrix; 
 the diagonal elements -- the quantum populations --   are interpreted as the 
 probability of finding the system in a given energy level while
 the off-diagonal elements --  the quantum coherences -- describe the connections between the different energy levels
 and the ability of the system in occupying 
 multiple states simultaneously, an effect which is crucial 
 in quantum technologies.
 The density matrix 
  evolves according to the Liouville-Von Neumann equation
 \begin{equation}\label{Liouv}
i \hbar \frac{\ud}{\ud t} \rho = [H, \rho]
\end{equation}
where the hamiltonian $H$ splits into two parts:
a 
  free hamiltonian and 
  a part depending on an interaction potential $\Phi$:
  \[H=H_{\mathrm{free}} +\Phi V,\]
  with $V$ an hermitian matrix.
  The coefficients of the hermitian matrices $H_{\mathrm{free}}$ and $V$ are constant, but 
  $\Phi$ is a time dependent function.
Interactions with the environment, embodied in the definition of the potential $\Phi$,  
are expected to
cause decoherence, leading  the coherences to fade.
This effect is often compared  
to friction in classical 
mechanics, that leads a particle to slow down \cite{CCT}.
This is what we want to analyse.

The question we address belongs to the field of open systems in classical and quantum mechanics:
the particle together with the environment form a conservative system endowed with an Hamiltonian structure, 
but it is expected that energy is eventually dissipated in the ``large'' environment at the expense of the particle, thus causing 
 the damping of momentum or coherences.
 We refer the reader to \cite{CL2,CL,JP2} for formalization of such phenomena.
 In phenomenological models of
open quantum systems, this transfer is usually encoded directly in a
dissipative master equation, see, for instance,
\cite{BreuerPetruccione,CL2,CL,JP2}.  Here we adopt a different point of
view.  The finite-level system and its environment are described together by a
\emph{energy-conservative} evolution, and the effective irreversibility of the subsystem is
expected to result from the propagation of energy towards spatial infinity: radiative relaxation is  the mechanisms by which the quantum
system transfers energy to its environment. 
The purpose of this paper is to make this mechanism quantitative in a simple
Hamiltonian model and to identify the finite-dimensional
structure responsible for the decay.
 Of course, the analysis of the damping 
 largely depends on the modeling adopted for describing the environment: it turns out that  
 dealing with a vibrational field, embodying the environment into a mere scalar wave equation, has received 
 a lot of attention, see  \cite{BdB,KKS,KKSb,Leg,Soffer}.
 An extension of such models to the quantum framework with a quantum particle described 
 by means of the Schr\"odinger equation is discussed in 
 \cite{Simple,SRN3,Vi4,Vi3}.
  Therefore, in what follows the coupling with the environment is  obtained by setting 
 \[
 \Phi(t)=\ds\int_{\mathbb R^3} \sigma(z)\psi(t,z)\ud z\]
 where $\sigma$ is a given form function and $\psi$ satisfies the wave equation
 \begin{equation}\label{wave}
 (\partial^2_t-c^2\Delta)\psi(t,z)=-c^2\sigma(z)\mathrm{Tr}(\rho V),\end{equation}
  with wave speed $c>0$.
In the right hand side of \eqref{wave}, 
$\mathrm{Tr}(\rho V)$, which thus corresponds to the observable that acts as a source for the environment field,
 can be interpreted as the expectation value of a dipole-type transition operator.
 Throughout the paper, $\sigma$ is supposed to be a non negative radially symmetric $C^\infty_c$ function.
Equation \eqref{wave} is set for $z\in \mathbb R^3$ and that the dimension is 3 will be further commented 
later on.
 In \cite{twostates}, we use such a coupling to establish that quantum decoherence holds in the large time asymptotics 
 for \eqref{Liouv}, when the 
 quantum system has only two energy levels: in this case, 
 which is reminicient to the standard spin-boson system \cite[Chapter~5]{CCT}, \cite{Leg},
 as time goes to infinity the 
 quantum system 
 tends to a coherence-free ground state with the populations concentrated on the state of minimal energy.  In fact the analysis of this simple situation is quite close to the case of a single classical particle 
 and we can proceed with arguments quite directly inspired from \cite{BdB,KKS,KKSb}.
 This simple situation already permits to bring out the 
 critical role of the coupling parameters and the wave speed.

 Although the coupled system is conservative, its reduced atomic dynamics has
a genuine dissipative signature.  
It can be understood by reformulating the coupling in order to make
a Volterra memory term appear.  
The leading effects embodied in the memory term are two-fold: on the one-hand it modifies the natural frequencies of the system, on the other hand, it dissipates certain combination of the 
components of the density matrix.
The perturbation on the frequencies should remain moderate enough to keep  conservative features, 
and this is measured by means of a \emph{weak-coupling} assumption.
Next, the damping has to be propagated  to the entire density matrix.
On this issue, the passage from two to several energy levels is not merely a question of
larger matrices.  The field observes only a scalar quantity, so that
the effective damping has rank-one.   Most coherences are not directly
dissipated; decay can occur only if the conservative oscillations repeatedly
transfer every relevant mode into the observed direction.  This is precisely
the finite-dimensional counterpart of the compensating mechanism in the
Kawashima--Shizuta  and in the hypocoercivity theories
\cite{nanard2,nanard,KS,Vill}. 
Here, this can be naturally expressed
through Kalman observability or the Popov--Belevitch--Hautus criterion, an approach strongly inspired from 
\cite{AAM,AAM2,BZ,Hautus,Popov}.  By using these techniques, not only we provide 
 qualitative 
 convergence results, but we also establish sharp decay estimates with explicit coefficients and   decay rate.

The paper is organized as follows.
In Section~\ref{sec:prelim},  we collect some basic facts about density matrices and we establish 
fundamental conservation properties of the system \eqref{Liouv}-\eqref{wave}.
It conserves energy and all quantities $\mathrm{Tr}(\rho^m)$, $m\in \mathbb N\setminus\{0\}$.
Section~\ref{sec:Vmod} focuses on the   three-level V-system, consisting of one ground level
coupled to two excited levels, a standard configuration in quantum optics
\cite{CCT}.  This low-dimensional model already displays the two complementary
features of the problem. 
We start by identifying ground states solutions, obtained by minimization of the energy functional 
under constraints. They are natural candidates for being large time attractors; however, we shall see 
that this intuition is correct only in the case of pure states, when the density matrix is rank-one.
Owing to the small dimension of the system, the large time behavior can be investigated by means of quite ``elementary'' methods and direct computations. It permits us to bring out key mechanisms
 that shape the approach for dealing with larger systems.
In Section~\ref{sec:symmetrizer}, we set up the algebraic
framework in order to deal with the leading linear part of the equations. It makes clearly appear the role of weak damping assumptions and 
of  compatibility conditions between the free hamiltonian and the coupling matrix $V$
 in the construction of an adapted energy 
functional.
Section~\ref{sec:generalized_V} applies this framework to generalized V-systems with arbitrarily large dimensions.
In Section~\ref{sec:comm} we further comments the obtained results and their apparent limitations, 
discussing why they are nevertheless sharp.
We end this Introduction, we a brief recap of our findings:
\begin{itemize}
\item a weak coupling assumption appears as a key structural hypothesis.
\item as far as one deals with V-systems, one observes the damping   of the ground-to-excited coherences; other quantum coherences are only damped ``in average''.
\item we can establish that the asymptotic regime is the ground state in the case of pure quantum states.
\item in this case, the convergence can be shown to hold exponentially fast, with a rate of order $\mathscr O(1/c)$, at least for large enough $c$, meaning with a strong capability of the environment to evacuate energy at infinity. 
\end{itemize}

 \section{Preliminaries}\label{sec:prelim}
 
 In all what follows, given a complex number $z=a+ib$, with $a,b\in \mathbb R$, we denote by $z^*=a-ib$ its conjugate, and given a (non necessarily square) matrix $A$ with complex entries, we denote by $A^*$ its transpose-conjugate. When the matrix $A$ has real coefficients, we use the notation  $A^\top$ for its transpose.
 We identify $u=(z_0,z_1,...,z_{n})\in \mathbb C^{n+1}$ with its column matrix; accordingly we get
 $u^*u=\sum_{j=0}^{n} |z_j|^2\in \mathbb R$ while $uu^*$ is the $\mathcal M_{n+1}(\mathbb C)$
 with entries $z_i z_j^*$. 
 As usual, we denote $e_0,...,e_n$ the canonical basis in $\mathbb C^{n+1}$, with the $j$th component of $e_k$ given by $\delta_{j,k}$. 
Depending on the context, we indifferently denote $u\cdot v$ or $u^*v$ (respectively $u^\top v$) the inner product $\sum_{j=0}^n u_j^*v_j$ of two vectors in $\mathbb C^{n+1}$  (respectively in $\mathbb R^{n+1}$).
  We can also construct the matrix $uv^*$, which is rank-one,  with $\mathrm{Ran}(uv^*)=\mathrm{Span}(u)$ and 
 $\mathrm{Ker}(uv^*)=(\mathrm{Span}(v))^\perp$.
 \\
 
 Here we consider the case where the free hamiltonian $H_{\mathrm{free}}$ is simply given by 
 $\hbar$ times
 a diagonal  matrix $\Omega=\mathrm{diag}(\omega_0,..,\omega_n)$, 
 where 
 \begin{equation}\label{omega}
 0\leq \omega_0<\omega_1<...<\omega_n.\end{equation}
 Hence, $\hbar \omega_j$ gives the amount of energy associated to the $j$th level.
 The coupling matrix $V$ is supposed to be a real-symmetric $(n+1)\times (n+1)$ matrix, with vanishing diagonal terms:
 \[
 V_{k,\ell}\in \mathbb R,\qquad V_{k,\ell}=V_{\ell, k},\qquad 
 V_{k,k}=0.\]
 Given a hermitian matrix $\rho$, 
 we  consider the observables 
$$ \mathrm{Tr}(\rho \Omega)= \ds\sum_{j=0}^n \omega_j \rho_{j,j},$$
and
\[\begin{array}{lll}
\mathrm{Tr}(\rho V)
&=&
\ds\sum_{j,k=0}^n  \rho_{j,k}V_{k,j}
=
\ds\frac12\ds\sum_{j,k=0}^n  (\rho_{j,k}+\rho_{k,j})V_{k,j}
=\ds\frac12\ds\sum_{j,k=0}^n  (\rho_{j,k}+\rho_{j,k}^*)V_{k,j}
\\[.4cm]&&=\ds\sum_{j,k=0}^n  \mathrm{Re}(\rho_{j,k})V_{k,j}
=
 \mathrm{Re}\Big(\mathrm{Tr}(\rho V)\Big)
,\end{array}\]
 which are both real valued.

 \subsection{Properties of density matrices}
 
 For a matrix $\rho\in \mathcal M_{n+1}(\mathbb C)$ to be a density matrix 
 it should  be hermitian $\rho^*=\rho$, non negative (for any $\xi\in \mathbb C^{n+1}$, we have $\rho \xi\cdot \xi\geq 0$; equivalently $\rho$ being diagonalizable, all its eigevalues are real and non negative) and such that 
 \[
 \mathrm{Tr}(\rho)=1.\]
 That $\rho$ is non negative  implies the estimates
\[
|\rho_{j, j+1}|^2\leq \rho_{j, j}\rho_{j+1, j+1}.\]
Indeed, that the density matrix $\rho$ is non negative implies  
the non negativity of the $2\times2$ hermitian matrices 
\[M_j=
\begin{pmatrix}\rho_{j, j} & \rho_{j, j+1}
\\
 \rho_{j, j+1}^*& \rho_{j+1, j+1}
\end{pmatrix}\]
(since the quadratic form $(\alpha, \beta)\mapsto \rho(\alpha e_j +\beta e_{j+1})\cdot (\alpha e_j +\beta e_{j+1})=M_j\begin{pmatrix}\alpha \\ \beta\end{pmatrix}\cdot \begin{pmatrix}\alpha \\ \beta\end{pmatrix}$ should be non negative).
The diagonal terms are non negative, so that  $\mathrm{Tr}(M_j)\geq 0$ and non negativity of $M_j$ 
is now equivalent to $\mathrm{det}(M_j)=\rho_{j, j}\rho_{j+1, j+1}-|\rho_{j, j+1}|^2\geq0.$

Let $$T_2=\mathrm{Tr}(\rho^2)=\ds\sum_{j,k=0}^n|\rho_{j,k}|^2.$$  
Cauchy-Schwarz' inequality yields $$\ds\sum_{j=0}^n\rho_{j,j}=1\leq (n+1)\ds\sum_{j=0}^n\rho_{j,j}^2
\leq 
(n+1)T_2.$$
Moreover, $\rho$ 
being a non negative hermitian matrix, it can be diagonalized in orthonormal basis, with real eigenvalues $\lambda_j\geq 0$. Then, we get $\mathrm{Tr}(\rho)=\sum_{j=0}^n \lambda _j =1$, so that $0\leq \lambda_j\leq 1$ and 
$\mathrm{Tr}(\rho^2)=\sum_{j=0}^n \lambda _j^2 \leq \sum_{j=0}^n \lambda _j=1$.
Hence, we actually have 
\[T_2\in \left[\ds\frac{1}{n+1},1\right].\]
It is interpreted as a purity index. The lower bound is reached with the uniform repartition $\rho_{j, j}=\frac{1}{n+1}(1,...,1)$, 
$\rho_{j, j+1}=0$ and the upper bound,  
referred to as the pure-state case, is
attained  by  rank-one orthogonal projectors,
 the case where populations are concentrated on a single state $\rho_{j,j}=\delta_{j,j_0}$ being a special example.

\begin{proposition}
Let $\rho\in \mathcal M_{n+1}(\mathbb C)$ be  density matrix.
The following assertions are equivalent.
\begin{enumerate}
\item $\mathrm{Tr}(\rho^2)=1$,
\item $\rho=\rho^2$,
\item There exists a vector $u\in \mathbb C^{n+1}$
such that $u^*u=1$ and $\rho=uu^*$.
\end{enumerate}
\end{proposition}

\noindent
{\bf Proof.}
(3) implies (2) since $(uu^*)^2=u (u^* u)u^*=uu^*$. (2) implies (1) since $\mathrm{Tr}(\rho)=1$.
Suppose (1). The matrix $\rho$ being hermitian, there exists $P$ such that $PP^*=P^*P=\mathbb I$ and
$\rho=PDP^*$, with $D=\mathrm{diag}(\lambda_0,...,\lambda_n)$, the diagonal of the eigenvalues of $\rho$. We know that $\lambda_j\geq 0$ for all $j\in \{0,...,n\}$, and since $\mathrm{Tr}(\rho)=\sum_{j=0}^n\lambda_j=1$, we have  $\lambda_j\in [0,1]$.
By (1), we get $\mathrm{Tr}(\rho^2)=\sum_{j=0}^n\lambda_j^2=1$, so that
\[
\ds\sum_{j=0}^n\lambda_j(1-\lambda_j)=\mathrm{Tr}(\rho)-\mathrm{Tr}(\rho^2)=0.\]
Since all terms in the sum are non negative, we deduce that $\lambda_j(1-\lambda_j)=0$ and thus 
either $\lambda_j=0$ or $\lambda_j=1$.
Since the sum of the eigenvalues is 1, only one of them can be equal to 1, the others vanish.
It implies that $\rho^2=PD^2P^*=PDP^*=\rho$ and (2) holds.
Moreover, assuming $\lambda_{j_0}=1$ for a certain $j_0\in \{0,...,n\}$, 
 we denote by $u$ the $j_0$th column of $P$ and we get
 $\rho=PDP^*=uu^*$ while $P^*P=\mathbb I$ implies $u^*u=1$, so that (3) holds too.
\QED
 
 \subsection{Some important quantities}
 \label{Nota}
 
 Throughout the paper, 
 we suppose that 
 the coupling in \eqref{wave} and the definition of the potential is driven by a function 
 $\sigma:\mathbb R^3\to [0,\infty)$ which is radially symmetric and such that 
 \begin{equation}
 \label{hyp_supp}
 \sigma\in C^\infty_c(\mathbb R^3),\qquad \mathrm{supp}(\sigma)=B(0,R_*),
 \end{equation}
 \begin{equation}
 \label{hyp_hatsig}
 \textrm{for any $k\in \mathbb R^3$, $\widehat \sigma(k)\neq0$.}
 \end{equation}
 For further purposes, let us introduce $\Gamma$, solution of $-\Delta \Gamma=\sigma$, that can be obtained by means of its Fourier transform
$\widehat \Gamma(\xi)=\frac{\widehat \sigma(\xi)}{|\xi|^2}$,
and set 
\begin{equation}\label{def_kappa}
\kappa=\ds\int_{\mathbb R^3} \Gamma\sigma\ud z=
\ds\int_{\mathbb R^3} 
\ds\frac{|\widehat \sigma(k)|^2}{|k|^2}
\ds\frac{\ud k}{(2\pi)^3}>0.
\end{equation}
n what follows we will work with the space $\overbigdot H^1(\mathbb R^3)$, obtained as the adherence of $C^\infty_c(\mathbb R^3)$, for $\|\psi\|^2_{\overbigdot H^1}=\int_{\mathbb R^3} |\nabla\psi|^2\ud z$. 
It is a subspace of $L^{6}(\mathbb R^3)$.

As it will be detailed later on, the coupling with the wave equation makes the following 
 kernel $k_c$ appear
 \[
 t\longmapsto k_c(t)=\ds\int_{\mathbb R^3}
  \frac{c\sin(c t|k|)}{|k|} |\widehat \sigma(k)|^2
\ds\frac{\ud k}{(2\pi)^3}.
\]
Indeed, 
 the coupling induces a \emph{memory effect}: the potential at time $t$ depends on the past states of the system. 
We collect in  the following claim the critical properties of this kernel  \cite[Lemma~14]{dBGV},
 \cite[Section~2.4]{Vi2}, \cite{twostates} (note that $k_c(t)=ck_1(ct)$).
 
 \begin{lemma}\label{lem_AV}
 Let $\sigma$ satisfy \eqref{hyp_supp}.
 Then, $ k_c\in L^1([0,\infty))$, with 
 $$\ds\int_0^\infty  k_c(t)\ud t=
 \kappa=\int_{\mathbb R^3}\frac{|\widehat \sigma(k)|^2}{|k|^2}\frac{\ud k}{(2\pi)^3}\in (0,\infty),$$
 $t\mapsto  k_c(t)$ is compactly supported in $[0,2R_*/c]$, 
and  it lies in $L^\infty([0,\infty))$,  with $\|k_c\|_{L^\infty}\leq c M_*$ 
($M_*$ being independent of $c$). 
Moreover, 
 $$\ds\int_0^\infty tk_c(t)\ud t=\ds\frac
1{4\pi c}\left( \ds\int_{\mathbb R^3} \sigma(z)\ud z\right)^2=\ds\frac{\upgamma}{c}>0.$$
 \end{lemma}

\begin{rmk}
The variable $z$ does not necessarily has the meaning of a space dimension
and one may wonder what happens when the wave equation is set for $z\in \mathbb R^N$.
Firstly, the finiteness of $\kappa$ requires $N\geq 3$.
Secondly, the damping properties of the system are embodied into the first order moment of $k_c$ 
and, 
as already observed in \cite{BdB}, it makes the case $N=3$ specific since otherwise this quantity  
 vanishes, see also \cite{twostates}. 
\end{rmk}

These quantities appear by revisiting the  formulation of the problem.
We split the solution of  the wave equation \eqref{wave}
$$\psi=\psi_I+\psi_S
$$
where $\psi_I$ depends only on the initial data $\Psi_0,\Psi_1$
$$
\widehat\psi_I(t,k)=
\widehat{\Psi_0}(k) \cos(c|k|t) + \widehat{\Psi_1}(k) \frac{\sin(c|k|t)}{c|k|}.
$$
It satisfies the free wave equation with initial data $\Psi_0,\Psi_1$, namely
$$
(\partial^2_t-c^2\Delta)\psi_I=0,\qquad (\psi_I,\partial_t\psi_I)\big|_{t=0}=(\Psi_0,\Psi_1).$$
 Accordingly, $\psi_S$ satisfies 
 \[
 (\partial^2_t-c^2\Delta)\psi_S=-c^2\sigma\mathrm{Tr}(\rho V),\qquad (\psi_S,\partial_t\psi_S)\big|_{t=0}=(0,0).\]
 Its Fourier transform is given by
 \[
 \widehat\psi_S(t,k)=
 - \ds\int_0^t  \widehat{\sigma}(k) \frac{c\sin(c s|k)}{|k|} \mathrm{Tr}(\rho V)(t-s) \ud s.\]
 Therefore, we can decompose the potential as follows
 \begin{equation}\label{defPhiIS}
 \begin{array}{lll}
 \Phi&=&\Phi_I+\Phi_S,
 \\[.4cm]
 \Phi_I(t)&=&\ds\int_{\mathbb R^3}\sigma(z)\psi_I(t,z)\ud z=\ds\int_{\mathbb R^3}\widehat\sigma(k)\widehat\psi_I^*(t,k)\ds\frac{\ud k}{(2\pi)^3},
 \\[.4cm]
 \Phi_S(t)&=&\ds\int_{\mathbb R^3}\sigma(z)\psi_S(t,z)\ud z
 =
 \ds\int_{\mathbb R^3}\widehat\sigma(k)\widehat\psi_S^*(t,k)\ds\frac{\ud k}{(2\pi)^3}
 \\[.4cm]&=&-\ds\int_0^t k_c(t-s)\mathrm{Tr}(\rho V)(s)\ud s.
 \end{array}\end{equation}
 
 From now on, we assume 
 \begin{equation}\label{hyp_CIW}\Psi_0 \in \overbigdot H^1(\mathbb R^3),\, \Psi_1\in L^2(\mathbb R^3),\,
 \mathrm{supp}(\Psi_0,\Psi_1)\subset B(0,R_I).\end{equation}
In turn, the initial data term does not influence that much
the large time behavior.

\begin{lemma}\label{CIW}
Suppose \eqref{hyp_supp} and \eqref{hyp_CIW}.
Then $\Phi_I$ is supported in $[0,\frac{R_*+R_I}{c}]$.
\end{lemma}

\noindent
{\bf Proof.}
This is a direct consequence of Huygens' principle
which tells us that $\psi_I$ is supported in $\{c|t|-R_I\leq |z|\leq c |t|+R_I\}$.
\QED

 \begin{rmk}
 The support assumptions \eqref{hyp_supp} and \eqref{hyp_CIW} are certainly far form optimal.
 They greatly simplify the proofs, but very likely  they can be relaxed by using some fast decay assumptions (see for instance \cite[Lemma~3.3]{KKS}).
 \end{rmk}
 
 \subsection{Conservation laws}
 \label{sec:cons}
 
 The solutions of \eqref{Liouv}-\eqref{wave} satisfy fundamental conservation properties.

 \begin{proposition}[Energy conservation]
 Let us set 
 \[
\mathcal E =\ds\frac12 \int_{\mathbb R^3} \Big(\ds\frac1{c^2} |\partial_t \psi|^2 + |\nabla \psi|^2 \Big) \ud z 
+  \hbar  \mathrm{Tr}(\rho \Omega_{\mathrm{free}})
+ \Phi  \mathrm{Tr}(\rho V).
\]
Then, for any $t\geq 0$, we have $\mathcal E(t)=\mathcal E(0)$.
 \end{proposition}

\noindent
 {\bf Proof.}
 We get
 \[
 \begin{array}{lll}
 \ds\frac{\ud}{\ud t}\mathcal E&=&
 -\ds\int_{\mathbb R^3} \sigma \mathrm{Tr}(\rho V) \partial_t \psi \ud z 
 + \hbar\mathrm{Tr}( \dot \rho \Omega_{\mathrm{free}})
 + \mathrm{Tr}(\rho V)\ds\int_{\mathbb R^3}\sigma\partial_t\psi\ud z
 +\Phi\mathrm{Tr}(\dot \rho V)
 \\
 &=&\mathrm{Tr}\left( \ds\frac{1}{i}[\hbar \Omega_{\mathrm{free}} +\Phi V,\rho] \Omega_{\mathrm{free}}\right)+\Phi\mathrm{Tr}\left( \ds\frac{1}{i\hbar}[\hbar\Omega +\Phi V,\rho] V\right).
 \end{array}\]
 We use the elementary relations
 \[
 \mathrm{Tr}([A,B]A)= \mathrm{Tr}(ABA)-\mathrm{Tr}(BAA)=
 \mathrm{Tr}(A^2B)-\mathrm{Tr}(BA^2)=0,\]
 and 
 \begin{equation}\label{commut}\mathrm{Tr}([A,B]C)= \mathrm{Tr}(ABC)-\mathrm{Tr}(BAC)=
 \mathrm{Tr}(CAB)-\mathrm{Tr}(ACB)=\mathrm{Tr}([C,A]B),\end{equation}
 and we obtain 
 \[
 \ds\frac{\ud}{\ud t}\mathcal E=
 \ds\frac{\Phi}{i}\mathrm{Tr}([V,\rho]\Omega_{\mathrm{free}})
 + \ds\frac{\Phi}{i}\mathrm{Tr}([\Omega_{\mathrm{free}},\rho]V)
 =
 - i \Phi \Big(\mathrm{Tr}([\Omega_{\mathrm{free}},V]\rho])
 +
 \mathrm{Tr}([V,\Omega_{\mathrm{free}}]\rho])\Big)
  =
  0.
 \]
  \QED
 
 \begin{proposition}
 For  any $m\in \mathbb N$, $m\neq 0$, and any $t\geq 0$, we have 
 $$\mathrm{Tr}(\rho^m(t))=\mathrm{Tr}(\rho^m(0)).$$
 \end{proposition}
 
 \noindent
 {\bf Proof.}
 We compute 
 \[\begin{array}{lll}
 \ds\frac{\ud}{\ud t}\mathrm{Tr}(\rho^m)
 &=&
 \mathrm{Tr}(m\rho^{m-1}\dot \rho)
 =
 \ds\frac{m}{i\hbar}\mathrm{Tr}(\rho^{m-1}H \rho-\rho^m H)
 \\&=&\ds\frac{m}{i\hbar}\mathrm{Tr}(\rho^{m-1}H \rho)-m\mathrm{Tr}(\rho^m H)
 =\ds\frac{m}{i\hbar}\mathrm{Tr}(\rho^m H)-m\mathrm{Tr}(\rho^m H)=0.\end{array}\]
 \QED

  \begin{coro}
 For any $t\geq 0$, the matrix $\rho(t)$ has the same eigenvalues as $\rho(0)$.
 \end{coro}
 
 \noindent
 {\bf Proof.}
The characteristic polynomial of $\rho(t)$ can be expressed by means of the quantities
$\mathrm{Tr}(\rho(t))$, $\mathrm{Tr}(\rho^2(t))$,...,$\mathrm{Tr}(\rho^{n+1}(t))$, according to Newton's identities.
For instance for a $3\times 3$ density matrix,  the characteristic polynomial is given by
\[
  \det(\mu \mathbb I-\rho)
  =\mu^3-\mu^2+\frac{1-\mathrm{Tr}(\rho^2)}{2}\mu
   -\frac{1-3\mathrm{Tr}(\rho^2)+2\mathrm{Tr}(\rho^3)}{6},
\]
the zeroth order term being $-\mathrm{det}(\rho)$.
Hence  the characteristic polynomial does not change with time, and its roots neither.

In particular, if the initial data is a pure state, it remains a pure state forever: the solution has the form 
$\rho(t)=u(t)u^*(t)$, with $u:[0,\infty)\to \mathbb C^{n+1}$. 
 \QED

 \section{The V-model}
 \label{sec:Vmod}

 We consider  the V-model with only 3 energy levels,
a ground state and two  excited states, and where 
 the interaction matrix only connect the ground state and the excited states.
 Adopting a slight change of notation,  
 we write the density matrix  as follows
\[
\rho = \begin{pmatrix}
\rho_{0} & \rho_{1/2} & \rho_{3/2} \\[.3cm]
\rho_{1/2}^* & \rho_{1} & \rho_X \\[.3cm]
\rho_{3/2}^* & \rho_X^* & \rho_{2}
\end{pmatrix},
\]
making a clear distinction for the coherence between the excited levels 1 and 2.
 The Hamiltonian is thus  defined from 
\[
H_{\mathrm{free}}=\hbar \begin{pmatrix}
\omega_{0} & 0 &0 \\[.3cm]
0 & \omega_{1} & 0 \\[.3cm]
0 & 0 & \omega_{2}
\end{pmatrix},
\qquad
V= \begin{pmatrix}
0 & V_{1/2} & V_{3/2} \\[.3cm]
V_{1/2} & 0 & 0 \\[.3cm]
V_{3/2} & 0 & 0
\end{pmatrix},\]
where  the components $V_{1/2},V_{3/2}$ are supposed to be real.
 The wave equation is thus sourced by 
 $$\mathrm{Tr}(\rho V)=V_{1/2}(\rho_{1/2}+\rho_{1/2}^*)+V_{3/2}(\rho_{3/2}+\rho_{3/2}^*)
 =2\mathrm{Re}(V_{1/2}\rho_{1/2}+V_{3/2}\rho_{3/2})
 .$$
For further purposes, we introduce the following notations 
\[\omega_{1/2}=\omega_1-\omega_0,\qquad \omega_{3/2}=\omega_2-\omega_0,\qquad
\omega_X=\omega_2-\omega_1,\]
and we set
\[\begin{array}{ll}
\bar \omega=\ds\frac{\omega_0+\omega_1+\omega_2}{3},\qquad&
\sigma_\omega=\sqrt{\ds\sum_{j=0}^2 (\omega_j-\bar \omega)^2}
=\sqrt{\ds\sum_{j=0}^2 \omega_j^2-3\bar \omega^2},\\
V=(V_{1/2},V_{3/2})\in \mathbb R^2, 
\qquad&
\mathpzc{v}=|v|=\sqrt{V_{1/2}^2+V_{3/2}^2}\\
v_\omega=(V_{1/2}\omega_{1/2},V_{3/2}\omega_{3/2})\in \mathbb R^2.\end{array}\]
The following condition
\begin{equation}\label{small}
\ds\frac{2\kappa}
{\hbar}\left(\ds\frac{V^2_{1/2}}{\omega_{1/2}}+\ds\frac{V^2_{3/2}}{\omega_{3/2}} \right)<1
\end{equation}
will appear in the discussion. In can be interpreted as a weak coupling assumption.

Then, the evolution of the components of the density matrix is driven by the system
\begin{equation}\label{Vsystem}
\begin{array}{l}
\dot\rho_0=-\ds\frac2\hbar \Phi\mathrm{Im}(V_{1/2}\rho_{1/2}+V_{3/2}\rho_{3/2}),
\\[.3cm]
\dot\rho_1=\ds\frac2\hbar \Phi\mathrm{Im}(V_{1/2}\rho_{1/2}),
\\[.3cm]
\dot\rho_2=\ds\frac2\hbar \Phi\mathrm{Im}(V_{3/2}\rho_{3/2}),
\\[.3cm]
\dot\rho_{1/2}=i\omega_{1/2}\rho_{1/2}-\ds\frac i \hbar \Phi \big[V_{1/2}(\rho_1-\rho_0)+V_{3/2}\rho_X^*)\big],
\\[.3cm]
\dot\rho_{3/2}=i\omega_{3/2}\rho_{3/2}-\ds\frac i \hbar \Phi \big[V_{3/2}(\rho_2-\rho_0)+V_{1/2}\rho_X)\big],
\\[.3cm]
\dot\rho_{X}=i\omega_{X}\rho_{X}-\ds\frac i \hbar \Phi (V_{1/2}\rho_{3/2}-V_{3/2}\rho_{1/2}^*)
.\end{array}\end{equation}
where we remind the reader that 
\[
\Phi(t)=\ds\int_{\mathbb R^3} \sigma(z) \psi(t,z)\ud z,\qquad
(\partial^2_t-c^2\Delta)\psi
=-2c^2\sigma \mathrm{Re}(V_{1/2}\rho_{1/2}+V_{3/2}\rho_{3/2}).\]
Despite its simplicity, this model can be experimentally challenged in optical devices \cite{Nature}.
According to Section~\ref{sec:cons}, the V-model \eqref{Vsystem} conserves 
\[\begin{array}{l}
\mathrm{Tr}(\rho)=\rho_0+\rho_1+\rho_2=1,\\
\mathrm{Tr}(\rho^2)=
\rho_0^2+\rho_1^2+\rho_2^2
+2(|\rho_{1/2}|^2+|\rho_{3/2}|^2+|\rho_{X}|^2)
=T_2\in [1/3,1],\\
\mathrm{Tr}(\rho^3)
\in [1/9,1]
\end{array}\]
and
 the  energy functional
\[
\mathcal E =\ds\frac12 \int_{\mathbb R^3} \Big(\ds\frac{1}{c^2} |\partial_t \psi|^2 + |\nabla \psi|^2 \Big) \ud z 
+  \hbar (\omega_0 \rho_0 +\omega_1 \rho_1 +\omega_2 \rho_2 )
+ 2\Phi \, \mathrm{Re} \big( V_{1/2} \rho_{1/2} + V_{3/2} \rho_{3/2} \big).
\]

\begin{rmk}
The V-model is commonly used with  further approximations that get rid of the equation for $\rho_X$, 
intended to apply in the  
regime
$\omega_X\gg 1$: either one brutally disregards 
 $\rho_X$, replacing it by 0, or one sets
 $\rho_X=\frac\Phi{\omega_X\hbar}(V_{1/2}\rho_{3/2}-V_{3/2}\rho_{1/2}^*)$ (adiabatic approximation).
 However, in general the solution of the underlying system is not a density matrix.
 \end{rmk}

\subsection{Ground states}
\label{subsec:ground_states_fixed_spectrum}

A natural candidate for the large-time asymptotics is obtained by minimizing
the energy functional
\[
\mathcal E(\rho,\psi,\varpi)=\hbar\mathrm{Tr}(\rho\Omega_{\mathrm{free}})+\mathrm{Tr}(\rho V)\ds\int_{\mathbb R^3}\sigma\psi\ud z
+\ds\frac12
\ds\int_{\mathbb R^3}\Big(\ds\frac{\varpi^2}{c^2}+|\nabla\psi|^2\Big)\ud z
\]
over density matrices that fulfil  the constraints \[
\mathrm{Tr}(\rho)=1,\qquad \mathrm{Tr}(\rho^2)=T_2,\qquad
\mathrm{Tr}(\rho^3)=T_3.\]
Equivalently, the minimizer is sought in the set of density matrices with prescribed spectrum
\[\rho=
\left\{U \mathrm{diag}(\lambda_0,\lambda_1, \lambda_2) U^*,\ UU^*=\mathbb I\right\},\]
where the eigenvalues, entirely determined by $T_2$ and $T_3$, have been ordered in non increasing order $0\leq\lambda_2\leq \lambda_1\leq \lambda_0\leq 1$.

For a fixed
\(\rho\), minimization with respect to \(\varpi\) and \(\psi\) yields
\[
  \varpi=0,
  \qquad
  -\Delta\psi=-\sigma\operatorname{Tr}(\rho V),\]
and thus
\[
  \psi=-\Gamma\operatorname{Tr}(\rho V), \qquad \Phi=
\ds\int_{\mathbb R^3} \sigma\psi\ud z
=-\mathrm{Tr}(\rho V)\ds\int_{\mathbb R^3} \sigma\Gamma\ud z
=-\kappa\mathrm{Tr}(\rho V),
\]
which leads  to deal with the reduced energy 
\begin{equation}
\label{reduced_energy_fixed_spectrum}
  \mathscr E_{\rm red}(\rho)
  =\hbar\operatorname{Tr}(\rho\Omega_{\mathrm{free}})
   -\frac{\kappa}{2}\bigl(\operatorname{Tr}(\rho V)\bigr)^2.
\end{equation}

\begin{proposition}
\label{prop:ground_state_fixed_spectrum}
Assume \eqref{small}.
Then, the unique minimizer of \(\mathscr E\) under the constraints
$\mathrm{Tr}(\rho^2)=T_2$, $\mathrm{Tr}(\rho^3)=T_3$ is
given by the rearrangement of the prescribed spectrum 
\[
  \rho_{\mathrm {min}}
  =\operatorname{diag}(\lambda_0,\lambda_1,\lambda_2).
\]
Its energy is
$
  \mathscr E_{\min}
  =\hbar\bigl(
       \lambda_0\omega_0+
       \lambda_1\omega_1+
       \lambda_2\omega_2
     \bigr).
$
\end{proposition}

\begin{rmk}
The ground state 
 assigns the largest eigenvalue of \(\rho\) to the lowest energy level.
 As a matter of fact, if the initial state is pure, its spectrum is \((1,0,0)\), and
Proposition~\ref{prop:ground_state_fixed_spectrum} gives
\[
  \rho_{\mathrm {min}}
  =
  \begin{pmatrix}
    1&0&0\\
    0&0&0\\
    0&0&0
  \end{pmatrix}=e_0e_0^*.
\]
 We shall see that  condition \eqref{small} can be slightly  precised into
  \begin{equation}\label{small_prec}
  \ds\frac{2\kappa(\lambda_0-\lambda_2)}{\hbar}
  \left(
    \frac{V_{1/2}^2}{\omega_{1/2}}
    +
    \frac{V_{3/2}^2}{\omega_{3/2}}
  \right)<1\end{equation}
which requires a more detailed knowledge of the spectrum.
\end{rmk}

\noindent
{\bf Proof.}
Since $\rho$ derives from $ \operatorname{diag}(\lambda_0,\lambda_1,\lambda_2)$ by a unitary transformation
\[
  \rho=\sum_{\ell=0}^2\lambda_\ell u^{(\ell)} u^{(\ell)*},\qquad (u^{(0)},u^{(1)},u^{(2)})\in \mathbb C^3,\qquad
  \ds\sum_{\ell=0}^2 u^{(\ell)} u^{(\ell)*}=\mathbb I
,\]
we infer
\[  \lambda_2\leq \rho_j\leq\lambda_0,\qquad j\in \{0,1,2\}.
\]
Using \(\rho_0+\rho_1+\rho_2=\lambda_0+\lambda_1+\lambda_2=1\), we find
\begin{equation}\label{free_energy_excess}\begin{array}{lll}
  \operatorname{Tr}(\rho\Omega_{\mathrm{free}})
  -\operatorname{Tr}(\rho_{\mathrm {min}}\Omega_{\mathrm{free}})
  &=&(\omega_1-\omega_0)(\lambda_0-\rho_0)
    +(\omega_2-\omega_1)(\rho_2-\lambda_2)
  \\
  &  =&\omega_{1/2}(\lambda_0-\rho_0)
    +(\omega_{3/2}-\omega_{1/2})(\rho_2-\lambda_2)
    \geq 0.
\end{array}\end{equation}
It remains to control the non positive interaction term in
\eqref{reduced_energy_fixed_spectrum}.  Since
\[
  \operatorname{Tr}(\rho V)
  =2\operatorname{Re}
    \bigl(V_{1/2}\rho_{1/2}+V_{3/2}\rho_{3/2}\bigr),
\]
the Cauchy--Schwarz inequality gives
\begin{equation}
\label{observable_offdiag_bound}
  \bigl|\operatorname{Tr}(\rho V)\bigr|^2
  \leq4\left(
  \ds\frac{V^2_{1/2}}{\omega_{1/2}}+\ds\frac{V^2_{3/2}}{\omega_{3/2}}
  \right)\bigl(
  \omega_{1/2}|\rho_{1/2}|^2+\omega_{3/2}|\rho_{3/2}|^2\bigr).
\end{equation}
Based on the 
 spectral decomposition of $\rho$ we set
\[
  \alpha_\ell=|u^{(\ell)*}e_0|^2,
\]
so that 
\[\alpha_\ell\geq0,\qquad \ds\sum_{\ell=0}^2 \alpha_\ell=1,\qquad 
  \rho_0=\sum_{\ell=0}^2\alpha_\ell\lambda_\ell,
  \qquad
  (\rho^2)_{0}
  =\sum_{\ell=0}^2\alpha_\ell\lambda_\ell^2.
\]
By definition
\[
(\rho^2)_0=
  \rho_0^2+|\rho_{1/2}|^2+|\rho_{3/2}|^2
  =\ds\sum_{\ell=0}^2 \alpha_\ell \lambda_\ell^2
\]
yields
\[|\rho_{1/2}|^2+|\rho_{3/2}|^2=
\ds\sum_{\ell=0}^2 \alpha_\ell \lambda_\ell^2
-
\rho_0^2
\]
For any \(x\in[\lambda_2,\lambda_0]\), one has
\[
  x^2-(\lambda_0+\lambda_2)x+\lambda_0\lambda_2
  =- (\lambda_0-x)(x-\lambda_2)\leq0.
\]
In particular it holds for $x\in \{\lambda_0,\lambda_1,\lambda_2\}$.
It follows that 
\[
\ds\sum_{\ell=0}^2 \alpha_\ell \lambda_\ell^2 \leq 
(\lambda_0+\lambda_2)\ds\sum_{\ell=0}^2 \alpha_\ell \lambda_\ell-\lambda_0\lambda_2\ds\sum_{\ell=0}^2 \alpha_\ell 
=(\lambda_0+\lambda_2)\rho_0-\lambda_0\lambda_2
\]
which implies
\begin{align}
  |\rho_{1/2}|^2+|\rho_{3/2}|^2
  &\leq
  (\lambda_0+\lambda_2)\rho_0-\lambda_0\lambda_2-\rho_0^2
  =
  (\lambda_0-\rho_0)(\rho_0-\lambda_2)
  \leq(\lambda_0-\lambda_2)(\lambda_0-\rho_0).
  \label{offdiag_spectral_bound}
\end{align}
We proceed similarly working with $(\rho^2)_2$ and we obtain
\begin{align}
  |\rho_{3/2}|^2+|\rho_{X}|^2
    \leq(\lambda_0-\lambda_2)(\rho_2-\lambda_2).
  \label{offdiag_spectral_bound2}
\end{align}
Combining \eqref{offdiag_spectral_bound} and \eqref{offdiag_spectral_bound2}
we deduce that 
\[\begin{array}{lll}
\omega_{1/2}|\rho_{1/2}|^2+\omega_{3/2}|\rho_{3/2}|^2
&\leq &
\omega_{1/2}(|\rho_{1/2}|^2+|\rho_{3/2}|^2)+(\omega_{3/2}-\omega_{1/2})|\rho_{3/2}|^2
\\&\leq&
(\lambda_0-\lambda_2)\Big(\omega_{1/2}(\lambda_0-\rho_0)
+
(\omega_{3/2}-\omega_{1/2})(\rho_2-\lambda_2)\Big)
\\&\leq&
(\lambda_0-\lambda_2)\big(
 \mathrm{Tr}(\rho\Omega_{\mathrm{free}})- \mathrm{Tr}(\rho_{\mathrm{min}}\Omega_{\mathrm{free}})\big),
\end{array}\]
where we use \eqref{free_energy_excess}.
Coming back to 
 \eqref{observable_offdiag_bound}
yields
\[
\big| \mathrm{Tr}(\rho V)\big|^2
 \leq4  (\lambda_0-\lambda_2) \left(
  \ds\frac{V^2_{1/2}}{\omega_{1/2}}+\ds\frac{V^2_{3/2}}{\omega_{3/2}}
  \right)
 \big(
 \mathrm{Tr}(\rho\Omega_{\mathrm{free}})- \mathrm{Tr}(\rho_{\mathrm{min}}\Omega_{\mathrm{free}})\big).
\]
Since \(\operatorname{Tr}(\rho_{\mathrm {min}}V)=0\), identities
\eqref{reduced_energy_fixed_spectrum} and
\eqref{free_energy_excess} now imply
\[
  \mathscr E_{\rm red}(\rho)
  -\mathscr E_{\rm red}(\rho_{\mathrm {min}})
  \geq
   \Bigg[
    \hbar
    -
    2\kappa(\lambda_0-\lambda_2)
    \left(
      \frac{V_{1/2}^2}{\omega_{1/2}}
      +
      \frac{V_{3/2}^2}{\omega_{3/2}}
    \right)
  \Bigg] \big(
 \mathrm{Tr}(\rho\Omega_{\mathrm{free}})- \mathrm{Tr}(\rho_{\mathrm{min}}\Omega_{\mathrm{free}})\big).
\]
All terms on the right hand side are non-negative under
\eqref{small_prec}, showing that  \(\rho_{\mathrm {min}}\) is a minimizer.

Coming back to \eqref{free_energy_excess}, we see that equality holds iff $
  \rho_0=\lambda_0$ and 
$  \rho_2=\lambda_2$, thus also $\rho_1=\lambda_1$, by using the trace constraint.
Then, the equality $\mathrm{Tr}(\rho^2)=
\rho_0^2+\rho_1^2+\rho^2_2+2(|\rho_{1/2}|^2+|\rho_{3/2}|^2
+|\rho_{X}|^2)=
\lambda_0^2+\lambda_1^2+\lambda^2_2
$ implies $\rho_{1/2}=\rho_{3/2}=\rho_X=0$:
we get \(\rho=\rho_{\mathrm {min}}\), which proves uniqueness.  Finally, since
\(\operatorname{Tr}(\rho_{\mathrm {min}}V)=0\), 
coming back to the wave field 
 gives \(\psi=0,\varpi=0\).
Observe that 
\eqref{small} implies
\eqref{small_prec} since \(\lambda_0-\lambda_2\leq1\).
\QED

While it looks natural to expect convergence for large times towards the ground state
exhibited in Proposition~\ref{prop:ground_state_fixed_spectrum}, we shall see that this happens only in the pure state case, the mixed case being more intricate and with persistent coherences.

\subsection{Large time behavior and decoherence}
\label{LargeTimeVsys}
We are going to prove the following statement.
This Section is devoted to the identification of the large time behavior, 
the convergence rate is discussed in the forthcoming section and it appeals to different methods.

\begin{theo}\label{th_cv} Assume \eqref{small}.
Then
$\lim_{t\to \infty}\big(q_{1/2},p_{1/2},q_{3/2},p_{3/2}\big)(t)=0$.
n the pure case ($T_2=1$), if, moreover, the initial energy $\mathcal E(0)$ is smaller than $\hbar \omega _1$, then we have \[
\ds\lim_{t\to \infty}\rho(t)=\begin{pmatrix}1 & 0 & 0 \\
0 & 0 & 0\\
0 & 0 & 0\end{pmatrix}=e_0e_0^\top.\]
Finally, in the latter case, there exists $K_*,\lambda_*,c_*>0$ such that for any $c\geq c_*$, 
$\|\rho(t)-e_0e_0^\top\|\leq K_*e^{-\lambda_* t/c}.
$
\end{theo}

We rewrite the system \eqref{Vsystem} by introducing new (real valued) unknowns
\[\begin{array}{ll}
2\rho_{1/2}=q_{1/2}+ip_{1/2},\qquad&
2\rho_{3/2}=q_{3/2}+ip_{3/2},\qquad
2\rho_{X}=q_{X}+ip_{X},
\\
\xi_{1/2}=\rho_1-\rho_0,\qquad &\xi_{3/2}=\rho_2-\rho_0.
\end{array}
\]
By using the contraint $\mathrm{Tr}(\rho)=1$, we can reduce the size of the system, 
dealing only with $8$ real quantities:
 \eqref{Vsystem} becomes
\begin{equation}\label{Vsystembis}
\begin{array}{l}
\dot\xi_{1/2}=\ds\frac\Phi\hbar (2V_{1/2}p_{1/2}+V_{3/2}p_{3/2}),
\\[.3cm]
\dot\xi_{3/2}=\ds\frac\Phi\hbar(V_{1/2}p_{1/2}+2V_{3/2}p_{3/2}),
\\[.3cm]
\dot q_{1/2}=-\omega_{1/2}p_{1/2}-\ds\frac \Phi \hbar V_{3/2}p_X,
\\[.3cm]
\dot p_{1/2}=\omega_{1/2}q_{1/2}-\ds\frac{2 \Phi} \hbar V_{1/2}\xi_{1/2}
-\ds\frac \Phi \hbar V_{3/2} q_{X},
\\[.3cm]
\dot q_{3/2}=-\omega_{3/2}p_{3/2}+\ds\frac \Phi \hbar V_{1/2}p_X,
\\[.3cm]
\dot p_{3/2}=\omega_{3/2}q_{3/2}-\ds\frac{2 \Phi} \hbar V_{3/2}\xi_{3/2}
-\ds\frac \Phi \hbar V_{1/2} q_{X},
\\[.3cm]
\dot q_{X}=-\omega_{X}p_{X}+\ds\frac \Phi {\hbar}  (V_{1/2}p_{3/2}+V_{3/2}p_{1/2}),
\\[.3cm]
\dot p_{X}=\omega_{X}q_{X}-\ds\frac \Phi {\hbar} \Phi (V_{1/2}q_{3/2}-V_{3/2}q_{1/2})
.\end{array}\end{equation}
We are going to use the following shorthand notation: 
we denote by $X$ the vector of the unknowns
in $\mathbb R^8$ that we split into
$\xi=(\xi_{1/2},\xi_{3/2})\in \mathbb R^2$, 
$q=(q_{1/2},q_{3/2})$, $p=(p_{1/2},p_{3/2})$ and $\Xi=(q_X,p_X)$; namely we have $X=(\xi,q,p,\Xi)$.
We will also denote $\mathbb V_\omega=(0,0,v_\omega,0)$.

\begin{defin}\label{AdmX}
We say that the vector $X=(\xi,q,p,\Xi)\in (\mathbb R^2)^4$ is \emph{admissible}
if its components are compatible with a definition from a density matrix; namely, 
setting 
$\rho_0=\frac13(1-\xi_{1/2}-\xi_{3/2})$,
$\rho_1=\frac13(1+2\xi_{1/2}-\xi_{3/2})$ and $\rho_2=\frac13(1+2\xi_{3/2}-\xi_{1/2})$,
$\rho_{1/2}=\frac{q_{1/2}+ip_{1/2}}{2}$, $\rho_{3/2}=\frac{q_{3/2}+ip_{3/2}}{2}$, 
$\rho_{X}=\frac{q_{X}+ip_{X}}{2}$,
then $\rho$ is a density matrix.
\end{defin}

Note that the simplex $\{\rho_0,\rho_1,\rho_2\geq 0,\ \rho_0+\rho_1+\rho_2=1\}$
becomes a triangle in the $(\xi_{1/2},\xi_{3/2})$ plane, with vertices $(-1,-1)$, $(1,0)$, $(0,1)$. 
Since the set of density matrices is compact and convex in $\mathcal M_3(\mathbb C)$, the set of admissible vectors in $\mathbb R^8$ is equally 
compact and convex.

\begin{lemma}\label{lem:estim} 
Given $X_{\mathrm{init}}$ an admissible vector as initial condition,  
the associated solution of \eqref{Vsystembis} satisfies
\[\begin{array}{l}
\textrm{$t\mapsto X(t)$ is  bounded in $L^\infty([0,\infty))$},
\\
\textrm{$(t,z)\mapsto \varpi(t,z)=\partial_t\psi(t,z)$ is  bounded in $L^\infty([0,\infty);L^2(\mathbb R^3))$},
\\
\textrm{$(t,z)\mapsto \psi(t,z)$ is  bounded in $L^\infty([0,\infty);\overbigdot {H}^1(\mathbb R^3))$}.
\end{array}\]
\end{lemma}

\noindent
{\bf Proof.}
The components of $X(t)$ are uniformly bounded since $\rho(t)$ is a density matrix 
and $\mathrm{Tr}(\rho)$, $\mathrm{Tr}(\rho^2)$ are conserved.
To be more specific, all components of $X$ have a modulus smaller than 1.
Next, we use the energy conservation  to estimate
\[\begin{array}{lll}
\ds\frac{1}{2c^2} \|\varpi\|^2_{L^2(\mathbb R^3)}+ \ds\frac{1}2 \|\nabla \psi\|^2_{L^2(\mathbb R^3)}
&\leq& \mathscr E-\mathrm{Tr}(\rho V) \ds\int_{\mathbb R^3} \sigma(z)\psi(z) \ud z
\\[.3cm]
&\leq &\mathscr E_{\mathrm{init}}+ 4\kappa \mathpzc{v}^2+\ds\frac{1}4 \|\nabla \psi\|^2_{L^2(\mathbb R^3)},\end{array}\]
with $|\mathrm{Tr}(\rho(t) V)|=|v\cdot q(t)|\leq \mathpzc{v}|q(t)|\leq \mathpzc{v}<\infty$.
It follows that 
\[\ds\frac{1}{2c^2} \|\varpi\|^2_{L^2(\mathbb R^3)}+\ds\frac{1}4 \|\nabla \psi\|^2_{L^2(\mathbb R^3)}\leq 
\mathscr E_{\mathrm{init}}+4\kappa \mathpzc{v}^2. \]
\QED

\begin{lemma}\label{lem:cv_p}
$\lim_{t\to \infty} 
v_\omega \cdot p(t)
=0$.
\end{lemma}

\noindent
{\bf Proof.}
The proof is strongly inspired  from \cite{BdB, KKS}.
The idea consists in 
 considering a \emph{localized} version of the energy: for $0<R<\infty$, let 
\[
\mathscr E_R(t)
=\hbar \mathrm{Tr}(\rho(t)\Omega_{\mathrm{free}})
+
\ds\frac12\ds\int_{B(0,R)}
\Big(\ds\frac{1}{c^2}|\varpi(t,z)|^2
+|\nabla \psi(t,z)|^2\Big)\ud z
+\mathrm{Tr}(\rho(t) V)\ds\int_{\mathbb R^3}\sigma(z)\psi(t,z)\ud z.
\]
It recasts as
\[
\mathscr E_R=\mathscr E
-\ds\frac12
\ds\int_{\{|z|\geq R\}}
\Big(\ds\frac{|\varpi|^2}{c^2}
+|\nabla \psi|^2\Big)\ud z.
\]
We start with the estimates
\[
\mathscr E_R\leq \mathscr E=\mathscr E_{\mathrm{init}},
\]
and 
\begin{equation}\label{below}
\mathscr E_R\geq -\mathscr E_{\mathrm{init}}-4\kappa\mathpzc{v}^2.
\end{equation}
Indeed, we can write
\[\begin{array}{lll}
\mathscr E_R&\geq& \mathscr E-\ds\frac1{2c^2}\|\varpi\|_{L^2(\mathbb R^3)}^2
-\ds\frac{1}{2}\|\nabla\psi\|_{L^2(\mathbb R^3)}^2\\[.3cm]
&\geq&  \hbar\mathrm{Tr}(\rho \Omega_{\mathrm{free}})+\mathrm{Tr}(\rho  V)\ds\int_{\mathbb R^3} \sigma\psi\ud z
\geq \mathrm{Tr}(\rho  V)\ds\int_{\mathbb R^3} \sigma\psi\ud z.
\end{array}\]
We use again
\[
\begin{array}{l}
\left| \mathrm{Tr}(\rho  V)\ds\int_{\mathbb R^3} \sigma(z)\psi(z)\ud z
\right|
\leq \mathpzc{v}
\sqrt \kappa  \|\nabla \psi\|_{L^2(\mathbb R^3)}
\leq 4\kappa\mathpzc{v}^2+\ds\frac{1}4 \|\nabla \psi\|^2_{L^2(\mathbb R^3)}.
\end{array}\]
The estimates established in Lemma~\ref{lem:estim}  finally yield \eqref{below}.

Let us  pick  $R>R_*$, so that 
\[
\ds\frac{\ud}{\ud t}\mathscr E_R
=
-\ds\frac12\ds\frac{\ud}{\ud t}
\ds\int_{\{|z|\geq R\}}
\Big(\ds\frac{|\varpi|^2}{c^2}
+|\nabla \psi|^2\Big)\ud z
=\ds\int_{\{|z|=R\}}
\varpi\nabla \psi\cdot\ds\frac{z}{|z|}
\ud \varsigma(z)
\]
with $\varsigma$ the Lebesgue measure on the sphere.
Integrating this relation over $[t_1,t_2]$, and using the already obtained estimates, we get
\[
\left|\int_{t_1}^{t_2}
\ds\int_{\{|z|=R\}}
\varpi\nabla \psi\cdot\ds\frac{z}{|z|}
\ud \varsigma(z)\ud t\right|=|\mathscr E_R(t_2)-\mathscr E_R(t_1)|\leq 
2\Big(\mathscr E_{\mathrm{init}}+4\kappa\mathpzc{v}^2\Big):=C_0.\]
We go back to the splitting $\psi=\psi_I+\psi_S$ introduced in Section~\ref{Nota}.
Accordingly, we have
\[
\varpi\nabla \psi
=
\varpi_S\nabla \psi_S
+\varpi_I\nabla \psi_I
+\varpi_I\nabla \psi_S+\varpi_S\nabla \psi_I
.\]
However, with the support assumption \eqref{hyp_CIW}, Huygens principle 
implies that $\psi_I$ is supported in $\{c|t|-R_I\leq |z|\leq c |t|+R_I\}$ which does not  
meet the sphere $\{|z|=R\}$ when $t\geq \frac{R+R_I}{c}$.
Therefore, for any $t\geq \frac{R+R_I}{c}$, we have
\begin{equation}\label{interm1}
\ds\int_{\{|z|=R\}}
\varpi\nabla \psi(t,z)\cdot\ds\frac{z}{|z|}
\ud \varsigma(z)
=
\ds\int_{\{|z|=R\}}
\varpi_S\nabla \psi_S(t,z)\cdot\ds\frac{z}{|z|}
\ud \varsigma(z).\end{equation}

The solution of the wave equation is expressed by using the following representation formula:
\begin{equation}\label{representation}
\psi_S(t,z)=-\ds\frac{1}{4\pi }
\ds\int_{|z-z'|\leq ct}
\ds\frac{\sigma(z')}{|z-z'|} \mathrm{Tr}(\rho V)(t-|z-z'|/c)
\ud z'.\end{equation}
Bearing in mind that we consider $|z|=R$
and $\mathrm{supp}(\sigma)\subset B(0,R_*)$, 
on the integration domain we always have $|z-z'|\leq R+R_*$.
Hence, when $t\geq \frac{R+R_*}{c}$, 
we actually have
\begin{equation}\label{psiS}
\psi_S(t,z)=-\ds\frac{1}{4\pi}
\ds\int_{|z'|\leq R_*}
\ds\frac{\sigma(z')}{|z-z'|} \mathrm{Tr}(\rho V) (t-|z-z'|/c)
\ud z'.\end{equation}
We bear in mind that $\mathrm{Tr}(\rho V)= v\cdot q$ 
so that, by using \eqref{Vsystembis},
\[
\ds\frac{\ud}{\ud t}\mathrm{Tr}(\rho V)
=-v_\omega \cdot p. 
\]
It follows that 
\begin{equation}\label{varpiS}
\varpi_S(t,z)=\partial_t \psi_S(t,z)=\ds\frac{1}{4\pi}
\ds\int_{|z'|\leq R_*}
\ds\frac{\sigma(z')}{|z-z'|} 
v_\omega\cdot p
(t-|z-z'|/c)
\ud z',
\end{equation}
and
\begin{equation}\label{gradpsiS}\begin{array}{lll}
\nabla  \psi_S(t,z)&=&\ds\frac{1}{4\pi }
\ds\int_{|z'|\leq R_*}
\ds\frac{\sigma(z')}{|z-z'|^2} \ds\frac{z-z'}{|z-z'|}v\cdot q(t-|z-z'|/c)\ud z'
\\[.3cm]&&-
\ds\frac{1}{4\pi c}
\ds\int_{|z'|\leq R_*}
\ds\frac{\sigma(z')}{|z-z'|} 
v_\omega\cdot p(t-|z-z'|/c)
 \ds\frac{z-z'}{|z-z'|}
\ud z'
\\[.3cm]
&=&
-\ds\frac1c\varpi_S(t,z)\ds\frac{z}{|z|} - \mathcal I_1(t,z)-\mathcal I_2(t,z)
\end{array}\end{equation}
where
\[\begin{array}{l}\mathcal I_1(t,z)=-
\ds\frac{1}{4\pi }\ds\int_{|z'|\leq R_*}
\ds\frac{\sigma(z')}{|z-z'|^2} \ds\frac{z-z'}{|z-z'|}v\cdot q(t-|z-z'|/c)\ud z',
\\[.3cm]
\mathcal I_2(t,z)=
\ds\frac{1}{4\pi c}
\ds\int_{|z'|\leq R_*}
\ds\frac{\sigma(z')}{|z-z'|} 
v_\omega\cdot p
(t-|z-z'|/c)
\left( \ds\frac{z-z'}{|z-z'|}-\ds\frac{z}{|z|}\right)\ud z'.
\end{array}\]
On the  integration domain, we have $|z-z'|\geq |z|-|z'|=R-|z'|\geq R-R_*$.
In particular, $\mathcal I_1$ can be readily dominated as follows 
\[|\mathcal I_1(t,z)|\leq \ds\frac{\mathpzc{v}\|\sigma\|_{L^1(\mathbb R^3)}}{2\pi (R-R_*)^2}
\]
by using $|q|\leq 1$ again.
For $\mathcal I_2$, we start by observing that
\[
\ds\frac{z-z'}{|z-z'|}-\ds\frac{z}{|z|}=\ds\int_0^1 \ds\frac{\ud}{\ud\tau}
\left(\ds\frac{z-\tau z'}{|z-\tau z'|}\right)\ud \tau 
=
\ds\int_0^1 
\ds\frac{- z' |z-\tau z'| +(z-\tau z')z'\cdot \frac{z-\tau z'}{|z-\tau z'|} }{|z-\tau z'|^2}\ud \tau
\]
is dominated by $\frac{|z'|}{|z-\tau z'|}\leq \frac{R_*}{R-R_*}$.
Hence, we get
\[
|\mathcal I_2(t,z)|\leq \ds\frac{  |v_\omega| R_*\|\sigma\|_{L^1(\mathbb R^3)}} {4\pi c(R-R_*)^2}
.\]
since  $|p|\leq 1$.
Going back to \eqref{gradpsiS}, we can find a constant $C_1>0$ such that 
\[
\nabla \psi_S(t,z)=-\ds\frac1c \varpi_S(t,z)\ds\frac{z}{|z|}+\mathcal I(t,z),\qquad
\ds\sup_{t\geq (R+R_*)/c,\ |z|=R} |\mathcal I(t,z)|\leq \ds\frac{C_1}{R^2}.\]
(To be specific, we can set 
$C_1=\frac{\mathpzc{v}\|\sigma\|_{L^1(\mathbb R^3)}}{2\pi }(1+\frac{\|\Omega\| R_*}{c})$.)
Coming back to \eqref{interm1}, we can write
\[
\ds\int_{\{|z|=R\}}
\varpi\nabla \psi(t,z)\cdot\ds\frac{z}{|z|}
\ud \varsigma(z)
=-\ds\frac1c 
\ds\int_{\{|z|=R\}}
\varpi^2_S
\ud \varsigma(z)
+ \ds\int_{\{|z|=R\}}
\varpi_S
\mathcal I \ud \varsigma(z)
\]
and, by using  Cauchy-Schwarz ans Young inequalities,  it follows that
\[\begin{array}{lll}
\ds\frac{1}2 \ds\int_{\{|z|=R\}}
|\varpi_S|^2(t,z)
\ud \varsigma(z)
&\leq& 
\ds\frac{c^2}2 \ds\int_{\{|z|=R\}}
\mathcal I ^2(t,z)
\ud \varsigma(z)
- c \ds\int_{\{|z|=R\}}
\varpi\nabla\psi(t,z)\cdot\ds\frac{z}{|z|}
\ud \varsigma(z)
\\[.4cm]&\leq& 
\ds\frac{c^2C_1^2}{2R^4}\times 4\pi R^2
- c \ds\int_{\{|z|=R\}}
\varpi\nabla\psi(t,z)\cdot\ds\frac{z}{|z|}
\ud \varsigma(z)
,
\end{array}\]
for any $t\geq \max(\frac{R+R_I}{c},\frac{R+R_*}{c})$.
Gathering the obtained estimates, 
for $t_2\geq t_1\geq \max(\frac{R+R_I}{c},\frac{R+R_*}{c})$,
we arrive at 
\[\begin{array}{l}
\ds\int_{t_1}^{t_2}\ds\int_{\{|z|=R\}}
|\varpi_S|^2\ud \varsigma(z)\ud t
\\[.3cm]=
\ds\frac{1}{16\pi^2 }\ds\int_{t_1}^{t_2}\ds\int_{\{|z|=R\}}
\left|\ds\int_{|z'|\leq R_*}
\ds\frac{\sigma(z')}{|z-z'|}v_\omega\cdot p(t-|z-z'|/c)
\ud z'\right|^2\ud \varsigma(z)\ud t
\\[.3cm]\leq 2cC_0+4\pi c^2 C_1^2\ds\frac{t_2-t_1}{R^{2}}.
\end{array}\]
Next, we change variables by writing $z=R\mathrm U$ with $\mathrm U\in \mathbb S^2$,
$s=t-R/c$, and  
$$
\begin{array}{lll}t-|z-z'|/c&=&t-R|\mathrm U-z'/R|/c=s+R(1-|\mathrm U-z'/R|/c)
\\&=&s+\ds\frac R c\left(1-\sqrt {1-2\mathrm U\cdot z'/R+|z'|^2/R^2}\right)
=s+\ds\frac{\mathrm U\cdot z'}{c}+\ds\frac{\varrho(\mathrm U,z'/R)}{c}\end{array}$$
where $|\varrho(\mathrm U,z'/R)|\leq C/R$ for some $C>0$.
We work with $t_1=\frac{R+R_*}{c}$ (assuming, without loss of generality, $R_*>R_I$), and $t_2=T+\frac Rc$, $T>\frac{R_*}{c}$.
Then, the quantity under consideration becomes
\[
\ds\frac{1}{16\pi^2 }\ds\int_{R_*/c}^{T}\ds\int_{\mathbb S^{2}}\left|
\ds\int_{|z'|\leq R_*}
\ds\frac{\sigma(z')}{|\mathrm U-z'/R|} 
v_\omega\cdot p
(s+\mathrm U\cdot z'/c+\varrho/c)
\ud z'\right|^2 \ud \varsigma(\mathrm U)\ud s.
\]
Letting $R$ go to $\infty$ leads to 
\[
\ds\frac{1}{16\pi^2 }\ds\int_{R_*/c}^{T}\ds\int_{\mathbb S^{2}}
\left|
\ds\int_{|z'|\leq R_*}
\sigma(z')
v_\omega\cdot p
(s+\mathrm U\cdot z'/c)
\ud z'\right|^2 \ud \varsigma(\mathrm U)\ud s\leq 2cC_0 .\]
This bound being independent on $T$, we deduce that 
\[\ds\int_{0}^{\infty}\ds\int_{\mathbb S^{2}}
\underbrace{\left|
\ds\int_{|z'|\leq R_*}
\sigma(z') v_\omega\cdot p(s+\mathrm U\cdot z'/c)
\ud z'\right|^2}_{:=\mathcal J(s,\mathrm U)} \ud \varsigma(\mathrm U)\ud s<\infty.\]
The function $\mathcal J(s,\mathrm U)$ is $C^1$ and, moreover, it is uniformly Lipschitz with respect to $s$ 
since 
$p$ and $\dot  p$ belong to $L^\infty((0,\infty))$. Applying Barbalat's lemma, we deduce that $\lim_{s\to \infty} \mathcal J(s,\mathrm U)=0$
uniformly on $ \mathbb S^2$.

Given $\mathrm U \in \mathbb S^2$, we are going to compute 
in further details $\mathcal J(s,\mathrm U)$. To this end, we can assume that the coordinate axis for the $z'$ variable 
are defined with $\mathrm U$ as the unit vector carrying the first coordinate; namely, we set 
$z'=(z_1=z'\cdot \mathrm U, \hat z=(\mathbb I-\mathrm U\otimes\mathrm U)z')$.
We set 
$$\sigma_c^\#(z_1)=\ds\int_{\mathbb R^{2}} c\sigma(c z_1,\hat z)\ud \hat z,
$$
which is even.
Then, we can write
\[\begin{array}{lll}
\mathcal J(s,\mathrm U)&=&\left|\ds\int_{-\infty}^{+\infty} \sigma_c^\#(\tau)
 v_\omega\cdot p(s+\tau) \ud \tau\right|^2
=\left| \sigma^\#_c\star  
v_\omega\cdot p
(s)\right|^2
\xrightarrow[s\to \infty]{}0.
\end{array}\]
With the assumption that $\widehat \sigma$ never vanishes, see \eqref{hyp_hatsig}, 
we get $\widehat{\sigma^\#_c}(\xi)=\widehat\sigma(\mathrm U\xi/c)\neq 0$ and
we can apply tauberian theorems \cite[Theorem~9.7]{Rud}
which allow us to conclude that 
Lemma~\ref{lem:cv_p} holds.
\QED

\begin{coro}\label{coro_cv_varpi}
$\lim_{t\to\infty}\varpi_S(t,z)=0$ for any $z\in \mathbb R^3$. 
\end{coro}

\noindent
{\bf Proof.}
We go back to \eqref{varpiS} and we apply Lebesgue's theorem, 
using that $z'\mapsto \frac{\sigma(z')}{|z-z'|}$ is integrable:
\[\begin{array}{lll}
\ds\int_{\mathbb R^3} \frac{|\sigma(z')|}{|z-z'|}\ud z'
&\leq& \ds\int_{|z-z'|\leq A} \frac{|\sigma(z')|}{|z-z'|}\ud z'+\ds\int_{|z-z'|\geq A} \frac{|\sigma(z')|}{|z-z'|}\ud z'
\\&\leq&
\|\sigma\|_{L^\infty(\mathbb R^3)}
|\mathbb S^2|\ds\int_0^A \ds\frac{r^2\ud r}{r}
+
 \ds\frac{\|\sigma\|_{L^1(\mathbb R^3)}}{A}
 \leq  \|\sigma\|_{L^\infty(\mathbb R^3)}
 |\mathbb S^2|
 \ds\frac{A^2}2
 + \ds\frac{\|\sigma\|_{L^1(\mathbb R^3)}}{A}.
 \end{array}\]
\QED

Let $\big(t^\nu\big)_{\nu\in\mathbb N}$ be an increasing sequence that tends to $+\infty$.
We are going to establish various estimates: in the discussion we denote by $C>0$ a constant
which depends on the parameters of the model, but which remains independent on $\nu$; the precise value of the constant may vary from a line to another.
We set $X^\nu(t)=X(t+t^\nu)$.
We  know that 
 \[\begin{array}{l}
\textrm{$t\mapsto X^\nu(t)$ is  bounded in $L^\infty([0,\infty))$},
\\
\textrm{$(t,z)\mapsto \varpi^\nu(t,z)$ is  bounded in $L^\infty([0,\infty);L^2(\mathbb R^3))$},
\\
\textrm{$(t,z)\mapsto \psi^\nu(t,z)$ is  bounded in $L^\infty([0,\infty);\overbigdot {H}^1(\mathbb R^3))$},
\end{array}\]
uniformly with respect to $\nu$. 
Moreover, going back to \eqref{Vsystembis}. we see that  
$\frac{\ud}{\ud t}X^\nu$ is equally uniformly bounded in $L^\infty([0,\infty))$.
Given $0<T<\infty$, by virtue of the Arzela-Ascoli theorem, 
we can thus assume, possibly at the price of extracting a subsequence, 
 that $X^\nu$ converge uniformly on $[0,T]$ to $X^\infty$.
 Moreover we can also suppose that $\nabla\psi^\nu$ and $\varpi^\nu$ converge weakly in $L^2((0,T)\times\mathbb R^3)$ with limits denoted $\nabla\psi^\infty$ and $\varpi^\infty$ respectively.
 
 In fact, we can show that $\nabla\psi^\nu$ and $\varpi^\nu$ converge  $C^0([0,T];L^2(\mathbb R^3)-weak)$.
Indeed, let $\zeta\in C^\infty_c(\mathbb R^3))$.
We have 
\[\left|\ds\int_{\mathbb R^3}\varpi^\nu(t,z)
\zeta(z) \ud z \right|
\leq C\|\zeta\|_{L^2(\mathbb R^3)}
\]
and 
\[\begin{array}{lll}
\left|
\ds\frac{\ud}{\ud t}\ds\int_{\mathbb R^3}\varpi^\nu(t,z)
\zeta(z) \ud z \right|&=&\left|\ds\int_{\mathbb R^3}\partial_t \varpi^\nu(t,z)
\zeta(z) \ud z \right|
\\&=&\left|c^2\ds\int_{\mathbb R^3}\nabla\psi^\nu(t,z)\cdot 
\nabla\zeta(z) \ud z -c^2\ds\int_{\mathbb R^3}q^\nu(t)\sigma(z)
\zeta(z) \ud z\right|
\\&\leq& C \|\zeta\|_{H^1(\mathbb R^3)}.\end{array}\]
The Arzela-Ascoli theorem implies that $\big\{\int_{\mathbb R^3}\varpi^\nu(t,z)
\zeta(z) \ud z ,\ \nu\in \mathbb N\big\}$ is compact in $C^0([0,T])$.
By density of $C^\infty_c(\mathbb R^3))$ in $L^2(\mathbb R^3)$ the conclusion equally applies to any trial function $\zeta 
\in L^2(\mathbb R^3)$.
By using the separability of $L^2(\mathbb R^3)$ and a diagonal argument, we conclude that we can still extract a subsequence such that 
\[\ds\lim_{\nu\to \infty}\ds\int_{\mathbb R^3}\varpi^\nu(t,z)
\zeta(z) \ud z=\ds\int_{\mathbb R^3}\varpi^\infty(t,z)
\zeta(z) \ud z\]
holds uniformly 
  on $[0,T]$, for any $\zeta\in L^2(\mathbb R^3)$.
  We justify similarly that 
  \[\ds\lim_{\nu\to \infty}\ds\int_{\mathbb R^3}\nabla\psi^\nu(t,z)
\zeta(z) \ud z=\ds\int_{\mathbb R^3}\nabla\psi^\infty(t,z)
\zeta(z) \ud z\]
holds uniformly 
  on $[0,T]$, for any $\zeta\in L^2(\mathbb R^3)$.
With Lemma~\ref{lem:cv_p} and  Corollary~\ref{coro_cv_varpi}, we obtain 
$$v_\omega\cdot p^\infty=0\textrm{ and } \varpi^\infty=0.$$
(The proof of Corollary~\ref{coro_cv_varpi} provides a $L^\infty$ estimate, uniform wrt $\nu$ on $\varpi^\nu$.)

Accordingly, we deduce that 
$$
\ds\lim_{\nu\to\infty}\Phi^\nu(t)=\Phi^\infty(t)=\ds\lim_{\nu\to \infty} \ds\int_{\mathbb R^3} \sigma\psi^\nu\ud z
=
\ds\int_{\mathbb R^3} \sigma\psi^\infty(t,z)\ud z
$$
satisfies
\[
\ds\frac{\ud}{\ud t}\Phi^\infty(t)=\ds\int_{\mathbb R^3} \sigma\varpi^\infty(t,z)\ud z
=0.\]
Therefore $\Phi^\infty$ is constant.
Passing to the limit in the wave equation, we have
\[-\Delta\psi^\infty=-\sigma v\cdot q^\infty\]
that is $\psi^\infty(t,z)=-\Gamma(z) v\cdot q^\infty(t)$.

Passing to the limit in \eqref{Vsystembis}, we obtain a \emph{linear} system, with coefficients depending
on $\Phi^\infty$: with 
$X^\infty=(
\xi^\infty_{1/2},\xi^\infty_{3/2},q^\infty_{1/2}, q^\infty_{3/2},
p^\infty_{1/2},p^\infty_{3/2},q^\infty_X,p^\infty_X
),
$ we have
\begin{equation}\label{syslininf}
\ds\frac{\ud}{\ud t}X^\infty=A^\infty X^\infty,
\end{equation}
where\[
A^\infty=\begin{pmatrix}
0&0 &0 &  0& 2\frac{\Phi^\infty}\hbar V_{1/2} & \frac{\Phi^\infty}\hbar V_{3/2} & 0 & 0 
\\
0&0 &0&0 & \frac{\Phi^\infty}\hbar V_{1/2} & 2\frac{\Phi^\infty}\hbar V_{3/2} & 0 & 0 
\\
0&0 &0&0 &  -\omega_{1/2} & 0 & 0 & - \frac{\Phi^\infty}\hbar V_{3/2}
\\
0&0 &0&0   & 0&-\omega_{3/2} & 0 &  \frac{\Phi^\infty}\hbar V_{1/2}
\\
- \frac{2\Phi^\infty}\hbar V_{1/2}&0 &\omega_{1/2}&0   & 0 &0& - \frac{\Phi^\infty}\hbar V_{3/2} & 0
\\
0&- \frac{2\Phi^\infty}\hbar V_{3/2} &0&\omega_{3/2}&  0 & 0 &  - \frac{\Phi^\infty}\hbar V_{1/2}& 0
\\
0&0 &0&0   & \frac{\Phi^\infty}{\hbar} V_{3/2} & \frac{\Phi^\infty}{\hbar} V_{1/2}& 0&-\omega_{X} 
\\
0&0 &\frac{\Phi^\infty}{\hbar} V_{3/2}&-\frac{\Phi^\infty}{\hbar} V_{1/2}   & 0&0 &\omega_{X}&0
\end{pmatrix},\qquad
\]
The crucial observation is that, owing to  Lemma~\ref{lem:cv_p}, the solution is constrained to satisfy 
\[v_\omega\cdot p^\infty 
=0,\]
which can be rephrased 
 as $$\mathbb V_\omega\cdot X^\infty=0.$$
In turn, we observe that 
\[
\ds\frac{\ud}{\ud t}v\cdot q^\infty=v_\omega\cdot p^\infty=0\]
and $t\mapsto v\cdot q^\infty$ is constant which thus satisfies
\[
\Phi^\infty=-
\ds\int_{\mathbb R^3}\sigma(z) \Gamma(z)\ud z\, v\cdot q^\infty
=-
\kappa v\cdot q^\infty.\]

 \begin{proposition}\label{prop_cv}
 Assume \eqref{small}.
 Then the limit satisfies
 $q^\infty=0$, $p^\infty=0$.
 
 Moreover, in the pure case, assuming that the initial 
 initial energy is smaller than $\hbar\omega_1$,  
  the limit is completely characterized:
 $\rho^\infty_0=1$, $\rho^\infty_1=\rho^\infty_2=0$, $q^\infty=0$, $p^\infty=0$, $\Xi^\infty=0$.
 \end{proposition}
 
 This will be obtained as a consequence of the following claim.
 
 \begin{lemma}\label{Phi0}
 Assume \eqref{small}.
  Then $\Phi^\infty=0$.
 \end{lemma}
 
 Let us temporarily assume that $\Phi^\infty=0$.
 Then, $\xi_{1/2}^\infty$ and $\xi_{3/2}^\infty$ are constant,
 $q,p,\Xi$ are oscillating according to  
 \[\begin{array}{ll}
 q_{1/2}=a_{1/2}\cos(\omega_{1/2}t)+b_{1/2}\sin(\omega_{1/2}t),\qquad&
 p_{1/2}=-a_{1/2}\sin(\omega_{1/2}t)+b_{1/2}\cos(\omega_{1/2}t),
 \\
 q_{3/2}=a_{3/2}\cos(\omega_{3/2}t)+b_{3/2}\sin(\omega_{3/2}t),
 \qquad & p_{3/2}=-a_{3/2}\sin(\omega_{3/2}t)+b_{3/2}\cos(\omega_{3/2}t),
 \\q_{X}=a_{X}\cos(\omega_{X}t)+b_{X}\sin(\omega_{X}t),\qquad 
&
 p_{X}=-a_{X}\sin(\omega_{X}t)+b_{X}\cos(\omega_{X}t).
 \end{array}\]
  However, that $t\mapsto v_\omega\cdot p^\infty(t)$ identically vanishes imposes 
  $a_{1/2}=a_{3/2}=b_{1/2}=b_{3/2}=0$, that means $q_{1/2}=q_{3/2}=0$, $p_{1/2}=p_{3/2}=0$.
  Therefore, the asymptotic solution has the form
  \[\begin{pmatrix}
  \rho_0^\infty & 0 & 0 \\ 0 & \rho_1^\infty & r_X^\infty e^{i\omega_X t}
  \\
  0 & r_X^\infty e^{-i\omega_X t}& \rho_2^\infty
  \end{pmatrix}.
  \]
  The limit coefficients have to fulfil the requirements:
  $$ 
 \rho_0^\infty+\rho_1^\infty+\rho_2^\infty=1
 ,
 \qquad
 \rho_0^{\infty2}+\rho_1^{\infty2}+\rho_2^{\infty2}+2r_X^{\infty2}=T_2,
 \\
 \qquad
 0\leq r_{X}^{\infty 2} \leq \rho_1^\infty\rho_2^\infty.
 $$
 In the pure case, this allows us to completely 
 characterize the 
 limit:
 since
 $$1=(\rho_0^\infty+\rho_1^\infty+\rho_2^\infty)^2=
 T_2$$
 we deduce that 
 \[
 2\rho_0^\infty\rho_1^\infty+2\rho_0^\infty\rho_2^\infty+2(\rho_1^\infty\rho_2^\infty
 -r_X^{\infty 2})=0.\]
 This is a sum of nonnegative terms so that they all vanish.
 Then, either $\rho^\infty_1=\rho^\infty_2=0$, $r_X^\infty=0$ and $\rho^\infty_0=1$, or 
 $\rho^\infty_0=0$.
 We disregard the latter case by energetic considerations.
 Indeed, since $\varpi^\nu$ and $\nabla\psi^\nu$ converge  in $C^0([0,T];L^2(\mathbb R^3)-weak)$, we have
 \[
 \mathscr E_{\mathrm{init}}\geq  \mathscr E^\infty=
 \hbar\ds\sum_{j=0}^2\omega_j\rho^\infty_j.\]
 Since $\omega_1\rho+\omega_2(1-\rho)\geq \omega_1$ for any $\rho\in [0,1]$, 
 we conclude that $\rho_0^\infty$ cannot vanish if $\mathscr E_{\mathrm{init}}<\hbar \omega _1 $.

 We are thus left with the task of showing that $\Phi^\infty$ vanishes.
 To this end, we will be led to study $\mathrm{Ker}(A^\infty)$ 
and the following claim will be critical to conclude.

 \begin{lemma}\label{admker}
 Let $\Phi^\infty\neq 0$. Then the kernel of $A^\infty$ is two-dimensional; it is spanned by
 \[\begin{array}{l}
 X_1=\left(
 -\ds\frac{\Phi^\infty}{2\hbar}\ds\frac{V_{3/2}}{\omega_X},
 \ds\frac{\hbar}{2\Phi^\infty}\ds\frac{1}{V_{3/2}}\Big(\omega_{3/2}-\ds\frac{\Phi^{\infty 2}}{\hbar^2}\ds\frac{V_{1/2}^2}{\omega_X}\Big)
 ,0,1,0,0,
 \ds\frac{\Phi^\infty}{\hbar}\ds\frac{V_{1/2}}{\omega_X},0
 \right)
 ,\\
 X_2=\left(
 \ds\frac{\hbar}{2\Phi^\infty}\ds\frac{1}{V_{1/2}}\Big(\omega_{1/2}+\ds\frac{\Phi^{\infty 2}}{\hbar^2}\ds\frac{V_{3/2}^2}{\omega_X}\Big),
 \ds\frac{\Phi^\infty}{2\hbar}\ds\frac{V_{1/2}}{\omega_X}
 ,1,0,0,0
- \ds\frac{\Phi^\infty}{\hbar}\ds\frac{V_{3/2}}{\omega_X},0
 \right)
.\end{array}\]
In particular any $X\in \mathrm{Ker}(A^\infty)$ satisfies $v_\omega\cdot p=0$.
 If \eqref{small} holds, 
 then $\mathrm{Ker}(A^\infty)$ does not contain an admissible vector such that $v\cdot q
 =-\frac{\Phi^\infty}{\kappa}$. \end{lemma}
 
 \noindent
 {\bf Proof.}
 The eigenspace $\mathrm{Ker}(A^\infty)$ is found by direct inspection.
 The eigenspace $\mathrm{Ker}(A^\infty)$ is two-dimensional and its elements naturally fulfil the constraint $v_\omega\cdot p=0$.
 However, we are going to show that this eigenspace cannot contain admissible  vectors verifying  $v\cdot q
 =-\frac{\Phi^\infty}{\kappa}\neq 0$.
 Here, we detail the proof assuming the suboptimal condition 
 \begin{equation}\label{subopt}\ds\frac{2\kappa \mathpzc{v}^2}
{\hbar \omega_{1/2}}<1
\end{equation}
instead of \eqref{small}, in order to keep a quite elementary discussion.
 That the result holds with the sharp condition \eqref{small} is fully detailed in Proposition~\ref{prop:optim}, in a more general framework.
The components of the eigenvectors satisfy
\[\begin{array}{l}
\omega_{1/2}q_{1/2}=-\ds\frac{2\Phi^\infty}{\hbar}V_{1/2}\xi_{1/2}-\ds\frac{\Phi^\infty}{\hbar }V_{3/2}q_X,
\\[.4cm]
\omega_{3/2}q_{3/2}=-\ds\frac{2\Phi^\infty}{\hbar}V_{3/2}\xi_{3/2}-\ds\frac{\Phi^\infty}{\hbar }V_{1/2}q_X,
\end{array}
\]
Therefore, we get
\[
\mathrm{Tr}(\rho V)=
\mathpzc{v}\cdot q=
-\ds\frac{\Phi^\infty}{\kappa}
=
-\ds\frac{\Phi^\infty}{\hbar}
\left(\ds\frac{2V_{1/2}^2\xi_{1/2}}{\omega_{1/2}}
+\ds\frac{2V_{3/2}^2\xi_{3/2}}{\omega_{3/2}}
+V_{1/2}V_{3/2}
\left(\ds\frac{1}{\omega_{1/2}}+\ds\frac{1}{\omega_{3/2}}\right)
q_X\right).
\]
If $\Phi^\infty\neq 0$, we interpret this relation as
\[
\ds\frac{\hbar}{\kappa}=-
\mathrm{Tr}(\rho O)
\]
where 
\begin{equation}\label{defO}
O
=\begin{pmatrix}
-2V_{1/2}^2/\omega_{1/2}-2V_{3/2}^2/\omega_{3/2} & 0 & 0
\\
0 & 2V_{1/2}^2/\omega_{1/2}& V_{1/2}V_{3/2}(1/\omega_{1/2}+1/\omega_{3/2})
\\
0& V_{1/2}V_{3/2}(1/\omega_{1/2}+1/\omega_{3/2})& 2V_{3/2}^2/\omega_{3/2}
\end{pmatrix}.\end{equation}
Since $O$ is symmetric we have 
$-\|O\|\mathbb I\leq O\leq \|O\|\mathbb I$ and since 
$\rho$ is a density matrix, we deduce that
\[
|\mathrm{Tr}(\rho O)|\leq \|O\|
\leq \ds\frac{2\mathpzc{v}^2}{\omega_{1/2}}.\]
which contradicts \eqref{subopt}.
\QED

 \noindent
 {\bf Proof of Lemma~\ref{Phi0}.}
 We argue by contradiction, assuming that $\Phi^\infty\neq 0$.  The solution $t\mapsto X^\infty(t)$ of the system \eqref{syslininf} inherits the admissibility properties and bounds from the original vector $t\mapsto X(t)$.
 We have seen that 
 $t\mapsto v\cdot q^\infty(t)$ does not depend on time and by construction it coincides with $-\frac{\Phi^\infty}{\kappa}$.
 
  Let $T>0$ and consider the mean value
 \[\langle X^\infty\rangle _T=\ds\frac{1}{T}\ds\int_0^T X^\infty (s)\ud s.\] 
 We have
 \begin{equation}\label{lim0}
 A^\infty \langle X^\infty\rangle _T=\langle \dot X^\infty\rangle _T=\ds\frac{1}{T}\ds\int_0^T A^\infty X^\infty (s)\ud s
 =\ds\frac{1}{T}\ds\int_0^T \dot X^\infty (s)\ud s=\ds\frac{X^\infty(T)-X^\infty(0)}{T}
 \xrightarrow [T\to \infty]{} 0\end{equation}
 since $\{X^\infty(T)
, T>0\}$ is a bounded set.
We can consider an increasing sequence of positive numbers $\big(T_n\big)_{n\in \mathbb N}$ such that  $$\ds\lim_{n\to \infty} T_n=+\infty\quad \textrm{ and  } \quad  
\ds\lim_{n\to \infty} \langle X^\infty\rangle _{T_n}=\mathscr X.$$
As a matter of fact, the set of admissible vectors is convex and the average procedure is convex, so that 
$\mathscr X=(\xi^{\mathscr X}, q^{\mathscr X}, p^{\mathscr X},\Xi^{\mathscr X})$ is still admissible.
 Moreover, \eqref{lim0} implies $A^\infty \mathscr X=0$ and
 we have 
 \[
 v\cdot q^{\mathscr X}=\ds\lim_{n\to \infty} \ds\frac{1}{T_n}\ds\int_0^{T_n} v\cdot q^\infty(s)
\ud s= -\ds\frac{\Phi^\infty}{\kappa}\]
which is supposed to be $\neq 0$, contradicting Lemma~\ref{admker}.
\QED

In general,  fading of $q_X,p_X$ is  false, as shown by the following counter example:
pick $a,b\geq0$, with $a+b\le1$, and $r>0$ with $|r|^2\leq ab$
and set
\[
\rho(t)=
\begin{pmatrix}
1-a-b&0&0\\
0&a&r e^{i\omega_Xt}\\
0& r e^{-i\omega_Xt}&b
\end{pmatrix},
\qquad \psi(t,z)=0.
\]
We claim that this is an exact solution of \eqref{Vsystem}
where $\rho_X$ has a constant nonzero amplitude.
Nevertheless, the following property
completes Theorem~\ref{th_cv}.

\begin{lemma}\label{lem:cesaro}
In the conditions of Theorem~\ref{th_cv}, 
the coherence $\rho_X$ has a vanishing Cesaro mean:
\[
\ds\lim_{T\to \infty}\left(\frac1T\int_0^T\rho_X(t)\ud t\right)=0.
\]
\end{lemma}

\noindent
{\bf Proof.} 
This is a consequence of 
\[
\dot\rho_X
=i\omega_X\rho_X-\frac{i}{\hbar}\Phi
\left(
V_{1/2}\rho_{3/2}
-V_{3/2}\rho_{1/2}^*\right),
\]
where we can combine  
\[
\rho_{1/2}(t),\rho_{3/2}(t)\xrightarrow[t\to \infty]{}0
\]
and the boundedness of $\Phi$.
Integrating  between $0$ and $T$ gives
\[
\frac1T\int_0^T\rho_X(t)\ud t
=\frac{\rho_X(T)-\rho_X(0)}{i\omega_XT}
-\frac1{i\omega_XT}\int_0^T F(t)\ud t
,\]
where $F(t)=\frac{i}{\hbar}\Phi
(
V_{1/2}\rho_{3/2}
-V_{3/2}\rho_{1/2}^*)\to0$ as $t$ runs to $\infty$, which allows us to conclude.\QED

\subsection{Exponential rate of convergence to equilibrium (pure state)}

In this Section we establish, once knowing that the unknown $X$ admits a limit as $t\to \infty$, 
that the convergence holds with exponential rate, possibly at the price of 
introducing further assumptions 
on the coupling parameters. 
The proof thus covers the case of pure
 states; for mixed states
the unique identification of the limit is missing, see further comments in Section~\ref{sec:comm}.

The damping phenomena appears more neatly by working with the variables 
\[\begin{pmatrix}
y
\\
z
\end{pmatrix}=\mathscr V
\begin{pmatrix}
q_{1/2}
\\
q_{3/2}
\end{pmatrix},\qquad
\begin{pmatrix}
\eta
\\
\zeta
\end{pmatrix}=\mathscr V
\begin{pmatrix}
p_{1/2}
\\
p_{3/2}
\end{pmatrix},\qquad \mathscr V=\begin{pmatrix}
V_{1/2} & V_{3/2}
\\
V_{3/2}& -V_{1/2}
\end{pmatrix}
,\qquad 
\mathscr V^{-1}=\ds\frac{\mathscr V}{\mathpzc{v}^2}.
\]
As already observed above, the variable $y$ is the one that drives the coupling with the environment, since we have 
$(\partial^2_t-c^2\Delta)\psi(t,z)  =-c^2\sigma(z) y(t)$, and thus the self consistent potential
$\Phi_S$ depends only on the component $y$
\[
\Phi_S(t)=-\ds\int_0^t k_c(t-s)y(s)\ud s.\]
For future use, we denote
\[
a=\ds\frac{\omega_{1/2}V_{1/2}^2 +\omega_{3/2}V_{3/2}^2}{\mathpzc{v}^2},\qquad
b=\ds\frac{(\omega_{3/2}-\omega_{1/2})V_{1/2}V_{3/2}}{\mathpzc{v}^2},
\qquad
d=\ds\frac{\omega_{1/2}V_{3/2}^2 +\omega_{3/2}V_{1/2}^2}{\mathpzc{v}^2}
\]
which are all positive coefficients,
and
\[
\ell=\Big(0,1,0, -\ds\frac ba\Big),\qquad
\widetilde \ell_c=-\kappa e_0-\ds\frac{\upgamma a}{c}\ell.\]
We rewrite the self-consistent part by using integration by parts
\begin{equation}\label{PhiS}
\begin{array}{lll}
\Phi_S(t)&=&
-\left(\ds\int_0^t k_c(t-s)\ud s\right) y(t)
+
\left(\ds\int_0^t k_c(t-s)(t-s)\ud s\right) \dot y(t)
\\[.3cm]&& -\ds\int_0^t k_c(t-s)\left(\ds\int_s^t (\tau-s) \ddot y(\tau)\ud \tau\right)\ud s
\\
[.3cm]
&=&
-\kappa y(t) + \ds\frac\upgamma{c} \dot y(t)+\mathcal R(t) 
\end{array}\end{equation}
with the remainder
\begin{equation}\label{def_R}\begin{array}{lll}
\mathcal R(t)&=&-\left(\ds\int_0^t k_c(s)\ud s-\kappa\right) y(t)
+ \left(\ds\int_0^t s k_c(s)\ud s- \ds\frac\upgamma{c} \right) \dot y(t)
\\&&-\ds\int_0^t k_c(t-s)\left(\ds\int_s^t (\tau-s) \ddot y(\tau)\ud \tau\right)\ud s
\\
&=&-\left(\ds\int_t^\infty k_c(s)\ud s \right)y(t)
+ \left(\ds\int_t^\infty s k_c(s)\ud s\right) \dot y(t)
-\ds\int_0^t k_c(t-s)\left(\ds\int_s^t (\tau-s) \ddot y(\tau)\ud \tau\right)\ud s
.\end{array}\end{equation}
By virtue of Lemma~\ref{lem_AV}, the first two terms vanish for $t\geq \frac{2R_*}{c}$.
Now, from \eqref{Vsystembis} we obtain
\begin{equation}\label{mod_system3}
\begin{array}{lll}
\ds\frac{\ud}{\ud t}
\begin{pmatrix}
y\\\eta\\z\\\zeta
\end{pmatrix}
&=&
\begin{pmatrix}
0 & -a & 0 & b
\\
a& 0 & -b & 0
\\
0 & b & 0 &-d
\\
-b & 0 & d& 0
\end{pmatrix}
\begin{pmatrix}
y\\\eta\\z\\\zeta
\end{pmatrix}-
\ds\frac{2\Phi}\hbar 
\begin{pmatrix}
0\\ V^2_{1/2}\xi_{1/2}+V^2_{3/2}\xi_{3/2}
\\0
\\
V_{1/2}V_{3/2}(\xi_{1/2}-\xi_{3/2})
\end{pmatrix}
-\ds\frac{\Phi}{\hbar}
\begin{pmatrix}
0\\  
2V_{1/2}V_{3/2}q_X
\\
\mathpzc{v}^2p_X
\\
(V_{3/2}^2-V^2_{1/2})q_X
\end{pmatrix}.
\end{array}
\end{equation}
Now, coming back to \eqref{PhiS}, the potential recasts as
\[
\Phi(t)=-\kappa y(t)+\ds\frac\upgamma{c}(-a\eta(t)+b\zeta(t)) +\mathcal R(t)+\Phi_I(t)
=
\widetilde\ell_c\cdot Y(t) +\mathcal R(t)+\Phi_I(t).\]
The last term in \eqref{mod_system3} has the compact form
\[
-\ds\frac{\widetilde \ell_c\cdot Y(t)+\mathcal R(t)+\Phi_I(t)}\hbar B\Xi(t)
,
\]
with the $2\times 4$ matrix
\[
B=\begin{pmatrix} 0 & 0
\\
2V_{1/2}V_{3/2}& 0
\\
0 & \mathpzc{v}^2\\
V_{3/2}^2-V_{1/2}^2& 0\end{pmatrix}.
\]
The components $(q_X,p_X)$ can be expressed by means  of the Duhamel formula
\[
\Xi(t)=e^{\Omega_X t}\Xi(0)
+\ds\frac1\hbar \ds\int_0^t e^{\Omega_X(t-s)} \Phi(s) 
\begin{pmatrix}
2\frac{V_{1/2}V_{3/2}}{\mathpzc{v}^2}\eta(s)+ \frac{V^2_{3/2}-V^2_{1/2}}{\mathpzc{v}^2}\zeta(s)
 \\
 -z(s)
\end{pmatrix}\ud s,\]
where 
\begin{equation}\label{def_OmegaX}
\Omega_X=\begin{pmatrix}0 & -\omega_X\\\omega_X& 0\end{pmatrix},\qquad
e^{\Omega_X t}=\begin{pmatrix}\cos(\omega_X t) & - \sin(\omega_X t)
\\
\sin(\omega_X t) & \cos(\omega_X t)
\end{pmatrix}.\end{equation}
They are completely determined by means of $Y$.

Finally, we write \eqref{mod_system3} as follows
\begin{equation}\label{mod_system4}\begin{array}{lll}
\ds\frac{\ud}{\ud t}
\begin{pmatrix}
y\\\eta\\z\\\zeta
\end{pmatrix}
&=&
\underbrace{\begin{pmatrix}
0 & -a & 0 & b
\\
a-\tilde\kappa& -g  & -b & g\frac{b}{a} 
\\
0 & b & 0 &-d
\\
-b & 0 & d& 0
\end{pmatrix}}_{:=A} 
\underbrace{\begin{pmatrix}
y\\\eta\\z\\\zeta
\end{pmatrix}}_{:=Y}
+
\underbrace{\begin{pmatrix}
0\\ \widetilde{ \mathcal R}
\\0
\\
\widetilde{\mathcal S}
\end{pmatrix} 
-\ds\frac{(\tilde\ell_c\cdot Y+\mathcal R+\Phi_I)}\hbar B\Xi
}_{:=S}
\end{array}
\end{equation}
where we have set
\begin{equation}\label{coef}
\tilde\kappa=\frac{2\kappa \mathpzc{v}^2}{\hbar },\qquad 
g=\frac{2\upgamma a\mathpzc{v}^2}{\hbar c},
\end{equation}
and
\[\begin{array}{l}
\widetilde{ \mathcal R}= 
-\ds\frac{V^2_{1/2}\xi_{1/2}+V^2_{3/2}\xi_{3/2}}{\hbar}\left(\mathcal R+\Phi_I 
\right)
-\ds\frac{V^2_{1/2}(1+\xi_{1/2})+V^2_{3/2}(1+\xi_{3/2})}{\hbar}
\widetilde \ell_c\cdot Y
,
\\[.4cm]
\widetilde{\mathcal S}=-\ds\frac{V_{1/2}V_{3/2}}\hbar (\xi_{1/2}-\xi_{3/2})(\tilde\ell_c\cdot Y+\mathcal R+\Phi_I).
\end{array}\]
We remind the reader that $\widetilde \ell_c\cdot Y=-\kappa y
-\frac{a\upgamma}{ c}\eta
+\frac{b\upgamma}{ c}\zeta
$ and $\Phi_I,\mathcal R$ are defined in \eqref{defPhiIS} and \eqref{def_R} respectively.
Even if its evolution remains coupled to  other variables,  this formulation 
makes a damping appear on the quantity  $\eta$.
 We will see that not only $\eta $ is effectively damped, but the damping propagates to the other variables, which is far from clear at first sight. 

\subsubsection{Analysis of the linear problem \texorpdfstring{$\dot Y=AY$}{Y'=AY}}

The goal of this Section is to investigate the spectral 
properties of the matrix $A$ and more precisely to establish the following decay estimate.

\begin{proposition}\label{estim_expA}
There exists $K_*,\lambda_*,c_*>0$ such that for any $c\geq c_*$, 
we have
$\|e^{At}\|\leq K_*e^{-\lambda_*t/c}$.
\end{proposition}

Proposition~\ref{estim_expA} is certainly a key statement in the analysis of the large time behavior 
of the solutions of the V-system \eqref{Vsystem}. 
 Owing to the small dimension of the system, 
 that make explicit computations affordable,
 it is possible to develop a direct approach. 
 However, in view of extension to more complex systems, it could be important to derive the result 
 with a different viewpoint.
 We think important to present the two complementary strategies since
 the simple framework permits us to identify key mechanisms. We will see
 that the matrix $A$ can be 
  split into an oscillatory contribution $\Theta$ and a damping term $gD$ which acts on a reduced set of variables (see \eqref{defSplitA} below). 
That the damping becomes sensitive for all the components of the unknown is reminiscient to  structural conditions arising in the study of hyperbolic systems
with partial diffusion, a criterion referred to as the Kawashima-Shizuta condition \cite{nanard2,nanard,KS}.
Such phenomena can also be related to the Kalman condition that arises in control theory
and have connection to hypocoercivity \cite{Vill}, as investigated in details  \cite{AAM,AAM2,BZ}.
\\

\noindent{\bf Direct spectral analysis.}

\noindent
The characteristic polynomial of $A$ reads
\[
\chi_g(\lambda)=\lambda^4+g\lambda^3+
P_2\lambda^2+g P_1
\lambda 
+P_0\]
with 
\[\begin{array}{l}
P_2=a(a-\tilde \kappa)+2b^2+ d^2,
\\
P_1=(ad-b^2)(d/a),
\\
P_0=  a(a-\tilde \kappa)d^2 -adb^2 -(a-\tilde \kappa)db^2+b^4.
\end{array}\]
The zeroth order term can be recast as
\[
P_0=
(ad-b^2)((a-\tilde \kappa)d -b^2).
\]
Coming back to the definition of the coefficients, we observe that 
\[
ad-b^2
=
\omega_{1/2}\omega_{3/2}
>0.\]
Assumption
\eqref{small} guarantees $a-\tilde \kappa>0$,
as well as the positivity of $P_0$.
Hence, when \eqref{small} holds, all coefficients of 
$P$ are strictly positive and the Routh-Hurwitz criterion \cite{Hurwitz1895,Routh1877} is fulfilled since it reads
\[\begin{array}{l}
g^2 P_1P_2-g^2P_1^2-g^2P_0
\\
\qquad
=\ds\frac{g^2(ad-b^2)}{a^2}\Big(
ad\big(a(a-\tilde \kappa)+2b^2+d^2\big)-d^2(ad-b^2)-a^2\big(
(a-\tilde \kappa)d-b^2
\big)
\Big)
\\
\qquad
=\ds\frac{g^2(ad-b^2)}{a^2} b^2(a+d)^2
>0.
\end{array}
\]
Hence all eigenvalues of $A$ have negative real part.
This allows us to identify  the asymptotic behavior of the linearized system $\dot X=AX$.
 So far, however, this does not provide sufficient information about the dependence of the eigenvalues on $g$, which is essential for the analysis of the nonlinear problem.

We observe that 
\[
\chi_g(\lambda)=\chi_0(\lambda)+ g q(\lambda),\qquad
q(\lambda)=\lambda(\lambda^2+ P_1).\]
This structure can be understood by using the following splitting 
\begin{equation}\label{defSplitA}
A=\Theta-gD,
\qquad
\Theta=\begin{pmatrix}
0 & -a & 0 & b
\\
a-\tilde \kappa& 0  & -b & 0
\\
0 & b & 0 &-d
\\
-b & 0 & d& 0
\end{pmatrix},\qquad
D=\begin{pmatrix}
0 & 0 & 0 & 0
\\
0 & 1  & 0 &  -b/a
\\
0 & 0 & 0 &0
\\
0 & 0 & 0& 0
\end{pmatrix}
\end{equation}

The characteristic polynomial
of $\Theta$ is nothing but \[
\chi_0(\lambda)
=
\lambda^4+P_2\lambda^2+P_0.\]
The discriminant 
$$
\begin{array}{lll}
P_2^2-4P_0&=&a^2(a-\tilde \kappa)^2 + 4b^2a(a-\tilde \kappa)+2d^2a(a-\tilde \kappa)+4b^4+d^4+4b^2d^2
\\
&&
-
4d^2a(a-\tilde \kappa)+4adb^2+4db^2(a-\tilde \kappa)-4b^4
\\
&
=&
\big(a(a-\tilde \kappa)-d^2\big)^2 + +4(a+d)b^2\big(
(a-\tilde \kappa)+d
\big)
\end{array}
$$
 is positive, and thus the  eigenvalues $\lambda$ are obtained as square roots of $\frac12(-P_2\pm\sqrt{P_2^2-4P_0})=-\omega_\pm^2$, which are both negative 
 numbers. Therefore $\Theta$ admits only purely imaginary eigenvalues 
$$\sigma(\Theta)=\{-i\omega_-,i\omega_-,-i\omega_+,i\omega_+\}.$$
This part of the operator leads to an oscillating behavior, while damping relies on the $D$ part.
However, the image of $D$ acts only on the second variable of the system; hence a  transfer mechanism should propagate the damping to the other variables.
We will come back to the analysis of such mechanism later on.
For the time being, we use the explicit formula and  investigate the behavior of the eigenvalues with respect to $g$ by means of the implicit function theorem.
We take the derivative 
\[
\ds\frac{\ud}{\ud g}\big[\chi_g(\lambda (g))\big]=0
=
\big[\chi_0'(\lambda(g))+gq'(\lambda(g))\big]\lambda'(g)+ q(\lambda(g)) 
\]
at $g=0$, with $\lambda(0)= i\omega\in \sigma(\Theta)$.
We get
\[\lambda'(0)=
-\ds\frac{i\omega 
(( i \omega)^2+P_1)}{ 4(i \omega)^3+2P_2( i\omega)}
=
-\ds\frac{P_1-\omega^2}{2(P_2 -2\omega^2)}.
\]
 Coming back to the definition of $\omega_\pm$ it becomes
\[
\lambda'(0)=\mp \ds\frac{P_1-\omega_\pm^2}{2\sqrt{P_2^2-4P_0}}.
\]
We remark that 
\[\begin{array}{lll}
(P_1-\omega_-^2)(P_1-\omega_+^2)
&=&P_1^2-(\omega_-^2+\omega_+^2)P_1+\omega_-^2\omega_+^2
=P_1^2-P_1P_2+\ds\frac14(P_2^2-P_2^2+4P_0)
\\&=&P_1^2-P_1P_2+P_0<0
\end{array}\]
as it as been observed when checking the Routh criterion.
We deduce that 
\[
\omega_+^2<P_1<\omega_-^2\]
holds and thus $\lambda'(0)$ is always a negative real number.
Therefore, we distinguish four branches
\[\begin{array}{ll}
\lambda_1(g)=-i\omega_-- g\ds\frac{\omega_-^2-P_1}{2\sqrt{P_2^2-4P_0}}+g\epsilon_1(g),
\qquad & 
\lambda_2(g)=i\omega_--g\ds\frac{\omega_-^2-P_1}{2\sqrt{P_2^2-4P_0}}+g\epsilon_2(g),
\\
\lambda_3(g)=-i\omega_+-g\ds\frac{P_1-\omega_+^2}{2\sqrt{P_2^2-4P_0}}+g\epsilon_3(g),
\qquad & 
\lambda_4(g)=i\omega_+-g\ds\frac{P_1-\omega_+^2}{2\sqrt{P_2^2-4P_0}}+g\epsilon_4(g),
\end{array}\]
with $\lim_{g\to 0}\epsilon_k(g)=0$. 
Let us set (for instance) $\theta=\frac{\min(P_1-\omega_+^2,\omega_-^2-P_1)}{4\sqrt{P_2^2-4P_0}}$.
Then, we can find $g_0>0$ such that 
\[\mathrm{Re}(\lambda_k(g))\leq - g\theta \]
holds for any $0<g\leq g_0$.

Moreover, for $g>0$ small enough, the $4\times 4$ matrix $A$ admits 
4 distinct eigenvalues; it it therefore diagonalizable: we can find an invertible matrix $\Pi_g$ the columns of which are defined by the eigenvectors $v_1(g),v_2(g),v_3(g), v_4(g)$, such that 
$A=\Pi_g\Lambda_g\Pi_g^{-1}$, with $\Lambda_g=\mathrm{diag}(\lambda_1(g),\lambda_2(g),\lambda_3(g),\lambda_4(g))$.
By applying standard continuity arguments, see e.~g.~\cite[Chapter~II, Sect. 1.4]{Kato}
we can assume that $\|\Pi_g\|\leq 2\|\Pi_0\|$ and $\|\Pi_g^{-1}\|\leq 2\|\Pi_0^{-1}\|$ holds for $g\in [0,g_0]$, with $g_0$ small enough, and thus $\mathrm{Cond}(\Pi_g)\leq 4 \mathrm{Cond}(\Pi_0)$.
It follows that
\[\|e^{At}\|
=\|\Pi_g e^{\Lambda_g t}\Pi_g^{-1}\|\leq \mathrm{Cond}(\Pi_g)\|e^{\Lambda_g t}\|
\leq 4 \mathrm{Cond}(\Pi_0) e^{-\theta g t}\]
Coming back to the definition of $g$ by means of the parameters  of the problem, we can rephrase our findings as stated in Proposition~\ref{estim_expA} (where $\lambda_*=\frac{2\theta\upgamma}{\hbar }(\omega_{1/2}V^2_{1/2}+\omega_{3/2}V^2_{3/2})$).
\\

\noindent{\bf Kalman's conditions.}

\noindent
As announced above,  the proof 
of Proposition~\ref{estim_expA} can be revisited  within this framework of control theory 
\cite{AAM,AAM2,BZ},
 Kawashima-Shizuta's methods \cite{BZ,nanard,nanard2,KS}, or hypocoercivity techniques \cite{Vill}.
This permits us to exhibit key structural properties that can 
be fruitful to extend the analysis for systems in higher dimensions.
Here, we note that
\[\mathrm{rank}(D)=1,\qquad 
\mathrm{Ker}(D)=\{(y,\eta,z,\zeta)\in\mathbb R^4,\ \eta=(b/a)\zeta\}.\]
In fact, we simply have
\[
D=e_1\ell^\top,\qquad e_1=(0,1,0,0),\qquad \ell=(0,1,0,-b/a).\]
In contrast to the standard framework, $\Theta$
is not skew-symmetric -- and the symmetry is precisely broken by the influence of the coupling with the environment -- and $D$ is not symmetric
 and it is not positive semidefinite neither since 
$DY\cdot Y=\eta^2-(b/a)\eta \zeta$ has no definite sign.
The strategy is the following, starting from \eqref{defSplitA}
\begin{itemize}
\item Step 1. We construct a symmetrizer, that is a positive definite matrix $S=S^\top$ such that
$S\Theta+\Theta^\top S=0$ and $SD+D^\top S$ is semidefinite positive.
This matrix allows us to construct a dissipated energy functional $SY\cdot Y\geq 0$ such that
\[
\ds\frac{\ud}{\ud t}
SY\cdot Y= S(\Theta -gD)Y\cdot Y+SY\cdot(\Theta-gD)Y
=-g(SD+D^\top S)Y\cdot Y\leq 0.\]
In contrast, $\Theta$ has a conservative effect on the dynamics since
$\frac{\ud}{\ud t} [e^{\Theta^\top t} S e^{\Theta t}]=
e^{\Theta^\top t} (\Theta^\top S+S\Theta) e^{\Theta t}=0$ 
and $e^{\Theta^\top t} S e^{\Theta t}\big|_{t=0}=S$.
As a consequence $S e^{\Theta t}Y_0\cdot e^{\Theta t}Y_0=SY_0\cdot Y_0$.
\item Step 2. We decompose $(SD+D^\top S)=C^\top C$, and then 
we check the Kalman criterion $$\ds\bigcap_{k\in \mathbb N}\mathrm{Ker}(C\Theta^k)=\{0\}.$$
(Note that $\mathrm{Ker}(C)=\mathrm{Ker}(SD+D^\top S)$ so that the criterion can be equivalently expressed as 
$\bigcap_{k\in \mathbb N}\mathrm{Ker}((SD+D^\top S)\Theta^k)=\{0\}$.)
\end{itemize}

These properties imply the semigroup estimate in Proposition~\ref{estim_expA}, 
see \cite[Sections~1.4--1.5, Remark~6.1.4, and the proof
of Proposition~7.4.5]{TucsnakWeiss2009}.
However, the estimate is not fully explicit, nor sharp which motivates a more constructive approach inspired from the hypocoercivity framework \cite{AAM,AAM2,BZ}, that will be discussed later on.

With the 
requirements in Step~1, we find a symmetrizer
\[
S=\begin{pmatrix}
1-\tilde \kappa/a & 0 & -b/a& 0
\\
0& 1& 0& -b/a
\\
-b/a& 0 & d/a& 0
\\
0 & -b/a& 0& d/a
\end{pmatrix}.\]
Owing to the small size of the system it can be found  by trial and error
or solving linear systems. 
We will discuss a more general approach in the forthcoming section.
The rank-one  structure of $D$ drastically simplifies the computation:
\[SD+D^\top S=Se_1\ell^\top+\ell(Se_1)^\top.
\] 
This has the form $uv^\top +vu^\top$ and  the following statement applies.

\begin{lemma}\label{lem:SD}
Let $u,v$ be two vectors in $\mathbb R^N$. Then the symmetric matrix $M=uv^\top +vu^\top$ is 
 positive semidefinite iff $u=\tau v$ for some $\tau>0$.
\end{lemma}

This statement applies here since  $Se_1= \ell$ and thus 
we conclude that $SD+D^\top S=2\ell\ell^\top 
$.
It remains to check positivity of $S$, which 
 reduces to the positivity of the $2\times 2$ matrices
\[
S_1=\ds\frac1a\begin{pmatrix} a-\tilde\kappa& -b
\\-b& d\end{pmatrix},\qquad 
S_2=\ds\frac1a\begin{pmatrix} a& -b
\\-b& d\end{pmatrix}.\]
 Positivity is thus guaranteed by $a-\tilde \kappa>0$, $ad-b^2>0$ 
and $(a-\tilde\kappa)d-b^2>0$, the latter being 
a consequence of \eqref{small}. 
In fact, we recover quantities discussed when we analyzed the characteristic polynomial of 
$A$.
\QED

\noindent
{\bf Proof of Lemma~\ref{lem:SD}.}
We consider the quadratic form $$\mathfrak q: x\in \mathbb R^N\longmapsto Mx\cdot x
=2(u\cdot x)(v\cdot x).$$
It vanishes on $(\mathrm{Span}(u,v))^\perp$.
Therefore, it suffices to deal with 
\[
\tilde{\mathfrak q}:(\alpha,\beta)\in \mathbb R^2\longmapsto \ds\frac12M(\alpha u+\beta v)\cdot (\alpha u+\beta v)
=(\alpha |u|^2+\beta u\cdot v)(\alpha u\cdot v +\beta |v|^2).
\]
It turns out that 
\[\tilde{\mathfrak q}(\alpha,\beta)=
\tilde M\begin{pmatrix}\alpha\\\beta\end{pmatrix}\cdot \begin{pmatrix}\alpha\\\beta\end{pmatrix},\qquad
\tilde M=\ds\frac12\begin{pmatrix}
2|u|^2u\cdot v & (u\cdot v)^2+|u|^2|v|^2
\\ (u\cdot v)^2+|u|^2|v|^2 &  2|v|^2u\cdot v\end{pmatrix}.\]
We should impose that the trace  $(u|^2+|v|^2)(u\cdot v)$ is non negative,
and the determinant $-\frac12(|u|^2|v|^2-(u\cdot v)^2)^2 $ is non negative, too.
The former condition is equivalent to $(u\cdot v)\geq 0$ while the latter is fulfilled in  the equality case of the Cauchy-Schwarz inequality, that is when $u$ and $v$ are colinear.
\QED

For the V-model, we thus have
$SD+D^\top S=2\ell\ell^\top$ and we can set  
$C=\sqrt2\ell^\top$.
It simply expresses the dissipation relation
\[
\ds\frac{\ud}{\ud t}SY\cdot Y=-2g|\ell\cdot Y|^2
=-2g\Big(\eta-\ds\frac ba\zeta\Big)^2.\]

Due to the simplicity of the problem, it could be possible to compute explicitly 
$\ell^\top \Theta^k$ and reduce the verification of Kalman's criterion to the computation of a determinant.
Instead, it is convenient to  appeal to the PBH criterion
\cite{Bele,Hautus,Popov}
which here is particularly simple to  check and permits us to avoid the computation of 
the successive power $\Theta^k$.

\begin{lemma}
Let $\mathcal K=\bigcap_{k=0}^\infty\mathrm{Ker}(\ell^\top\Theta^k)\neq \{0\}$.
Then $\mathcal K=\{0\}$ iff the following assertion holds
\[\textrm{if $\Theta Y=\lambda Y$ and $\ell\cdot Y=0$, then $Y=0$.}\]
\end{lemma}

\noindent
{\bf Proof.}
We start by observing that  $\mathcal K=\bigcap_{k=0}^3\mathrm{Ker}(\ell^\top\Theta^k)\neq \{0\}$
by virtue of Cayley-Hamilton theorem.
If there exists $Y\neq 0$ such that $\Theta Y=\lambda Y$ and $\ell^\top Y=0$, 
then $\ell^\top \Theta^ k Y=\lambda^k \ell^\top Y=0$ for any $k\in \mathbb N$ and
$\bigcap_{k=0}^3\mathrm{Ker}(\ell^\top\Theta^k)\neq \{0\}$.

Conversely, suppose that $\mathcal K=\bigcap_{k=0}^3\mathrm{Ker}(\ell^\top\Theta^k)$
is non trivial.
For any $Y\in \mathcal K$ and $k\in\{0,1,2\}$, we have $\ell^\top \Theta^k(\Theta Y)=
\ell^\top \Theta^{k+1}Y=0$. As a consequence of Cayley-Hamilton theorem, we equally have 
$\ell^\top \Theta^3(\Omega X)=0$ so that $\Theta Y\in \mathcal K$.
Hence $\mathcal K$ is a non trivial subspace of $\mathbb C^4$, left invariant by $\Theta$.
Consequently, $\Theta\big|_{\mathcal K}$ admits an eigenvector $Y\in \mathcal K\setminus\{ 0\}$ which thus satisfies  $\ell\cdot Y=0$.
\QED

Here, it is particularly simple to verify that the PBH 
condition holds. We know that $0\notin \sigma(\Theta)$ and, with $Y=(y,\eta,z,\zeta)$, 
the condition $\ell^\top Y=0$ means $\eta=\frac ba \zeta$.
Therefore, the first eigenvalue-equation
yields $\lambda y=-a \eta + b\zeta=0$, thus $y=0$. Hence the second and forth equations become
$\lambda \eta=-b z$ and $\lambda\zeta =dz$, which leads to 
$z=\frac{\lambda}{ d} \zeta=\frac{\lambda a}{ bd} \eta=-\frac{a}{d}z$. Since $1+\frac ad>0$, this implies $z=0$ and, next,  $\zeta=0$, $\eta=0$.

Hence Step 1 \& 2 of the program described above can be accomplished and, as explained above, 
this leads to the estimate of Proposition~\ref{estim_expA}, see \cite{TucsnakWeiss2009}.

\subsubsection{Estimates for the non linear problem}
\label{sec:nonlin}

We take advantage of the fact that the limit has already been identified, see Theorem~\ref{th_cv}, Proposition~\ref{prop_cv} and Lemma~\ref{Phi0}. 
Let $\epsilon>0$: we can  
find $T^\epsilon>0$ large enough, such that $$0\leq (1+\xi_{1/2}(t))\leq \epsilon,\qquad 
 0\leq (1+\xi_{3/2}(t))\leq \epsilon,\qquad
 |\Xi(t)|\leq \epsilon$$
hold for any $t\geq T^\epsilon$.
Set
\[Y^\epsilon(t)=Y(t+T^\epsilon).\]
It satisfies the same system as $Y$, but with the initial data 
\[Y^\epsilon\big|_{t=0}=Y(T^\epsilon).\]
We are going to establish the decay estimate on $Y^\epsilon$.

We split $\psi^\epsilon=\psi^\epsilon_S+\psi^\epsilon_I$, and $\varpi^\epsilon=\varpi^\epsilon_S+\varpi^\epsilon_I$.
Owing to \eqref{hyp_supp}, we know that
$z\mapsto \psi_I(t,z)$ is supported in $B(0, ct+R_I\}$; hence $\mathrm{supp}(\psi_I(T^\epsilon,\cdot))\subset B(0,R_I+cT^\epsilon)$.
Similarly,  $z\mapsto \psi_S(T^\epsilon,z)
$ is supported in $B(0,R_*+cT^\epsilon)$, see e.~g.~\eqref{representation}.
Therefore, the wave field at time  $T^\epsilon$ is supported in 
$B(0,R^\epsilon)$, with $R^\epsilon=cT^\epsilon+\max(R_I,R_*)$.
By virtue of Huygens' principle, $\psi^\epsilon_I(t,\cdot)$  
is supported in $\{ct-R^\epsilon\leq |z|\leq ct+R^\epsilon\}$.
As a consequence, in the expression of the remainder $S^\epsilon(t):=S(t+T^\varepsilon)$, the terms
\[\ds\int_t^\infty k_c(s)\ud s, \qquad \ds\int_t^\infty s k_c(s)\ud s,\qquad
\Phi^\epsilon_I(t)=\ds\int_{\mathbb R^3} \sigma(z) \psi^\epsilon_I(t,z)\ud z
\]
vanish for $t$ large enough 
(the first two terms vanish for $t\geq \frac{2R_*}{c}$, the last integral vanish 
for $t\geq \frac{R^\epsilon+R_*}{c}$), see Lemma~\ref{CIW}.

Let $0<\lambda<\lambda_*$ and  set 
$$
\mathcal M^\epsilon_0(t)=\sup_{0\leq s\leq t}e^{\lambda s/c}|Y^\epsilon(s)|,\qquad
\mathcal M^\epsilon_1(t)=\sup_{0\leq s\leq t}e^{\lambda s/c}|\dot Y^\epsilon(s)|
.$$
We bear in mind that $Y^\epsilon$ is defined from an admissible solution of the original 
system; in particular we have $|\xi_{1/2}^\epsilon|,|\xi_{3/2}^\epsilon|\leq 1$, 
$|\Xi^\epsilon(0)|=|\Xi(T^\epsilon)|\leq 1$, and by construction, see \eqref{def_OmegaX},  we note that $\|e^{\Omega_X t}\|\leq 1$.
We use the shorthand notation $\lesssim$ 
for estimates that involve constants depending on the parameters of the models, but that do not depend on $\epsilon$ nor on $c\geq c_*$. We keep track of the worst terms; 
for instance, we denote $|\widetilde \ell_c|\lesssim \frac1{c^2}$ or $\epsilon \lesssim 1$.
The most critical term is
\[
e^{\lambda t/c}|\mathcal R^\epsilon(t)|
=e^{\lambda t/c}|\mathcal R(t+T^\epsilon)|
=e^{\lambda t/c}\left|\ds\int_0^t k_c(t-s)\left(\ds\int_s^t(\tau-s)\ddot y^\epsilon(\tau)\ud \tau\right)\ud s
\right|.
\]
We use the fact that $\dot y=-a\eta+b\zeta$, thus $\ddot y=-a\dot \eta+b\dot \zeta$ is a mere combination of components of $\dot Y$.
In turn, we get 
\[\begin{array}{lll}
e^{\lambda t/c}|\mathcal R^\epsilon(t)|
&\lesssim& 
\ds\int_0^t e^{\lambda (t-s)/c}|k_c(t-s)|\left(\ds\int_s^t(\tau-s)
e^{\lambda (s-\tau)/c}
e^{\lambda \tau/c}|\dot Y^\epsilon(\tau)|\ud \tau\right)\ud s
\\[.4cm]
&\lesssim& \mathcal M_1^\epsilon(t)
\ds\int_0^t e^{\lambda (t-s)/c}(t-s)^2|k_c(t-s)|
\ud s
\lesssim\mathcal M_1^\epsilon(t)
\ds\frac{e^{2\lambda R_*/c^2}}{c^2}
\end{array}\]
owing to the finite support property of $k_c$.
We have
\[
e^{\lambda t/c}|\Phi^\epsilon_I(t)|=e^{\lambda t/c}\left|\ds\int_{\mathbb R^3} \sigma(z)\psi_I^\epsilon(t,z)\ud z\right|\leq \sqrt{\kappa}\|\nabla \psi_I^\epsilon(t,\cdot)\|_{L^2(\mathbb R^3)}e^{\lambda (R^\epsilon+R_*)/c^2}
\lesssim e^{\lambda (R^\epsilon+R_*)/c^2},\]
using again finite support properties (and energy bounds).
Next, we estimate
\[\begin{array}{lll}
\left|
e^{\lambda t/c}\ds\frac{\widetilde\ell_c\cdot Y^\epsilon(t)+\mathcal R^\epsilon(t)+\Phi_I^\epsilon(t)}{\hbar}B\Xi^\epsilon(t)\right|&\leq& \ds\frac{\|B\|\epsilon}{\hbar} 
\big(|\widetilde\ell_c|\mathcal M^\epsilon_0(t) + e^{\lambda t/c}(|\mathcal R^\epsilon(t)|+|\Phi_I^\epsilon(t)|)\big)
\\[.4cm]
&\lesssim &\epsilon\left(\mathcal M^\epsilon_0(t)+
\mathcal M_1^\epsilon(t)
\ds\frac{e^{2\lambda R_*/c^2}}{c^2}+
e^{\lambda (R^\epsilon+R_*)/c^2}\right).
\end{array}\]
We obtain similarly estimates on the remainder terms 
$\widetilde {\mathcal R}^\epsilon$ and $\widetilde {\mathcal S}^\epsilon$
by taking advantage of 
the smallness of certain coefficients beyond $T^\epsilon$
\[\begin{array}{lll}
e^{\lambda t/c}|\widetilde {\mathcal R}^\epsilon(t)|&\lesssim &
\epsilon
\mathcal M_0^\epsilon(t)
+
\mathcal M_1^\epsilon(t)
\ds\frac{e^{2\lambda R_*/c^2}}{c^2}+
e^{\lambda (R^\epsilon+R_*)/c^2},
\\[.4cm]
e^{\lambda t/c}|\widetilde {\mathcal S}^\epsilon(t)|
&\lesssim &
\epsilon\left(\mathcal M_0^\epsilon(t)
+\mathcal M_1^\epsilon(t)
\ds\frac{e^{2\lambda R_*/c^2}}{c^2}+
e^{\lambda (R^\epsilon+R_*)/c^2}\right).
\end{array}\]
Therefore,  \eqref{mod_system4} yields
\[\begin{array}{lll}
e^{\lambda t/c}|\dot Y^\epsilon(t)|&\leq& \|A\|
e^{\lambda t/c}|Y^\epsilon(t)|+e^{\lambda t/c}|S^\epsilon(t)|
\\[.4cm]&\lesssim& 
\left(1+
\epsilon\right)\mathcal M^\epsilon_0(t)
+\ds\frac{e^{2\lambda R_*/c^2}}{c^2}
\mathcal M_1^\epsilon(t)
+
e^{\lambda (R^\epsilon+R_*)/c^2}
,\end{array}\]
which implies
\[\mathcal M_1^\epsilon(t)
\lesssim 
\left(1+\epsilon
\right) \mathcal M^\epsilon_0(t)
+\ds\frac{e^{2\lambda R_*/c^2}}{c^2}
\mathcal M_1^\epsilon(t)
+
e^{\lambda (R^\epsilon+R_*)/c^2}
.\]
Hence, 
from now on, we suppose that $c\geq c_*$ is large enough to ensure that
\[
0<\mu_0(c)=1-\ds\frac{e^{2\lambda R_*/c^2}}{c^2}<1,\]
and consequently
\[
\mathcal M_1^\epsilon(t)
\lesssim 
\ds\frac{1+\epsilon
}{\mu_0(c)}
\mathcal M^\epsilon_0(t)
+\ds\frac{e^{\lambda (R^\epsilon+R_*)/c^2}}{\mu_0(c)}.
\]
Now, we make use  of the Duhamel formula
\[
Y^\epsilon(t)=e^{At}Y(T^\epsilon)+\ds\int_0^t e^{A(t-s)}S^\epsilon(s)\ud s
\]
together with Proposition~\ref{estim_expA}.
We are led to 
\[
e^{\lambda t/c}|Y^\epsilon(t)|\leq 
K_*e^{(\lambda-\lambda_*)t/c}
\left(|Y(T^\epsilon)| +
\ds\int_0^t e^{(\lambda_*-\lambda)s/c} \ e^{\lambda s/c}|S^\epsilon(s)|\ud s\right)
.\]
Since $0<\lambda<\lambda_*$, with the estimates obtained above, it becomes
\[
e^{\lambda t/c}|Y^\epsilon(t)|\lesssim
C(\epsilon,c)+\mu_1(\epsilon,c)\ds\int_0^t 
\mathcal M^\epsilon_0(s)
\ud s
\]
with 
\[C(\epsilon,c)=\Big(1+\epsilon+\ds\frac{1}{\mu_0(c)}\Big)e^{\lambda (R^\epsilon+R_*)/c^2},\qquad
\mu_1(\epsilon,c)=
\epsilon+\ds\frac{e^{2\lambda R_*/c^2}}{\mu(c)c^2}.\]
From this, we infer
\[
\mathcal M^\epsilon_0(t)\lesssim C(\epsilon,c)+\mu_1(\epsilon,c)\ds\int_0^t 
\mathcal M^\epsilon_0(s)
\ud s,\]
%
and applying Gr\"onwall's lemma leads to
\[
| Y^\epsilon(t)|\lesssim C(\epsilon,c)e^{-(\lambda-c\mu_1(\epsilon,c))t/c}.
\]
Choosing $c$ large, and then, $\epsilon$ small, ensures the positivity of the rate
\[\lambda'=\lambda-c\mu_1(\epsilon,c)>0.\]
\QED

\section{Construction of the conservative symmetrizer and rank-one perturbation}
\label{sec:symmetrizer}

In this Section, we focus on the linear differential system 
\begin{equation}
\label{abstract_block_oscillator}
  \dot Y=(\Theta-gD)Y,
  \qquad
  Y=\begin{pmatrix}q\\p\end{pmatrix},
  \qquad
  \Theta=
  \begin{pmatrix}
    0&-\Omega\\
    \widetilde\Omega &0
  \end{pmatrix},
\end{equation}
where \(\Omega\) is symmetric definite positive. The matrix \(\widetilde\Omega\) is not assumed to be
symmetric and \(D\) has a reduced rank.
The strategy relies on the construction of suitable symmetrizers, 
which will be adapted to the  block structure.
It is useful to separate two  distinct questions:
\begin{enumerate}
  \item finding an inner product in which the conservative generator is
  skew-adjoint;
  \item deciding whether the  perturbation is dissipative in that
  inner product.
\end{enumerate}
The first question depends only on the conservative generator $\Theta$.  The second one
is an additional compatibility condition involving the dissipative matrix $D$.
The final touch will be to use this construction to improve the analysis of the large time asymptotics 
as performed in Section~\ref{LargeTimeVsys}: 
the asymptotics is governed by the estimate in Proposition~\ref{estim_expA}
which, as it has been derived by general arguments
so far, is not very explicit. 
Inspired by \cite{AAM,AAM2, BZ, Vill}, we will obtain sharper estimates.

\begin{proposition}
\label{prop:block_symmetrizer}
Assume that
$  \widetilde\Omega\Omega
$
is diagonalizable over \(\mathbb R\) with positive eigenvalues: 
\begin{equation}
\label{spectral_decomposition_G}
  \widetilde\Omega\Omega=R\Lambda R^{-1},
  \qquad
  \Lambda=\operatorname{diag}(\mu_0,\ldots,\mu_{n-1}),
  \qquad
  \mu_j>0.
\end{equation}
Then
\begin{equation}
\label{general_block_ML}
  M=(R^{-1})^\top R^{-1},
  \qquad
  L=\Omega^{-1}M\widetilde\Omega\Omega
\end{equation}
are both symmetric positive definite and
\begin{equation}
\label{general_block_S}
  S=
  \begin{pmatrix}
    L&0\\
    0&M
  \end{pmatrix}
\end{equation}
satisfies
\begin{equation}
\label{general_block_symmetry}
  S\Theta+\Theta^\top S=0.
\end{equation}

Conversely, if a block-diagonal positive definite matrix of the form
\eqref{general_block_S} satisfies \eqref{general_block_symmetry}, then
\(\widetilde\Omega\Omega\) is self-adjoint and positive in the scalar product induced by
\(M\).  In particular, \(\widetilde\Omega\Omega\) is diagonalizable and has positive real
eigenvalues.
\end{proposition}

\noindent {\bf Proof.}
The idea is to search for a symmetrizer under the form \eqref{general_block_S}
 with  $L,M$ required to be symmetric positive definite $n\times n$ matrices.
 We get
 \[
 S\Theta+\Theta^\top S=
 \begin{pmatrix}
0 & \widetilde \Omega^\top M -L\Omega 
\\
M\widetilde \Omega -\Omega L& 0
 \end{pmatrix} .\]
 This quantity vanishes provided 
 \[
 L=\Omega^{-1} M\widetilde \Omega
 =\widetilde \Omega^\top M\Omega^{-1}.\]
 If the second equality holds, then the formula indeed defines a symmetric matrix $L$.

Now, the construction relies on the spectral properties of $\widetilde \Omega\Omega$:
from \eqref{spectral_decomposition_G}, we set 
 \[M=(R^{-1})^\top R^{-1}.\]
 This symmetric matrix satisfies
 \[\begin{array}{lll}
 \Omega^{-1} M\widetilde \Omega
 &=&\Omega^{-1}(R^{-1})^\top R^{-1}R\Lambda R^{-1}\Omega^{-1}
 =\Omega^{-1}(R^{-1})^\top \Lambda R^{-1}\Omega^{-1}
 \\
 &=&
 \Omega^{-1}(R^{-1})^\top \Lambda  R^\top (R^\top)^{-1} R^{-1}\Omega^{-1}
 =\Omega^{-1}(R \Lambda  R^{-1})^\top (R^{-1})^\top R^{-1}\Omega^{-1}
 \\&=&\Omega^{-1}\big(\widetilde \Omega\Omega\big)^\top M\Omega^{-1}
 \\
 &=&
 \widetilde \Omega^\top M\Omega^{-1}.
 \end{array}\]
 If the eigenvalues $\mu_j$ are all positive, this also justifies that $L$ is positive definite.

Conversely, \eqref{general_block_symmetry} gives \(M\widetilde\Omega\Omega=\Omega L\).  Hence
$
  M\widetilde\Omega\Omega=\Omega L\Omega,
$
 is symmetric positive definite.  It means that  \(\widetilde\Omega\Omega\) is self-adjoint and
positive for the inner product induced by \(M\).
\QED

\begin{rmk}\label{im_spec}
It is intuitive to think of $\widetilde \Omega$ as a perturbation from $\Omega$
so that $\Theta$ keeps a purely imaginary spectrum, therefore inducing oscillations in the dynamics.
This is a consequence of the assumption on the spectrum of $\widetilde\Omega\Omega$.
Indeed, we have
\[
\Theta^2=-\begin{pmatrix}\Omega \widetilde\Omega & 0 \\
0 & \widetilde\Omega\Omega\end{pmatrix}.\]
Therefore for an eigenpair $(\lambda,(q,p))$ of  $\sigma(\Theta)$, we have $\widetilde\Omega\Omega p=-\lambda^2 p$  and $\widetilde\Omega\Omega (\Omega^{-1}q)=-\lambda^2 (\Omega^{-1}q)$ so that 
$-\lambda^2\in \sigma(\widetilde\Omega\Omega)\subset (0,\infty)$.

In fact, Proposition~\ref{prop:block_symmetrizer}
is equivalent to assert that $\Theta$ is diagonalizable in $\mathbb C$ with purely imaginary spectrum.
Indeed, from \eqref{general_block_symmetry}, set $J=S^{1/2}\Theta S^{-1/2}$. Since
$S\Theta+\Theta^\top S=S^{1/2}(J+J^\top)S^{1/2}=0$,
it satisfies
$J^\top =-J$ and thus 
it is diagonalizable with its spectrum included in $i\mathbb R$. 
The property transfers to the semblable matrix $\Theta$.
Conversely, from $\Theta=P\mathrm{diag}(\lambda_0,...,\lambda_{n-1})P^{-1}$, with all $\lambda_j\in i\mathbb R$, we infer the symmetrizer $S=(P^{-1})^\top P^{-1}$.
\end{rmk}

The construction simplifies considerably when $\widetilde \Omega$ is symmetric definite positive.

\begin{coro}\label{Prop:sym}
Let $\widetilde \Omega$ be symmetric definite positive. Then, the symmetrizer obtained in Proposition~\ref{prop:block_symmetrizer} reads
\begin{equation}\label{simple_symmetric_symmetrizer} 
S=\begin{pmatrix}\widetilde \Omega& 0\\
0 &  \Omega
\end{pmatrix}.\end{equation}
\end{coro}

\noindent
{\bf Proof.}
Indeed, both $M=\Omega$ and $L=\widetilde \Omega$ are now symmetric definite positive and satisfy
$M\widetilde \Omega=\Omega L$.
The connection with Proposition~\ref{prop:block_symmetrizer} 
appears through the diagonalized form of
\[
  \mathcal L=\Omega^{1/2}\widetilde \Omega \Omega^{1/2}
  =U\Lambda U^\top, 
\]
where $\Lambda$ is the diagonal matrix made of the (real and positive) eigenvalues of  
the (symmetric definite positive) matrix $\mathcal L$.
Then, we observe that  \(R=\Omega^{-1/2}U\) satisfies
\begin{equation}\label{defR}
  \widetilde \Omega \Omega=R\Lambda R^{-1},
  \qquad
  (R^{-1})^\top R^{-1}=\Omega.
\end{equation}
Therefore 
 \eqref{simple_symmetric_symmetrizer} just corresponds to
 \eqref{general_block_ML} when \(\widetilde \Omega\) is symmetric.
 It defines the energy functional 
\[
Y^\top SY
  =q^\top \widetilde \Omega q+p^\top \Omega p.
\]
\QED

\begin{coro}[Symmetric rank-one perturbation]
\label{cor:symmetric_rank_one_symmetrizer}
Let
$
  \widetilde\Omega=\Omega-\alpha vv^{\top}
$
for some $v\in \mathbb R^{n}\setminus\{0\}$ and $\alpha >0$.
If
\begin{equation}
  \alpha v^{\top }\Omega^{-1}v<1,
  \label{eq:symmetric_rank_one_weak_coupling}
\end{equation}
then \[
  S=
  \begin{pmatrix}
    \Omega-\alpha vv^{\top}&0\\
    0&\Omega
  \end{pmatrix}
  \]
is symmetric positive definite and it satisfies
$
  S\Theta
  +\Theta^{\top}S=0.
$
\end{coro}

\noindent
{\bf Proof.}
This is a direct application of Corollary~\ref{Prop:sym}.
We only have to check that $\widetilde\Omega$ is symmetric definite positive.
For every \(x\in\mathbb R^n\), the Cauchy--Schwarz inequality  gives
\[
  (v^\top  x)^2
  \leq
  \bigl(v^\top  \Omega^{-1}v\bigr)
  \bigl(x^\top  \Omega x\bigr).
\]
Consequently, \eqref{eq:symmetric_rank_one_weak_coupling} yields
\[
  x^\top  \widetilde\Omega x
  =x^\top  \Omega x-\alpha(v^\top  x)^2
  \geq
  \left(1-\alpha v^\top  \Omega^{-1}v\right)
  x^\top  \Omega x>0,
\]
proving the claim.  
\QED

Let us now discuss examples relevant for applications we are interested in
where  $\widetilde \Omega$ is obtained as a rank-one  perturbation of a diagonal matrix:
$$\Omega=\mathrm{diag}(\omega_{0},...,\omega_{n-1}), \qquad 
 0<\omega_{0}<...<\omega_{n-1}.$$

\begin{example}[Rank-one perturbation of a diagonal matrix]\label{Rk1pert}
Consider the case where $$\widetilde \Omega =\Omega-\alpha uv^\top,$$ 
with 
 $\alpha>0$, and $u,v\in \mathbb R^{n}$ with positive components.
Notice that $\widetilde \Omega$ is not symmetric, thus breaking the skew symmetry of $\Theta$ which holds when $\alpha$ vanishes.
The applicability of the symmetrization relies on the 
the spectral properties 
 of $\widetilde \Omega\Omega=\Omega^2-\alpha u(\Omega v)^\top$.
 Owing to the simple structure of the matrices, the construction can be performed explicitly, by proceeding as follows.

We remind that, given two vectors $x,y$ in $\mathbb R^{n}$, $\mathrm{Ran}(xy^\perp)=\mathrm{Span}(x)$, 
$\mathrm{Ker}(xy^\perp)=(\mathrm{Span}(y))^\perp$ and 
$(xy^\top)x=
(y^\top x)x$. Thus the eigenvalues
are 0, with multiplicity $n$ and $y^\top x$  with multiplicity $1$: it follows that
$\mathrm{det}(\mathbb I+xy^\top)=1+y\cdot x$.
Next, we write, for $\mu\notin\{\omega_{0}^2,...,\omega^2_{n-1})$,
\[\begin{array}{lll}
\mathrm{det}(\widetilde\Omega\Omega-\mu \mathbb I)
&=&\mathrm{det}(
(\Omega^2-\mu \mathbb I)(
(\mathbb I-\alpha(\Omega^2-\mu \mathbb I)^{-1}u(\Omega v)^\top))
\\&=&
\mathrm{det}
(\Omega^2-\mu \mathbb I)\times
\mathrm{det}((\mathbb I -\alpha(\Omega^2-\mu \mathbb I)^{-1}u(\Omega v)^\top)
\end{array}\]
and we take advantage of the fact that $\Omega$ is diagonal:
 $(\Omega^2-\mu \mathbb I)^{-1}u(\Omega v)^\top$ remains rank-one  and we get
 \begin{equation}\label{f_appears}
 \mathrm{det}(\widetilde\Omega\Omega-\mu \mathbb I)
 =\ds\prod_{j=0}^{n-1} (\omega^2_{j}-\mu)\times \left(
 1-\alpha\ds\sum_{j=0}^{n-1} \ds\frac{\omega_{j} u_jv_j}{\omega^2_{j}-\mu}\right).
 \end{equation}
 For any $\mu\in \mathbb C$ the formula generalizes into
 \[
 \mathrm{det}(\widetilde\Omega\Omega-\mu \mathbb I)=
 \ds\prod_{j=0}^{n-1} (\omega^2_{j}-\mu)-\alpha \ds\sum_{j=0}^{n-1}\left(\omega_{j} u_jv_j\prod_{\ell\neq j} (\omega^2_{\ell}-\mu)\right),
 \]
 since the expressions on both sides of the equality are polynomials wrt $\mu$ and
 agree away from finitely many points.
 Consequently,
 \[ \mathrm{det}(\widetilde\Omega\Omega-\omega_k^2 \mathbb I)=
 -\alpha \omega_ku_kv_k \prod_{\ell\neq k} (\omega^2_\ell-\omega_k^2)\]
 does not vanish since we assume $u_kv_k>0$. 
 Eigenvalues are determined by studying 
 when the quantity \eqref{f_appears} vanishes, that is  when 
 $\mu$ is a zero of the function
 \[
 f:\mu\longmapsto 1-\alpha\ds\sum_{j=0}^{n-1} \ds\frac{\omega_j u_jv_j}{\omega^2_j-\mu}.\]
 The function $f$ is continuous and decreasing on the interval $(\omega^2_j,\omega^2_{j+1})$ with 
 $f(\omega^2_j)=+\infty$, $f(\omega^2_{j+1})=-\infty$, so it vanishes
 on such intervals.
 We obtain this way $n-1$ positive eigenvalues. 
 It remains to discuss the case where $u_0v_0>0$ and $\mu$ lies in $(-\infty,\omega_0^2)$.
By the same argument, there exists a zero $\mu_0$ of $f$ in this interval.
We guarantee $\mu_0>0$ by assuming $ \alpha v^\top\Omega^{-1}u<1$ which implies $f(0)>0$.
The associated eigenvectors are obtained by the formula
\[  p_\mu=(\Omega^2-\mu I)^{-1}u.
  \]
Indeed, $f(\mu)=0$ gives
\(\alpha(\Omega v)^{\top}p(\mu)=1\), and hence
$
  (\widetilde\Omega\Omega-\mu \mathbb I)p(\mu)
  =u-\alpha u(\Omega v)^{\top}p(\mu)=0.
$
This permits us to construct the matrix $R$ and then the symmetrizer $S$.
\end{example}

\begin{example}[One-row perturbation of a diagonal matrix]
\label{rem:one_row_rank_one_symmetrizer}
The construction is completely explicit when
\[
  \widetilde \Omega=\Omega-\alpha e_0 v^{\top},\qquad v\in \mathbb R^{n}\]
In that case
\[
\widetilde \Omega\Omega  =\Omega^2-\alpha e_0(\Omega v)^{\top}
\]
differs from \(\Omega^2\) only in its first row.
To be more specific, this is the upper diagonal matrix
 given by  
  \[
  \left(
     \raisebox{0.5\depth}{%
       \xymatrixcolsep{1ex}%
       \xymatrixrowsep{1ex}%
       \xymatrix{
         \omega_{0}^2 -\alpha \omega_{0}v_0 & -\alpha \omega_{1}v_1\  \ar@{-}[rrr] &&& -\alpha  \omega_{n-1}v_{n-1}  
         \\ 0 \ar @{.}[rrrddd] \ar @{.}[ddd]& \omega_{1}^2\ar @{-}[rrrddd]
         &0\ar@{.}[rr] \ar@{.}[rrdd] && 0 \ar@{.}[dd]
           \\           &&&& 
         \\
          &&  & &   0
         \\
         0\ar @{.}[rrr]&&&0& \omega_{n-1}^2
       }%
     }
   \right).\]
   The eigenvalues 
   are $$\mu_0=\omega_{0}^2-\alpha \omega_{0}v_0,\quad
   \mu_1=\omega^2_{1},\quad ...,\quad \mu_{n-1}=\omega^2_{n-1}.$$
   They are all positive and distinct provided
   \begin{equation}\label{smallVS}0<\alpha<\ds\frac{\omega_0}{v_0}.
   \end{equation}
The corresponding eigenvectors, that 
   give the columns of $R$ in \eqref{defR}, are
   \[\textrm{
   $r_0=e_0$ and 
   $r_j=e_j+a_j e_0$ with $ a_j = \ds\frac{\alpha\omega_{j} v_{j}}{\omega^2_{0}-\omega^2_{j} -\alpha v_{0}\omega_{0}}$  for $j\in \{1,...,n-1\}$.}\]
 Therefore the matrix $R$ has a simple expression
 $$
 R=\left(
     \raisebox{0.5\depth}{%
       \xymatrixcolsep{1ex}%
       \xymatrixrowsep{1ex}%
       \xymatrix{
         1  \ar @{-}[rrrrdddd]& a_1 \ar@{-}[rrr] &&& a_{n}
         \\ 0 \ar @{.}[rrrddd] \ar @{.}[ddd]&&0\ar@{.}[rr] \ar@{.}[rrdd] && 0 \ar@{.}[dd]
           \\           &&&& 
         \\
          &&  & &   0
         \\
         0\ar @{.}[rrr]&&&0& 1
       }%
     }
   \right)=\mathbb I +e_0a^\top
 ,$$
 where $a=(0,a_1,...,a_{n}\}$. Since $a^\top e_0=0$, we get
$R^{-1}=\mathbb I-e_0 a ^\top$
and $$M
=(\mathbb I-e_0 a^\top )^\top(\mathbb I-e_0 a^\top )
=\mathbb I-a e_0 ^\top -  e_0a^\top+aa^\top.$$
By definition this is a symmetric definite positive matrix.
We finally set $
  L=\Omega^{-1}M \widetilde \Omega$
 to find  an explicit
positive symmetrizer.  

A similar argument can be used to study the spectrum of $\Theta$: 
that $(\lambda,(q,p))$ is an eigenpair for $\Theta$ leads to 
 \[
 \lambda ^2 q=-\Omega\Big(\Omega -\alpha e_0V^\top\Big)q,\]
 with $q\neq 0$.
 In other words $-\lambda^2$ is an eigenvalue of 
 $\Omega^2-\alpha\Omega e_0 V^\top$.
   The matrix being upper triangular, its eigenvalues can be found on its diagonal. Owing to \eqref{smallVS}, they are all positive and we conclude that $\sigma(\Theta)\subset i\mathbb R$.
\end{example}

\begin{example}[The 3 levels V-system]
In the analysis of the V-system, we decided to work with a new set of variables, obtained through the matrix $\mathscr V$. 
We observe that $\big(\frac{\mathscr V}{\mathpzc{v}}\big)^2=\mathbb I$.
The matrices
$$\mathscr V\Omega \mathscr V^{-1}=\begin{pmatrix} a& -b \\-b& d\end{pmatrix},\qquad
 \mathscr V\widetilde\Omega \mathscr V^{-1}=\begin{pmatrix} a-\tilde\kappa& -b \\-b& d\end{pmatrix}
 $$ remain symmetric.
The natural symmetrizer for the variable $\hat Y=(y,z,\eta,\zeta)$ thus reads
\[\widehat S=\begin{pmatrix}  \mathscr V\widetilde\Omega \mathscr V^{-1}& 0 \\0& \mathscr V\Omega \mathscr V^{-1}\end{pmatrix}.\]
The symmetrizer used above arises by coming back to the variables in the order $Y=(y,\eta,z,\zeta)$.
\end{example}

\begin{lemma}[PBH condition]
\label{lem:PBH}
Let $\mathcal L$ be a symmetric definite positive $n\times n$ matrix and $w\in\mathbb R^n$. 
Let $\mathcal L=U\operatorname{diag}(\mu_1,...,\mu_n) U^\top$
be an orthogonal diagonalization of $\mathcal L$, and set
$  a=U^\top w.$
We denote $(u_0,...,u_{n-1})$ the columns of $U$.
The
pair $(\mathcal L,w)$ satisfies the PBH condition
\begin{equation}
  \mathcal Lu=\mu u,\qquad w^\top u=0
  \quad\Longrightarrow\quad u=0.
  \label{eq:PBH_Lw}
\end{equation}
iff 
 $\mathcal L$ has simple eigenvalues and, for any $j\in\{0,...,n-1\}$,
\(
  w^\top u_j\neq0\).
 Equivalently,  
 \(a_j\neq0\)  for any $j\in\{0,...,n-1\}$.
\end{lemma}

\noindent
{\bf Proof.} Suppose that \eqref{eq:PBH_Lw} holds.
We first prove that the eigenvalues of $\mathcal L$ are simple. Suppose, by
contradiction, that 
\[
  E_\mu=\mathrm{Ker}(\mathcal L-\mu \mathbb I)
\]
has dimension at least two. The restriction to $E_\mu$ of the linear
form
\[
  u\longmapsto w^\top u
\]
maps a space of dimension at least two into $\mathbb R$ and therefore
cannot be injective. Consequently, there exists
$u\in E_\mu\setminus\{0\}$ such that
\(  w^\top u=0\).
which contradicts \eqref{eq:PBH_Lw}. 
Next, we show 
that the coefficients $a_j=u_j^\top w=w^\top u_j$ do not vanish. 
If $a_j=0$ for some $j$, then the nonzero eigenvector $u_j$ would
satisfy
 $ \mathcal Lu_j=\mu_j u_j,$ and 
 $ w^\top u_j=0,$
again contradicting \eqref{eq:PBH_Lw}. 
The opposite implication is direct.
\QED

\noindent
{\bf Compatibility with a rank-one dissipation}

\noindent
Suppose now that $D$ is rank-one acting only on momentum variables
\begin{equation}
\label{rank_one_D_general}
  D=
  \begin{pmatrix}
    0&0\\
    0&uv^\top 
  \end{pmatrix}.
\end{equation}
Then
\[
  SD+D^\top S
  =
  \begin{pmatrix}
    0&0\\
    0&(Mu)v^\top +v(Mu)^\top 
  \end{pmatrix}.
\]
Lemma~\ref{lem:SD} shows that the last matrix is
non-negative semidefinite iff
\begin{equation}
\label{rank_one_compatibility}
  Mu=\tau v
  \qquad\text{for some }\tau\geq0.
\end{equation}
When \(\tau>0\), 
the differential system 
\eqref{abstract_block_oscillator} with \eqref{rank_one_D_general}
entails damping since, with $Y=(q,p)$,
\[
 \ds \frac{\mathrm d}{\mathrm dt} SY\cdot Y
  =-2g\tau|v^\top p|^2.
\]
We point out that condition \eqref{rank_one_compatibility} is independent of the construction
of the conservative symmetrizer: it is the additional alignment which makes
the rank-one term genuinely dissipative.

\section{Generalized V-systems}
\label{sec:generalized_V}

The model we are dealing with couples the  state of lowest energy  to \(n\) mutually uncoupled states:
\[
V= \begin{pmatrix} 0 & v^\top
 \\
 v & 0\end{pmatrix},\qquad
 v=(V_{1/2},...,V_{n-1/2}),\qquad V_{j+1/2}>0.
\]
The full Hamiltonian has the arrowhead form
\[
H=
\left(
     \raisebox{0.5\depth}{%
       \xymatrixcolsep{1ex}%
       \xymatrixrowsep{1ex}%
       \xymatrix{
         \hbar \omega_0  \ar @{-}[rrrrddddd] & V_{1/2} \ar @{-}[rrr]  &&& V_{n-1/2} 
         \\V_{1/2}   \ar @{-}[ddd]   & &0  \ar @{.}[rrdd]   \ar @{.}[rr]  &&0 \ar @{.}[dd] 
           \\      
         \\
          &0  \ar @{.}[dd]  \ar @{.}[rrdd]  &  & &   0
          \\
         \\
         V_{n-1/2} &0 \ar @{.}[rr] &&0& \hbar \omega_n
       }%
     }
   \right).
\]
The  model has similarities to the system introduced in   \cite{HOR}
where distant atoms have distinct excited states dipole-coupled to one common ground state.
Another  discussion involving closely spaced  atomic levels coupled to a single state arises in \cite{Demkov, Kirova}.

In order to remain consistent with the notation in Section~\ref{sec:Vmod}, we distinguish 
ground-to-excited coherences and excited-to-excited coherences by setting
\[\begin{array}{l}
  \rho_j:=\rho_{j,j}, \quad j\in\{0,...,n\},
  \qquad
  \rho_{j+1/2}:=\rho_{0,j+1}, \quad  j\in\{0,...,n-1\},
  \\
  \Xi_{j,k}:=\rho_{j,k},\quad j,k\in \{1,...,n\},\, j<n.
  \label{eq:gV_density_entries}
\end{array}\]
The wave is sourced by the observable
\[
  \operatorname{Tr}(\rho V)
  =2\operatorname{Re}\left(\sum_{j=0}^{n-1} V_{j+1/2}\rho_{j+1/2}\right).
 \]
We introduce the Bohr frequencies
\[
 \textrm{$\omega_{j+1/2}=\omega_{j+1}-\omega_0>0$ for $j\in \{0,...n-1\}$ and 
$\nu_{j,k}=\omega_k-\omega_j>0$ for 
  $1\leq j<k\leq n$}.
  \]
In what follows $\Omega$ stands for $\mathrm{diag}(\omega_{1/2},...,\omega_{n+1/2})$.
The Liouville equation
$
  i\hbar\dot\rho=[\hbar\Omega_{\mathrm{free}}+\Phi V,\rho]
$
can be cast as the following componentwise system:
\begin{equation}
\label{eq:gV_complex_system}
\begin{array}{l}
  \dot\rho_0
  =-\ds\frac{2\Phi}{\hbar}
    \sum_{\ell=0}^{n-1} V_{\ell+1/2}\operatorname{Im}\rho_{\ell+1/2},
  \\
  \dot\rho_j
  =\ds\frac{2\Phi}{\hbar}V_{j+1/2}\operatorname{Im}\rho_{j+1/2},
  \qquad j\in\{1,..., n\},
  \\
  \dot \rho_{j+1/2}
  =i\omega_{j+1/2}\rho_{j+1/2}
    -\ds\frac{i\Phi}{\hbar}
      \left(
        V_{j+1/2}(\rho_j-\rho_0)
        +\ds\sum_{\substack{k=1\\k\neq j}}^n V_{k+1/2}\Xi_{k,j}
      \right),
  \qquad j\in\{0,..., n-1\},\\
  \dot \Xi_{j,k}
  =i\nu_{j,k}\Xi_{j,k}
    -\ds\frac{i\Phi}{\hbar}
      \left(V_{j+1/2}\rho_{k+1/2}-V_{k+1/2} \rho_{j+1/2}^*\right),
  \qquad j,k\in\{1,...,n\}.
\end{array}
\end{equation}
coupled to 
\[
  (\partial_t^2-c^2\Delta)\psi
  =-2c^2\sigma\operatorname{Re}
      \left(\sum_{j=0}^{n-1} V_{j+1/2}\rho_{j+1/2}\right).
  \]
The system has been derived by using, for  \(j\in\{0,..., n-1\}\),
\[
  [V,\rho]_{0,j+1}
  =V_{j+1/2}(\rho_{j+1}-\rho_0)
   +\sum_{\substack{k=0\\k\neq j}}^{n-1} V_{k+1/2}\Xi_{k+1,j+1},
\]
whereas, for \(j,k\in\{0,...,n-1\}\),
\[
  [V,\rho]_{j+1,k+1}=V_{j+1/2}\rho_{k+1/2}-V_{k+1/2}\rho^*_{j+1/2}.
\]
These identities explain both the coupling of the ground-to-excited
coherences to the excited block and the absence of any direct field term in
the evolution of the populations other than through the  \(\rho_{j+1/2}\)'s.

We now introduce  real variables as in Section~\ref{sec:Vmod}:
\[
  2\rho_{j+1/2}=q_{j+1/2}+ip_{j+1/2},
  \qquad
  \xi_{j+1/2}=\rho_j-\rho_0,
  \qquad j\in\{0,...,n-1\},
 \]
and, for \(j,k\in\{1,...,n\}\), $j<k$,
\[
  2\Xi_{j,k}=q^\Xi_{j,k}+ip^\Xi_{j,k}.
 \]
We collect the quantities into vectors $
  q=(q_{1/2},... ,q_{n+1/2})
  $,  $ p=(p_{1/2},...,p_{n+1/2})$ and the matrix $\Xi$ accordingly.  
The trace constraint gives
\[
  \rho_0=\frac{1-\sum_{\ell=0}^{n-1}\xi_{\ell+1/2}}{n+1},
  \qquad
  \rho_j
  =\frac{1-\sum_{\ell=0}^{n-1}\xi_{\ell+1/2}}{n+1}+\xi_{j+1/2}.
 \]
In these variables, \eqref{eq:gV_complex_system} becomes
\[
\begin{array}{l}
  \dot\xi_{j+1/2}
  =\ds\frac{\Phi}{\hbar}
    \left(V_{j+1/2}p_{j+1/2}+\sum_{\ell=0}^{n-1} V_{\ell+1/2} p_{\ell+1/2}\right),
  \qquad j\in \{0,...,n-1\},
 \\
  \dot q_{j+1/2}
  =-\omega_{j+1/2}p_{j+1/2}
    +\ds\frac{\Phi}{\hbar}
      \ds\sum_{\substack{k=0\\k\neq j}}^{n-1} V_{k+1/2}p^\Xi_{k+1,j+1},
 \qquad j\in \{0,...,n-1\},
 \\
  \dot p_{j+1/2}
  =\omega_{j+1/2}q_{j+1/2}
    -\ds\frac{2\Phi}{\hbar}V_{j+1/2}\xi_{j+1/2}
    -\ds\frac{\Phi}{\hbar}
      \sum_{\substack{k=0\\k\neq j}}^{n-1} V_{k+1/2}q^\Xi_{k+1,j+1},
  \qquad j\in \{0,...,n-1\},
 \\
  \dot q^\Xi_{j+1,k+1}
  =-\nu_{j+1,k+1}p^\Xi_{j+1,k+1}
    +\ds\frac{\Phi}{\hbar}(V_{j+1/2}p_{k+1/2}+V_{k+1/2}p_{j+1/2}),
  \quad j,k\in \{0,...,n-1\}, \ j<k,
  \\
  \dot p^\Xi_{j+1,k+1}
  =\nu_{j+1,k+1}q^\Xi_{j+1,k=1}
    -\ds\frac{\Phi}{\hbar}(V_{j+1/2}q_{k+1/2}-V_{k+1/2}q_{j+1/2}),
  \qquad j,k\in \{0,...,n-1\}, \ j<k.
  \end{array}
\]
The source of the wave equation takes the particularly simple form
\(
  \operatorname{Tr}(\rho V)
=v^{\top}q.
 \)
With only 3 levels, we recover \eqref{Vsystembis}.

\begin{theo}\label{th_cv_VS} Assume 
\begin{equation}
\ds\frac{2\kappa}\hbar v^\top \Omega^{-1}v<1.
  \label{small2}
\end{equation}
Then,
$\lim_{t\to \infty}\big(q_{j+1/2},p_{j+1/2}\big)(t)=0$ for any $j\in \{0,...,n\}$.
%
In the pure case, if, moreover, the initial energy $\mathcal E(0)$ is smaller than $\hbar \omega _1$, then we have $
\lim_{t\to \infty}\rho(t)=e_0e_0^\top.$
Finally, in the latter case, there exists $K_*,\lambda_*,c_*>0$ such that for any $c\geq c_*$, 
$\|\rho(t)-e_0e_0^\top\|\leq K_*e^{-\lambda_* t/c}.
$
\end{theo}

By using the localized energy as in Section~\ref{LargeTimeVsys}
and reproducing the same reasoning mutatis mutandis, 
 we obtain the analogue to Lemma~\ref{lem:cv_p} and Corollary~\ref{coro_cv_varpi}:
 $\lim_{t\to \infty} (\Omega V)^\top p=0$ and $\lim_{t\to \infty} \omega_S(t,z)=0$
 (in $C^0([0,T];L^2(\mathbb R^3)-weak)$.
 Accordingly, for an certain increasing sequence $t^\nu$ 
  that tends to $\infty$, we can pass to the limit $X(t+t^\nu)\to X^\infty(t)$ uniformly on $[0,T]$, 
  and $\Phi(t+t^\nu)\to \Phi^\infty$. We can indeed show that 
  $(\Omega V)^\top p^\infty=0$ and 
  $v^\top q^\infty=-\kappa\Phi^\infty$. 
We turn to the analog of Lemma~\ref{admker}.
 
 \begin{proposition}\label{prop:optim}
Assume 
\eqref{small2}.
If $\Phi^\infty\neq 0$, then there is no density matrix \(\rho\) satisfying simultaneously
\[
  [\hbar \Omega_{\mathrm{free}}+\Phi^\infty V,\rho]=0,
  \textrm{ and }
  \Phi^\infty=-\kappa\mathrm{Tr}(\rho V).
\]
\end{proposition}

\noindent
{\bf Proof.}
For the generalized model, it is finally easier to deal with the density matrix instead of the coordinates.
Let $K\in \mathcal M_{n+1}(\mathbb R)$ (that will be constructed later on) be such that 
\[[K,\Omega_{\mathrm{free}}]=-V.\]
Then, we have
\[\begin{array}{lll}
\mathrm{Tr}(\rho V)&=&-\mathrm{Tr}([K,\Omega_{\mathrm{free}}]\rho)
=-\ds\frac1\hbar \mathrm{Tr}([K,\hbar \Omega_{\mathrm{free}} +\Phi^\infty V]\rho)
+\ds\frac{\Phi^\infty}\hbar \mathrm{Tr}([K, V]\rho)
\\&=&
+\ds\frac1\hbar \mathrm{Tr}([\hbar \Omega_{\mathrm{free}} +\Phi^\infty V,\rho]K)
+\ds\frac{\Phi^\infty}\hbar \mathrm{Tr}([K, V]\rho)
.\end{array}\]
by using the commutation relations \eqref{commut}.
Hence, if $\rho$ lies in the kernel of the operator 
$A^\infty:\rho\mapsto [\hbar \Omega_{\mathrm{free}} +\Phi^\infty V,\rho]$, then we obtain 
\[
\mathrm{Tr}(\rho V)=\ds\frac{\Phi^\infty}\hbar \mathrm{Tr}([K, V]\rho)
.\]
Imposing $\Phi^\infty=-\kappa\mathrm{Tr}(\rho V)$ with $\Phi^\infty\neq 0$ yields
\[
\mathrm{Tr}([K, V]\rho)=-\ds\frac{\hbar}{\kappa}.\]
Inspired by \eqref{defO},
we find 
\[
K=\begin{pmatrix}
    0&-(\Omega^{-1} v)^\top  \\
    \Omega^{-1} v&0
  \end{pmatrix}
,\qquad 
O:=[K,V]=
\begin{pmatrix}
    -2v^\top  \Omega^{-1}v&0\\
    0&\Omega^{-1}v v^{\top}+v(\Omega^{-1}v)^\top  
  \end{pmatrix}
.\]
Hence, assuming $\Phi^\infty\neq 0$ we arrive at the 
relation
\begin{equation}\label{eq_tr_rhoO} \ds\frac{\hbar}{\kappa}
  =-\mathrm{Tr}(\rho O).\end{equation}
  Bearing in mind that $\rho$ is a density matrix, we can derive  as in the proof of Lemma~\ref{admker} a rough estimate with $\|O\|$, but this would miss 
the sharp criterion \eqref{small2}.
We need to use further information coming from the fact that we are looking at density matrices that satisfy 
\begin{equation}\label{ZeroComm}[\hbar \Omega_{\mathrm{free}}+\Phi^\infty V,\rho]=0.\end{equation}
Equation \eqref{ZeroComm} tells us that  \(\hbar \Omega_{\mathrm{free}}+\Phi^\infty V\) and \(\rho\) are commuting Hermitian matrices; therefore they admit a common orthonormal eigenbasis and we can expand
$$
\rho=\sum_{m=0}^n \uppi_m u^{(m)}u^{(m)*},
\qquad
\uppi_m\geq0,\qquad \sum_m \uppi_m=1,
$$
where
$$
(\hbar \Omega_{\mathrm{free}}+\Phi^\infty V) u^{(m)}=\lambda_m u^{(m)}.
$$
We are now reduced to compute 
\[
\mathrm{Tr}(\rho O)=\ds\sum_{m=0}^n \uppi_m u^{(m)*}Ou^{(m)}.\]

{\it Step 1: Spectral decomposition.}
We write
\[\hbar \Omega_{\mathrm{free}}+\Phi^\infty V
=\begin{pmatrix} 0 & \Phi^\infty v^\top
\\
\Phi^\infty v & \hbar \Omega\end{pmatrix} + \hbar \omega_0\mathbb I.\]
We consider the associated quadratic form 
\[\begin{array}{lll}
u=(u_0,\hat u)\in \mathbb C\times \mathbb C^n \longmapsto 
u^*\begin{pmatrix} 0 & \Phi^\infty v^\top
\\
\Phi^\infty v & \hbar \Omega\end{pmatrix} u
&=&
\hbar \hat u^*\Omega \hat u + 2\Phi^\infty \mathrm{Re}(u_0^* v^\top \hat u)
\\&=
&\hbar 
\left|\Omega^{1/2}\Big(\hat u +u_0\ds\frac{\Phi^\infty}{\hbar}\Omega^{-1}v\Big)\right|^2 
\\
&&\qquad-\ds\frac{\Phi^{\infty 2}}{\hbar}v^\top \Omega^{-1} v |u_0|^2.\end{array}\]
Setting $y=(\hat u +u_0\frac{\Phi^\infty}{\hbar}\Omega^{-1}v)$, it casts as
$\hbar y^*\Omega y -\frac{\Phi^{\infty}2}{\hbar}v^\top \Omega^{-1} v |u_0|^2$: the signature is $(n,1)$.
By Sylvester’s law of inertia, we deduce that $\hbar \Omega_{\mathrm{free}}+\Phi^\infty V$ has one eigenvalue below $\hbar \omega_0$ --- that will be denoted $\lambda_-$ from now on ---  and $n$ eigenvalues above $\hbar \omega_0$.

{\it Step 2: Computation of $u^*Ou$.}
Let $(\lambda,u)$ be an eigenpair of $\hbar \Omega_{\mathrm{free}}+\Phi^\infty V$:
\begin{equation}\label{evH}
\Phi^\infty v^\top \hat u=(\lambda - \hbar \omega_0)u_0,\qquad 
\Phi^\infty vu_0+ \hbar \Omega \hat u=(\lambda - \hbar \omega_0)\hat u.\end{equation}
The second equation yields 
$\Phi^\infty v^\top \Omega^{-1} v u_0+ \hbar v^\top \hat u=(\lambda - \hbar \omega_0)v^\top \Omega^{-1}\hat u
$.
We use these relations to compute
\[\begin{array}{lll}
u^*Ou&=&-2v^\top\Omega^{-1} v |u_0|^2
+ 2\mathrm{Re}\big[(v^\top \Omega^{-1} \hat u )(\hat u^*v )
\big]
 \\[.4cm]
 &=&
 -2v^\top\Omega^{-1} v |u_0|^2+
 2\mathrm{Re}\left[\ds\frac{\Phi^\infty v^\top \Omega^{-1} v u_0+ \hbar v^\top \hat u}{\lambda - \hbar \omega_0}
 \times \ds\frac{(\lambda - \hbar \omega_0)u_0^*}{\Phi^\infty}
 \right]
 \\[.4cm]
 &=&
 \ds\frac{2\hbar }{\Phi^\infty}\mathrm{Re}\left[v^\top \hat u
\, u_0^*
 \right]
=\ds\frac{2\hbar(\lambda - \hbar \omega_0) }{\Phi^{\infty 2}} |u_0|^2
. \end{array}\]
In particular, $u^*Ou$ has the sign of $\lambda - \hbar \omega_0$.

{\it Step 3: Estimation of $\mathrm{Tr}(\rho O)$.}
Let $u_-\neq 0$ stand for a normalized eigenvector associated to $\lambda_-$.
From \eqref{evH}, we infer that necessarily $u_{-0}\neq 0$
(otherwise we would have $(\hbar \Omega +(\hbar\omega_0-\lambda_-)\mathbb I)\hat u=0$
where  $\hbar \Omega +(\hbar\omega_0-\lambda_-)\mathbb I$ is invertible, and thus $\hat u=0$ as well), and 
by writting
$\hat u_-=\Phi^\infty u_0 \big((\lambda - \hbar \omega_0)\mathbb I- \hbar \Omega\big)^{-1}v$, we obtain 
$$
(\lambda_- - \hbar \omega_0)
= \Phi^{\infty 2}v^\top  \big((\lambda_- - \hbar \omega_0)\mathbb I- \hbar \Omega\big)^{-1}v.
$$
Consequently, we get
\[\begin{array}{lll}
0\leq -u_-^*Ou_-&=&\ds\frac{2\hbar(\hbar \omega_0-\lambda_-) }{\Phi^{\infty 2}} |u_{-0}|^2
=2\hbar  |u_{-0}|^2\,  v^\top  \big(\hbar \Omega +( \hbar \omega_0-\lambda_-)\mathbb I \big)^{-1}v
\\[.4cm]&=&
2\hbar  |u_{-0}|^2\ds\sum_{j=0}^{n-1}\ds\frac{V_{j+1/2}^2}{\hbar \omega_{j+1/2} +
\hbar \omega_0-\lambda_-}
\\[.4cm]&\leq&
2 |u|^2\ds\sum_{j=0}^{n-1}\ds\frac{V_{j+1/2}^2}{ \omega_{j+1/2}}=
 2  v^\top\Omega^{-1} v. 
\end{array}\]
It leads to the following estimate
\[
-\mathrm{Tr}(\rho O)=(-\uppi_- u_-^*Ou_-) -\ds\sum_{\uppi_m\neq p_-} \uppi_mu_m^*Ou_m
\leq (-\uppi_- u_-^*Ou_-)\leq 2  v^\top\Omega^{-1} v,\]
since the coefficients $\uppi_m$ belong to $[0,1]$, and for $\uppi_m\neq \uppi_-$, $u_m^*Ou_m\geq 0$.
Therefore, \eqref{eq_tr_rhoO} and \eqref{ZeroComm} 
cannot be realized together when \eqref{small2} holds, and we finally deduce  that $\Phi^\infty=0$.
\QED

 The  limit system 
reduces to $\dot \rho^\infty=-i\hbar [\Omega_{\mathrm{free}},\rho^\infty]$, thus exhibiting an oscillating behavior for $\rho^\infty_{j+1/2}=\lambda_{j+1/2}e^{-i \omega_{j+1/2}t}$,
with $j\in \{0,...,n-1\}$ and for $\rho^\infty_{j,k}=\lambda_{j,k}e^{-i \nu_{j,k}t}$
with $j,k\in \{1,...,n\}$, while the populations $\rho_j^\infty$ 
remain constant.
The conditions $(\Omega v)^\top p^\infty=0$ and 
 $v^\top q^\infty=-\frac{\Phi^\infty}{\kappa}=0$ imply  $\rho^\infty_{j+1/2}=0$.
 In turn, we deduce that $\psi(t+t^\nu,z)$ tends to 0 weakly in $C^0([0,T];L^2(\mathbb R^3))$, and 
 passing to the limit in the energy tells us that
$
\mathcal E^\infty=\sum_{j=0}^n \hbar \omega_j \rho_j^\infty\leq \mathcal E_{\mathrm{init}}.
$
As a consequence, we note that if $\mathcal E_{\mathrm{init}}<\hbar \omega_1$, 
then necessarily $\rho^\infty_0>0$.
At this stage, we can also show that $q^\Xi,p^\Xi$ converge to 0 
in the sense of 
Cesaro, by reproducing the proof of Lemma~\ref{lem:cesaro}.
We can perform another  step  in the case of pure states where we additionally know that  $\rho^\infty=u^\infty u^{\infty*}$,
with $u^\infty\in \mathbb C^{n+1}$, is rank-one.
In particular, the following relations hold
\[
|\rho_{j+1/2}^\infty|^2
=u^\infty_0 u^{\infty*}_{j+1}u^{\infty*}_0 u^{\infty}_{j+1}
=|u^\infty_0 |^2 |u^{\infty}_{j+1}|^2
\qquad j\in \{0,...,n-1\}.\]
Since  $\rho_{j+1/2}^\infty=0$ and $\rho^\infty_0= |u^\infty_0 |^2>0$, 
we deduce that $u^{\infty}_{j+1}=0$ for any $j\in \{0,...,n-1\}$; in turn, 
$
\rho_{j,k}^\infty=u^\infty_j u^{\infty*}_{k}$
equally vanish for any pair $j,k\in\{1,...,n\}$, $j\neq k$.
We summarize our findings as follows.

 \begin{proposition}\label{prop_cv_gen}
 Assume \eqref{small2}.
 Then the limit satisfies
 $q^\infty=0$, $p^\infty=0$.
 
 Moreover, in the pure case, assuming that the initial 
 initial energy is smaller than $\hbar\omega_1$,  
  the limit is completely characterized:
 $\rho^\infty=e_0e_0^*$.
\end{proposition}

In order to establish the exponential decay, we proceed as follows.
 On the one hand, we rewrite the 
 self-consistent potential as we did in \eqref{PhiS} and we get
 \[
 \Phi_S(t)=
 -\kappa v^\top q(t) - \ds\frac{\upgamma}{c}(\Omega v)^\top p(t)+\mathcal R(t) 
 \]
with 
\[\begin{array}{lll}
\mathcal R(t)
&=&-\left(\ds\int_t^\infty k_c(s)\ud s \right)v^\top q(t)
- \left(\ds\int_t^\infty s k_c(s)\ud s\right)(\Omega v)^\top p(t)
\\&&+\ds\int_0^t k_c(t-s)\left(\ds\int_s^t (\tau-s) (\Omega v)^\top \dot p(\tau)\ud \tau\right)\ud s.
\end{array}\]
On the other hand, we bear in mind that Proposition~\ref{prop_cv_gen} 
tells us, in the pure-state case,
that
$$\xi_{j+1/2}(t)=\rho_{j+1}(t)-\rho_0(t)\xrightarrow[t\to \infty]{} -1$$  
and  what follows restricts to pure solutions.

Proceeding this way, we 
arrive at following system for $Y=(q,p)$
\begin{equation}\label{Ysys}
\dot Y=\Big(\Theta-\ds\frac{2\upgamma}{\hbar c}D\Big)Y+\mathcal S
\end{equation}
with the matrices
\[
\Theta=\begin{pmatrix}0 & -\Omega\\
 \widetilde \Omega& 0\end{pmatrix},\qquad 
  \widetilde \Omega=\Omega-\ds\frac{2\kappa} \hbar vv^\top ,
  \qquad
  D=
  \begin{pmatrix}
    0&0\\
    0&v(\Omega v)^\top 
  \end{pmatrix}.
\]
The remainder term
$\mathcal S$ is given by
\[
\ds\frac{\Phi}{\hbar}\ds\sum_{k\neq j}V_{k+1/2}p^\Xi_{k+1,j+1}
\]
on the $q$-components
and 
the following sum 
\[
-\ds\frac{2\Phi}{\hbar}V_{j+1/2}(1+\xi_{j+1/2})
+\ds\frac{2}{\hbar}V_{j+1/2}
(\mathcal R+\Phi_I
)
+\ds\frac{\Phi}{\hbar}\ds\sum_{k\neq j}V_{k+1/2}q^\Xi_{k+1,j+1}
\]
on the $p$-components.
The non linear analysis and the treatment of the 
remainder terms can be performed exactly as for the 3 levels system in Section~\ref{sec:nonlin}.
We focus on the linear problem and the justification of the critical estimate analog of  Proposition~\ref{estim_expA}. 
in fact we will obtain sharp estimates, by using the framework developed in Section~\ref{sec:symmetrizer}.
Note that we recover exactly the case handled in Corollary~\ref{cor:symmetric_rank_one_symmetrizer}
(with $\alpha=\frac{2\kappa}{\hbar}$).

As remarked above, under the weak-coupling condition \eqref{small2}
the matrix \(\widetilde \Omega\) is symmetric positive definite.  Hence
\eqref{simple_symmetric_symmetrizer} applies.  Moreover, 
we are led to
\eqref{rank_one_D_general} with $(u,v)\to (v\Omega, v)$ and $M\to \Omega$;
so that 
the compatibility condition \eqref{rank_one_compatibility} holds with
\(\tau=1\). Therefore, the linear system obtained from \eqref{Ysys} when disregarding the remainder 
$\mathcal S$ satisfies the energy dissipation relation
\[
 S=\begin{pmatrix}\widetilde \Omega  &0\\0& \Omega\end{pmatrix}
,\qquad
  SY\cdot Y
  =
  q^\top 
       \Big(\Omega-\ds\frac{2\kappa}{\hbar } vv^\top \Big)q
   +
   p^\top \Omega p,
  \qquad
  \ds\frac{\ud}{\ud t}SY\cdot Y=-\ds\frac{4\upgamma} {\hbar c}|(\Omega v)^\top p|^2.
\]
The simplicity of this identity comes from the fact that the conservative
rank-one perturbation and the dissipative rank-one term involve the same
coupling vector \(v\).
With 
\[
b=\begin{pmatrix}0\\
 \Omega v\end{pmatrix},\qquad d=\begin{pmatrix}0\\
v\end{pmatrix}\]
we get $D=db^\top$, $Sd=b$, $SD+D^\top S=2bb^\top$.
Accordingly, the dissipation casts as $-\frac{4\upgamma}{\hbar c}|b^\top Y|^2$; this is the rank-one dissipation mechanism of the model. 
We should now modify the energy functional to see that the dissipation propagates to all components of $Y$, and to quantify it.
We seek a symmetric matrix $K$ such that
\[K\Theta+\Theta^\top K= SD+D^\top S-2G\]
where $G$ is positive definite.
Since $\Theta$ flips the variables $q/p$, we search for $K$ with the form
\[
K=\begin{pmatrix} 0 &H\\ H^\top & 0\end{pmatrix} 
\]
so that 
\[K\Theta+\Theta^\top K= 
\begin{pmatrix}
 H\widetilde \Omega +\widetilde \Omega H^\top 
 & 0 \\
 0 & -H^\top \Omega- \Omega H
 \end{pmatrix}\]
 defines, owing to the off-diagonal form of $K$, 
pure $q$ or $p$-quadratic terms.

We wish to apply Lemma~\ref{lem:PBH} 
with 
$\mathcal L=\Omega^{1/2}\widetilde \Omega\Omega^{1/2}$ and 
$w=\Omega^{1/2}v$.
Indeed, 
\(
 \Omega^{1/2}\widetilde \Omega\Omega^{1/2}
\)
is symmetric positive definite, and there exists an orthogonal matrix $U$
and positive numbers $\mu_1,...,\mu_n$ such that
\[
  \Omega^{1/2}\widetilde \Omega\Omega^{1/2}=UMU^{\top},
  \qquad
  M=\operatorname{diag}(\mu_0,\ldots,\mu_{n-1}).
\]
Set
\[
  \Pi=\Omega^{-1/2}U.
\]
Then
\[
 \Pi^\top  \widetilde \Omega\Omega \Pi=M,
  \qquad
  \Pi^{\top}\Omega \Pi=\mathbb I.
\]
The PBH condition for $(\Omega v,\widetilde \Omega\Omega)$  is equivalent to
\eqref{eq:PBH_Lw}. Indeed
$ \mathcal Lu=\mu u$ iff
$u'=\Omega^{-1/2}u$ satisfies $ \widetilde \Omega\Omega u'=\mu u'$, while 
$w^\top u=(\Omega^{1/2} v^\top \Omega^{1/2}u'=(\Omega v)^\top u'$.
The PBH criterion can now be verified by using the spectral analysis in Example~\ref{Rk1pert}:
  the eigenvalues of $\widetilde \Omega\Omega$ are obtained as zeroes of 
  \[\mu\longmapsto 1-\ds\frac{2\kappa}{\hbar} \ds\sum_{j=0}^{n-1} \ds\frac{\omega_{j+1/2} v_{j+1/2}^2}{
  \omega_{j+1/2}^2-\mu}.\]
  We find $n$ distinct eigenvalues satisfying the   interlacing property
    $$0<\mu_0<\omega_{1/2}^2<\mu_1<...<\mu_{n-1}<\omega_{n-1/2}^2,$$ 
    where positivity of \(\mu_0\) follows from the weak-coupling condition \eqref{small2}.
  From the eigenvectors $u'_0,...,u'_{n-1}$ of $\widetilde \Omega\Omega$, also constructed in Example~\ref{Rk1pert},
  we set $a_j=(\Omega v)^\top u'_j= w^\top u_j$, with $w=\Omega^{1/2}v$, $u_j=\Omega^{1/2}u'_j$.

Now, let $0<\epsilon<\min\big\{a_j^2,\ j\in\{0,...,n-1\}\big\}$.
We construct by hand a matrix $Z_\epsilon$ such that
\[
 Z_\epsilon+Z_\epsilon^\top =2(aa^\top -\epsilon\mathbb I),\qquad
  Z_\epsilon M+MZ_\epsilon^\top =2\mathrm{diag}(\mu_j(a^2_j-\epsilon)).\]
  We find 
  \[
  (Z_\epsilon)_{j,k}=\left\{
  \begin{array}{ll} a^2_j-\epsilon\quad & \textrm{ if $j=k$},
  \\[.4cm]
  \ds\frac{2\mu_j}{\mu_j-\mu_k}a_ja_k\quad & \textrm{ if $j\neq k$.}\end{array}\right.\]
  Now we set 
  \[
  K_\epsilon=-
  \begin{pmatrix}
  0 & \Pi Z_\epsilon \Pi^{-1}
  \\
  (\Pi^{-1})^\top Z_\epsilon^\top \Pi^\top & 0
  \end{pmatrix}.\]
  It is symmetric by construction and we get
  \[
  K_\epsilon\Theta+\Theta^\top K_\epsilon=SD+D^\top S
  -
2
    \underbrace{\begin{pmatrix}
  \Pi\mathrm{diag}(\mu_j(a_j^2-\epsilon)) \Pi^\top  & 0
    \\
  0 & \epsilon\Omega
  \end{pmatrix}}_{:=G_\epsilon}\]
  where $G_\epsilon$ is indeed symmetric definite positive.
  
  We finally consider the modified energy
  \[\mathcal E_\epsilon(Y)=SY\cdot  Y + \ds\frac{\upgamma}{\hbar c}K_\epsilon Y \cdot Y.\]
  We obtain 
  \begin{equation}\label{derEnMod}\ds\frac{\ud}{\ud t} \mathcal E_\epsilon(Y)=
  -\ds\frac{2\upgamma}{\hbar c}
  \Big(|(\Omega v)^\top p|^2
  + 
  \underbrace{q^\top \Pi\mathrm{diag}(\mu_j(a_j^2-\epsilon)) \Pi^\top q
  +\epsilon p^\top \Omega p}_{=Y^\top G_\epsilon Y} + \ds\frac{2\upgamma}{\hbar c} (\Omega v)^\top p d^\top K_\epsilon Y
  \Big).\end{equation}
  The first three terms are coercive of order \(\upgamma/c\); the last term is of order \(\upgamma^2/c^2\) and is absorbed for \(\upgamma/c\) sufficiently small, by using a standard Young argument.
 Besides, we have to check that  $\mathcal E_\epsilon$ defines a norm, equivalent to the euclidian norm. This is where a smallness condition on $1/c$ comes up.

For a symmetric definite positive matrix $\Sigma$, we denote by $\lambda_{\min}(\Sigma)$, $\lambda_{\max}(\Sigma)$ its extremal eigenvalues.
Since $ |Y^{\top}K_\epsilon Y| \leq \|K_\epsilon\| |Y|^2$, we get
$$ \left(\lambda_{\min}(S)-\ds\frac{ \upgamma}{\hbar c}\|K_\epsilon\|\right)|Y|^2 \leq \mathcal E_{\epsilon}(Y) \leq \left(\lambda_{\max}(S)+\frac { \upgamma}{\hbar c}\|K_\epsilon\|\right)|Y|^2. $$
In particular, if
$ \frac{ \upgamma}{\hbar c}\|K_\epsilon\|\leq \frac{\lambda_{\min}(S)}{2},$
then
$$  \frac{\lambda_{\min}(S)}{2}|Y|^2 \leq \mathcal E_{\epsilon}(Y) \leq \frac{3\lambda_{\max}(S)}{2}|Y|^2.$$
Next, the last term in \eqref{derEnMod} is dominated by 
\[
 \ds\frac{|(\Omega v)^\top p|^2}{2}+  \ds\frac{2\upgamma^2}{\hbar^2 c^2} |K_\epsilon d|^2 | Y|^2.
\]
Hence, provided $$\ds\frac{2\upgamma^2}{\hbar^2 c^2} |K_\epsilon d|^2<\frac{\lambda_{\mathrm{min}}(G_\epsilon)}{2},$$
we deduce that 
\[\ds\frac{\ud}{\ud t} \mathcal E_\epsilon(Y)\leq -\ds\frac{\upgamma \lambda_{\mathrm{min}}(G_\epsilon)}{\hbar c}|Y|^2
\leq-\ds\frac{\upgamma}{\hbar c}\ds\frac{\lambda_{\mathrm{min}}(G_\epsilon)}{\lambda_{\max}(S)+\upgamma\|K_\epsilon\|/(\hbar c)} \mathcal E_\epsilon(Y)
\leq
-\ds\frac{\upgamma}{\hbar c} \ds\frac{2\lambda_{\mathrm{min}}(G_\epsilon)}{3\lambda_{\max}(S)} \mathcal E_\epsilon(Y)
.\]
 With Gr\"onwall's lemma, we arrive at 
\[
|Y(t)|\leq \sqrt{\ds\frac{3\lambda_{\max}(S)}{\lambda_{\min}(S)}}\exp\left(-\ds\frac{\upgamma}{\hbar c}
\ds\frac{\lambda_{\mathrm{min}}(G_\epsilon)}{3\lambda_{\max}(S)}t
\right)
|Y(0)|.\]

We can provide explicit estimates on \(\lambda_{\min}(S),\lambda_{\max}(S)\).
The construction \eqref{simple_symmetric_symmetrizer} yields,  
since $\widetilde\Omega\leq \Omega$,  $\lambda_{\mathrm{max}}(S)= \omega_{n+1/2}$
and $\lambda_{\mathrm{min}}(S)=\lambda_{\mathrm{min}}(\widetilde\Omega)\leq \omega_{1/2}$ can be determined 
reasoning as in 
 Example~\ref{Rk1pert}.
A rough estimate is given by using the Cauchy-Schwarz inequality
$
(v^\top q)^2\leq (v^\top \Omega^{-1}v)(q^\top \Omega q)$; it leads to 
$q^\top (\Omega-\frac\kappa\hbar vv^\top)q\geq (1- \frac\kappa\hbar v^\top\Omega^{-1}v)
q^\top\Omega q\geq (1- \frac\kappa\hbar v^\top\Omega^{-1}v) \omega_{1/2}|q|^2$.
Therefore, we get 
$\lambda_{\mathrm{min}}(S)\geq  \omega_{1/2}(1- \frac{2\kappa}\hbar v^\top\Omega^{-1}v)$, which is indeed positive by virtue of \eqref{small2}.

Let us derive further rough but explicit estimates. 
With the specific form of $K_\epsilon$, we have $\|K_\epsilon\|=\|\Pi Z_\epsilon \Pi^{-1}\|$ and
$|K_\epsilon d|=|\Pi Z_\epsilon a|$.
From $\Pi\Pi^\top=\Omega{-1}$, we infer 
$\|\Pi\|=\frac{1}{\sqrt{\omega_{1/2}}}$ and $\|\Pi^{-1}\|=\sqrt{\omega_{n-1/2}}$, and hence
$\|\Pi Z_\epsilon \Pi^{-1}\|\leq \sqrt{\frac{\omega_{n-1/2}}{\omega_{1/2}}\|Z_\epsilon}\|$,
$|\Pi Z_\epsilon a|\leq \frac{|Z_\epsilon a|}{\sqrt{\omega_{1/2}}}$.
Besides, we get 
$\lambda_{\mathrm{min}}(\Pi \mathrm{diag}(\mu_j(a_j^2-\epsilon)) \Pi^\top)\geq \frac{\min_j \mu_j(a_j^2-\epsilon)}{\omega_{n-1/2}}$.
Choosing $\epsilon=\frac{\min_j a_j^2}{2}$ yields a more explicit decay rate estimate.
  \QED

\section{Comments}
\label{sec:comm}

In this Section we comment on two obstructions of the previous results. Firstly,  a distinction appeared between pure states and mixed states; in the latter situations there remain persistent coherences, that only vanish ``in average''. Secondly, 
the analysis deals with the very specific ``V-structure' where the ground states in directly coupled to all  other excited states, but there is no direct excited-to-excited connections. One may wonder whether or not the damping
 holds for more general structures.

\subsection{Universality of the V-structure and  specificities of the pure case}
\label{subsec:pure_linearization}

\noindent
{\bf Linearization at a pure stationary state.}

\noindent
We keep
$
  H_{\rm free}=\hbar\Omega_{\rm free}$, with $\Omega_{\rm free}=\operatorname{diag}(\omega_0,\ldots,\omega_n),$
  satisfying \eqref{omega} 
and \(V\) is a real symmetric matrix with vanishing diagonal coefficients.  Let 
\[
  \rho_{\mathrm{min}}=e_0e_0^*
\]
stand for  the pure ground state.  Since
\(\operatorname{Tr}(\rho_{\mathrm{min}}V)=V_{0,0}=0\), the pair $
  \rho= \rho_{\mathrm{min}},
  \psi=0$
is a stationary solution of \eqref{Liouv}--\eqref{wave}.
Consider a perturbation of this state of the form
\[
  \rho^\varepsilon= \rho_{\mathrm{min}}+\varepsilon r+...,
  \qquad
  \psi^\varepsilon=\varepsilon\tilde \psi+...,
  \qquad
  \Phi^\varepsilon=\varepsilon\phi+....
\]
with remainders assumed to be negligible with respect to $\varepsilon$. 
At leading order, the coupled system becomes
\begin{equation}
\label{linearized_general_system}
\begin{array}{l}
  i\hbar\dot r
    =\hbar[\Omega_{\rm free},r]+\phi[V,\rho_{\mathrm{min}}],
    \\
  (\partial_t^2-c^2\Delta)\tilde \psi
    =-c^2\sigma\operatorname{Tr}(rV),
    \\
    \phi(t)=\ds\int_{\mathbb R^3}\sigma(z)\tilde \psi(t,z)\,\mathrm dz.
\end{array}
\end{equation}
The commutator with \(\rho_{\mathrm{min}}\) has a very specific form:
\[
  [V,\rho_{\mathrm{min}}]
  =\ds\sum_{j=1}^n
  V_{0,j}\left(e_je_0^*-e_0e_j^*\right).
\]
In particular, 
\(
  [V,\rho_{\mathrm{min}}]_{j,k}=0\)
whenever \(j,k\geq 1\).
Consequently, the entries of \(r\) satisfy
\[
\begin{array}{ll}
\dot r_{0,j}
    =i(\omega_j-\omega_0)r_{0,j}+\ds\frac{\phi}{\hbar} V_{0,j},
    &j\in\{1,\ldots,n\},
   \\
\dot r_{j,k}
    =i(\omega_k-\omega_j)r_{j,k},
    &j,k\in\{1,\ldots,n\},
   \\
  \dot r_{j,j}=0,
    &j\in\{0,\ldots,n\}.
\end{array}
\]
Notice that the equations for the excited-to-excited coherences is independent of the
potential \(\phi\):  these coherences may contribute to
the source \(\operatorname{Tr}(rV)\), and force the wave and the
ground-to-excited coherences, but there is no feedback on their own evolution.

We next identify the part of \eqref{linearized_general_system} which is
compatible with the constraint that the state remain pure.
Let
\[
  \mathscr M=\{uu^*,\ u\in\mathbb C^{n+1},\ u^*u=1\}
\]
be the manifold of pure density matrices.  We claim that its tangent space at \(\rho_{\mathrm{min}}\) is
\[
  \mathscr T_{\rho_{\mathrm{min}}}\mathscr M
  =\{we_0^*+e_0w^*,\ w\in\mathbb C^{n+1},\ w^*e_0=0\}.
\]
In particular, if \(r\in  \mathscr T_{\rho_{\mathrm{min}}}\mathscr M\), then
\[
  r_{j,k}=0\quad\text{for }j,k\geq1,
  \qquad
  r_{j,j}=0\quad\text{for every }j.
\]
Indeed, let  \(\varepsilon\mapsto u(\varepsilon)\) be a differentiable curve of unit
vectors such that \(u(0)=e_0\), and write
\[
  \dot u(0)=\alpha e_0+w,
  \qquad w^*e_0=0.
\]
Differentiating \(u(\varepsilon)^*u(\varepsilon)=1\) at
\(\varepsilon=0\) gives \(\operatorname{Re}\alpha=0\).  The purely imaginary
component \(\alpha e_0\) corresponds only to a change of phase and disappears
when differentiating \(u(\varepsilon)u(\varepsilon)^*\).  Hence
\[
  \left.\frac{\mathrm d}{\mathrm d\varepsilon}
  \big(u(\varepsilon)u(\varepsilon) ^*\big)
  \right|_{\varepsilon=0}
  =we_0^*+e_0w^*.
\]
Conversely, for every \(w\perp e_0\), the curve
\[
  \varepsilon\longmapsto u(\varepsilon)
  =\frac{e_0+\varepsilon w}{|e_0+\varepsilon w|}
\]
produces the tangent vector \(we_0^*+e_0w^*\).

We can now formulate the universality property of the V-diagram.
Setting
\(
  z_{j+1/2}=r_{0,j+1}\),  \(
  V_{j+1/2}=V_{0,j+1}\),   and,
\(
  \omega_{j+1/2}=\omega_{j+1}-\omega_0\),
 for $ j\in\{0,...,n-1\},\
$, one obtains
\[\begin{array}{l}
 \dot z_{j+1/2}=-i\omega_{j+1/2} z_{j+1/2}-i\ds\frac{\phi}{\hbar} V_{j+1/2},
   \\[.4cm]
  (\partial_t^2-c^2\Delta)\varphi
    =-2c^2\sigma
      \operatorname{Re}\left(\ds\sum_{j=0}^{n-1} V_{j+1/2}z_{j+1/2}\right).
\end{array}
\]
The two-way coupling between 
the scalar field and
the \(n\) ground-to-excited coherences is
driven  through the
single vector
\(
 v=(V_{1/2},...,V_{n-1/2})
\).
Hence, the restriction of the linearized system
\eqref{linearized_general_system} to  $\mathscr T_{\rho_{\mathrm{min}}}\mathscr M$ is a generalized
V-system, independently of the transition diagram encoded by \(V\).

The situation is different if perturbations are not restricted to the
pure-state manifold.  Let
\[
  Q=I-\rho_{\mathrm{min}}
\]
be the orthogonal projection onto the excited subspace.  
Since $\Omega_{\mathrm{free}}e_0=\omega_0 e_0$, 
we have $\Omega_{\mathrm{free}}\rho_{\mathrm{min}}=\omega_0\rho_{\mathrm{min}}=
\rho_{\mathrm{min}}\Omega_{\mathrm{free}}$ and $Q\Omega_{\mathrm{free}}\rho_{\mathrm{min}}
=0=\rho_{\mathrm{min}}\Omega_{\mathrm{free}}Q$.
Consequently, 
for any matrix $B$, we get 
$Q[\Omega_{\mathrm{free}},B]Q=[Q\Omega_{\mathrm{free}}Q,QBQ]$.
Projecting the first equations in
\eqref{linearized_general_system} on both sides by \(Q\), and using
\(Q[V,\rho_{\mathrm{min}}]Q=0\), gives
\[
  i\frac{\mathrm d}{\mathrm dt}(QrQ)
  =[Q\Omega_{\mathrm{free}} Q,QrQ],
\]
and thus 
\[
  Qr(t)Q
  =e^{-itQ\Omega_{\mathrm{free}}  Q}\,Qr(0)Q
  e^{itQ\Omega_{\mathrm{free}}  Q}.
 \]
  Since multiplication on the left and on the
right by unitary matrices preserves the euclidean norm, we have
\(
  \|Qr(t)Q\|
  =\|Qr(0)Q\|
\)
for all \(t\geq0\): the full linearized density matrix cannot
converge to zero unless \(Qr(0)Q=0\), 
which happens at first order for a
pure perturbation. 

\subsection{The ladder model as an example of  failure of dissipative compatibility}
\label{sec:ladder_obstruction}

The ladder model is a 
multi-level generalization of two-level Bloch equations where only nearest-neighbor coherences and transitions are retained.
The model arises as a simple description of  atoms or spin chains. 
We thus consider $n+1$ energy levels, labeled from 0 to $n$ and the system is driven by  transition only between
 the levels $j$ and $j+1$, 
 described by the coupling constant $V_{j+1/2}>0$.
 Therefore, the hamiltonian is defined with
 \[
 V= \left(
     \raisebox{0.5\depth}{%
       \xymatrixcolsep{1ex}%
       \xymatrixrowsep{1ex}%
       \xymatrix{
         0  \ar @{.}[rrrrdddd] & V_{1/2}  \ar @{-}[rrrddd] &0 \ar @{.}[rr] \ar @{.}[rrdd]&& 0\ar @{.}[dd] 
         \\V_{1/2} \ar @{-}[rrrddd]  & &&&
           \\          0  \ar @{.}[dd]\ar @{.}[ddrr]&&&& 0 
         \\
          &&  & &   V_{n-1/2}
         \\
         0 \ar @{.}[rr]&&0&V_{n-1/2}& 0
       }%
     }
   \right)=
   \ds\sum_{j=0}^{n-1}V_{j+1/2}
  \bigl(e_je_{j+1}^*+e_{j+1}e_j^*\bigr).
\]
It provides a useful counterpoint to the generalized
$V$-system:  the existence of directions which are not seen by
the interaction with the field leads to  a simple example where the
conservative part can be symmetrized, but the symmetrization does not define
a non-negative dissipation in the corresponding metric.

The pure ground state $\rho_{\mathrm{min}}=e_0e_0^*$, together with a vanishing wave
field, is a stationary solution.  
Writing the equations componentwise make the quantities 
\(
  \xi_j=\rho_{j+1,j+1}-\rho_{jj},
\)
appear; at equilibrium the corresponding vector reduces to 
\(
(-1,0,...,0)=-e_0.
\)

The failure of attraction to $\rho_{\mathrm{min}}$ can be seen with a simple explicit solution:
let $n\geq2$, pick $k\in\{2,...,n\}$ and let
  $a,b\in\mathbb C$ satisfy $|a|^2+|b|^2=1$.  Then, set
  \begin{equation}
    u(t)=a e^{-i\omega_0t}e_0+b e^{-i\omega_kt}e_k,
    \qquad
    \rho(t)=u(t)u(t)^*,
    \qquad
    \psi(t)=\partial_t\psi(t)=0.
    \label{eq:ladder_dark_pure_solution}
  \end{equation}
  We check that 
  \(
    \operatorname{Tr}(\rho(t)V)
    = (Vu)^*(t)u(t)=0.
  \) since the vertices $0$ and $k$ 
  are not adjacent in the ladder graph. Hence,
  \eqref{eq:ladder_dark_pure_solution} defines an exact solution of the coupled system.
    If $b\neq0$, this solution
  does not converge to $\rho_{\mathrm{min}}$ ast $t\to \infty$.
 Even the ground-to-excited coherences
  \[
    \rho_{0,k}(t)
    =a b^*\,e^{i(\omega_k-\omega_0)t},
  \]
  have a  modulus that does not decay when $ab\neq0$.  Taking $|b|$ arbitrarily
  small produces non-convergent pure solutions arbitrarily close to $\rho_{\mathrm{min}}$.
\\

The linearized  equation around $\rho_{\mathrm{min}}$ reads
\begin{equation}
  i\hbar\dot r
  =\hbar[\Omega_{\rm {free}},r]+\phi[V,\rho_{\mathrm{min}}],
  \label{eq:ladder_linearized_liouville}
\end{equation}
where $[V,\rho_{\mathrm{min}}]=V_{1/2}(e_1e_0^*-e_0e_1^*)$.
Consequently, we have
\[
  \dot r_{0,1}=i(\omega_1-\omega_0)r_{0,1}+i\frac{V_{1/2}}{\hbar}\phi,
  \qquad
  \dot r_{0,k}=i(\omega_k-\omega_0)r_{0,k},
  \quad \textrm{ for $k\in\{2,...,n\}$}.
  \]
Even,  the pure-state tangent space identified above already contains undamped directions.
There are additional undamped directions when the perturbation is not
restricted to the pure-state manifold.  Indeed, define
\[
  2r_{j,j+1}=q_{j+1/2}+ip_{j+1/2},\qquad 
  \nu_{j+1/2}=\omega_{j+1}-\omega_j.
  \qquad j\in\{0,...,n-1\}.
\]
These neighbouring coherences occur in the source:
\(
  \operatorname{Tr}(rV)=\sum_{j=0}^{n-1}V_{j+1/2}q_{j+1/2}.
  \)
Equation~\eqref{eq:ladder_linearized_liouville} gives
\[\begin{array}{ll}
  \dot q_{1/2}=-\nu_{1/2}p_{1/2},
  \qquad & 
  \dot p_{1/2}=\nu_{1/2}q_{1/2}+\ds\frac{2V_{1/2}}{\hbar}\phi,
 \\[.4cm]
  \dot q_{j+1/2}=-\nu_{j+1/2}p_{j+1/2},
  \qquad& 
  \dot p_{j+1/2}=\nu_{j+1/2} q_{j+1/2},
  \qquad j\in\{1,...,n-1\}.
\end{array}\]
In particular, we have the conservation property
\[
  q_{j+1/2}(t)^2+p_{j+1/2}(t)^2=q_{j+1/2}(0)^2+p_{j+1/2}(0)^2,
  \qquad \textrm{ for any $j\in\{1,...,n-1\}$}.
  \]
  The upper transitions can therefore force the wave and, through it, the
first transition, but there is no feedback on the upper transitions.  This
one-way coupling can be related to a  block-triangular matrix in the ODE system.
We relate these obstructions to the framework introduced in Section~\ref{sec:symmetrizer}.
Retaining the first two terms in the expansion of the memory kernel and
disregarding the remainder gives
\[
  \phi=-\kappa v^\top q-\frac{\gamma}{c}(\Omega_{\mathrm{lad}} v)^\top p,
  \qquad
  \Omega_{\mathrm{lad}}=\operatorname{diag}(\nu_{1/2},\ldots,\nu_{n-1/2}).
  \label{eq:ladder_local_field_approximation}
\]
We collect the ground-to-excited coherences in a single vector $Z=(z_{1/2},...,z_{n-1/2})$ with 
$z_{j+1/2}=
(q_{j+1/2},p_{j+1/2})$.
We have
\[
  \frac{\mathrm d}{\mathrm dt}Z
  =AZ,\qquad A=
  \begin{pmatrix}
    A_{00}&A_{01}&\cdots&A_{0,n-1}\\
    0&J_1&&0\\
    \vdots&&\ddots&\\
    0&0&&J_{n-1}
  \end{pmatrix},
  \]
with
\[
  \begin{aligned}
  A_{00}
    &=\begin{pmatrix}
       0&-\nu_0\\
       \nu_{1/2}-\frac{2\kappa V_{1/2}}{\hbar} V_{1/2}&-\frac{2\upgamma V_{1/2}}{\hbar c} \nu_{1/2}V_{1/2}
     \end{pmatrix},
  &
  A_{0j}
    &=\begin{pmatrix}
       0&0\\
       -\frac{2\kappa V_{1/2}}{\hbar} V_{j+1/2}&-\frac{2\upgamma V_{1/2}}{\hbar c} \nu_{j+1/2}V_{j+1/2}
     \end{pmatrix},\\[1ex]
  J_j
    &=\begin{pmatrix}0&-\nu_{j+1/2}\\ \nu_{j+1/2}&0\end{pmatrix},
  & &\hspace{-2em}j\in\{1,...,n-1\}.
  \end{aligned}
  \]
  It follows that
\[
  \det(\lambda I-A)
  =\det(\lambda I-A_{00})
   \prod_{j=1}^{n-1}(\lambda^2+\nu_{j+1/2}^2),
 \]
which implies
\begin{equation}
  \big\{\pm i\nu_{j+1/2},\, j\in \{1,...,n-1\}\big\}\subset\sigma(A).
  \label{eq:ladder_imaginary_spectrum}
\end{equation}
Consequently,  the reduced linearized system does  not verify an exponential decay estimate (as soon as
there are at least three levels).

Reorganizing the variables as $Y=(q,p)$, the
linearized system has the form
\[
  \begin{array}{l}
  \dot Y=(\Theta-gD)Y,
  \,
  \Theta=\begin{pmatrix}0&-\Omega_{\mathrm{lad}}\\ \widetilde\Omega_{\mathrm{lad}}&0\end{pmatrix},\, \widetilde \Omega_{\mathrm{lad}}=
  \Omega_{\mathrm{lad}}-\frac{2\kappa V_{1/2}}{\hbar} e_0v^\top,\,
  D=\begin{pmatrix}0&0\\0&e_0(\Omega_{\mathrm{lad}} v)^\top\end{pmatrix}.
  \end{array}
 \]
Under the positivity and nonresonance assumptions stated in Example~\ref{rem:one_row_rank_one_symmetrizer}, the matrix $\widetilde \Omega_{\mathrm{lad}}\Omega_{\mathrm{lad}}$ has
positive real eigenvalues.  That construction therefore provides a positive
definite conservative symmetrizer
\[
  S=\begin{pmatrix}L&0\\0&M\end{pmatrix},
  \qquad
  S\Theta+\Theta^\top S=0.
\]
This part does not involve $D$; the
obstruction occurs with the compatibility with the matrix $D$.

\begin{lemma}
  \label{prop:ladder_failure_compatibility}
  Assume that $n\geq2$, and that $V_{j+1/2}>0$, 
  $\nu_{j+1/2}\neq\nu_{1/2}$ for every $j\in\{1,...,n-1\}$.  Then, there is no symmetric positive
  definite matrix $S$ such that
  \begin{equation}
    S\Theta+\Theta^\top S=0,
    \qquad
    SD+D^\top S\geq0.
    \label{eq:ladder_incompatible_requirements}
  \end{equation}
\end{lemma}

\noindent{\bf Proof.}
  Write a general symmetric matrix in block form as
  \[
    S=\begin{pmatrix}L&O\\O^\top&M\end{pmatrix}.
  \]
  The lower-left block of the first identity in
  \eqref{eq:ladder_incompatible_requirements} gives
  \(
    M\widetilde \Omega_{\mathrm{lad}}=\Omega_{\mathrm{lad}}\ L.
  \)
  Since $L$ is symmetric, it follows that
  \begin{equation}
    M\widetilde \Omega_{\mathrm{lad}}\Omega_{\mathrm{lad}}=\Omega_{\mathrm{lad}}
     \widetilde \Omega_{\mathrm{lad}}^\top M.
    \label{eq:ladder_conservative_symmetry_identity}
  \end{equation}
  We remark that  the upper-left block of $SD+D^\top S$ vanishes.
 Then,  non-negativity of $SD+D^\top S$ therefore forces the off-diagonal block to vanish and,
  by the rank-one criterion of Lemma~\ref{lem:SD},
  \begin{equation}
    Oe_0=0,
    \qquad
    Me_0=\tau\Omega_{\mathrm{lad}} v
    \quad\text{for some }\tau\geq0.
    \label{eq:ladder_required_alignment}
  \end{equation}
  Since $M$ is positive definite, necessarily $\tau>0$.

  The $(j,0)$ component of
  \eqref{eq:ladder_conservative_symmetry_identity} reads
  \[
    (\nu_{1/2}^2-\nu_{j+1/2}^2)M_{j0}
    =\ds\frac{2\kappa V_{1/2}}{\hbar} \bigl((Me_0)_j\nu_{1/2}V_{1/2}-\nu_{j+1/2}V_{j+1/2}M_{00}\bigr),
    \qquad j\in \{1,...,n-1\}.
  \]
 The right-hand side
  vanishes,
  since \eqref{eq:ladder_required_alignment} yields \(M_{j0}=(Me_0)_j
=\tau\nu_{j+1/2}V_{j+1/2}\),
   whereas the left-hand side becomes
  \[
    \tau(\nu_{1/2}^2-\nu_{j+1/2}^2)\nu_{j+1/2}V_{j+1/2}\neq0,
  \]
 a contradiction.
\QED

For the ladder system,
the non zero block is $e_0(\Omega_{\mathrm{lad}} v)^\top$, and
Proposition~\ref{prop:ladder_failure_compatibility} shows that the required
alignment $Me_0\parallel\Omega_{\mathrm{lad}} v$ is impossible.
This also explains why the algebraic condition
\[
  \Theta Y=\lambda Y,
  \qquad
  DY=0
  \quad\Longrightarrow\quad
  Y=0
\]
(which can be readily checked) cannot be used  as a stand-alone stability criterion;  
non-negativity dissipation $SD+D^\top S\geq 0$ should also be verified.  That preliminary structural
property fails for the ladder system, consistently with the undamped modes
displayed in \eqref{eq:ladder_imaginary_spectrum}.

\subsection{Numerical illustrations}
\label{sec:V_numerical_study}

The purpose of this section is to challenge numerically the results obtained for the three-levels V-system. We pay a specific attention to verify  the 
$\frac1c$ decay law, to investigate the behavior of the system for reduced values of the wave-speed, and the effect
of violating the weak-coupling hypothesis; we will also explore the fate of the excited-state
coherence in a mixed state.  

The wave equation is treated in radial symmetry: up to a slight redefinition of the source $\sigma$, 
we are led to solve the one-dimensional wave equation \cite{Vi4}.
The wave equation is thus solved on a computational domain $[0,R]$.
Since the mechanisms of energy evacuation at infinity are critical, we have to treat carefully the boundary condition at the right-end of the domain $R\gg R_*$.
Fortunately, the one dimensional geometry allows us to use the exact outgoing condition $\partial_t \psi+c\partial_r \psi\big|_{r=R}=0$
 \cite{EnMa}.
At the end of the day, with standard methods, we readily obtain a second-order accurate scheme for the wave equation.
For the source, we use the radial $C^2$ Wendland function
\[
 \sigma(r)=C\left(1-\frac r{R_*}\right)_+^4
 \left(1+4\frac r{R_*}\right),
\]
with $R_*<R$, 
normalized by $\|\sigma\|_{L^2(\mathbb R^3)}=1$. This function is compactly supported 
and its Fourier transform does not vanish \cite{Wend}.

 We can mix  up the resolution of the wave equation with a standard ODE solver for the
 components of the density matrix, with a construction inspired from Verlet's scheme in order to preserve the Hamiltonian structure of the problem \cite{Vi4}.
 However, this approach does not preserve the spectrum of the density matrix, and can produce slightly negative values of the populations: the scheme does not conserve the spectrum of the density matrix. 
 It can be completed by a mere projection step to restore these properties.
 An efficient alternative consists in updating the density matrix by 
 \[
 \rho_{n+1}=U_n\rho_nU_n^*, \qquad
 U_n=\exp\left[-\frac{i\Delta t}{\hbar}
 \bigl(\hbar \Omega_{\mathrm{free}}+\Phi_{n+1/2}V\bigr)\right],
\]
 where the potential is defined from the wave-resolution step.
 As far as we are concerned by the three level system, 
 the exponential can be evaluated at reasonable cost by computing the spectral decomposition of the matrix $\bigl(H_0+\Phi_{n+1/2}V\bigr)$.
 The method has been tested against mesh and time step refinements, and domain variations.  
 Trace, purity  and spectrum are robustly preserved and we do not observe a drift of the spectral invariants.
 The parameters for the simulations are collected in Table~\ref{tab1}.
 We numerically obtain 
\[
 \int_{\mathbb R^3}\sigma=3.60876486,
 \qquad
 \kappa=0.89067140,
 \qquad
 \gamma
 =1.03635204.
\]
 The displayed simulations have been realized with $\Delta r=.04$, $\Delta t=\min(10^{-2},\frac{.75\Delta r}{c})$.
 The initial data for the wave equation is always set to $0$.

\begin{table}[htbp!]
\begin{center}
\begin{tabular}{|c|c|c|c|c|c|c|} 
\hline 
 $\hbar$  & $\omega_{1/2}$ & $\omega_{3/2}$ & $V_{1/2}$ & $V_{3/2}$
& $R$ & $R_\sigma$ 
\\
 \hline 
 $1 $ & $1
 $  & $1.65
 $ & $.45
 $ & $.36
 $  & 14& $2.5$\\  
\hline 

\end{tabular}
\caption{Parameters of the model}
\label{tab1}
\end{center}
\end{table}

 \noindent
 {\bf Pure state, the $\frac1c$ decay rate.}
 We start from 
 \[\rho(0)=u_0u_0^*,\qquad
 u_0=(\sqrt{.78}, \sqrt{.15}e^{.2 i}, \sqrt{.07}e^{-.6i}).\]
 The initial energy is below  the first excited energy level.
 Results are displayed in Figure~\ref{fig:V_num_c_sweep} and 
 numerical rates are collected in Table~\ref{tab:V_num_c_sweep}.
 For $c\geq 1.5$, we obtain very similar decay curves $\frac tc\mapsto 
 \|\rho(t)-e_0e_0^*\|$, with a slope  coherent with \(
 \gamma_{\rm num}(c)\sim\frac{0.1355}{c}
\).
  For smaller $c$'s, we observe different slopes or the decay disappears on the observed time scale.
  For instance, 
  the computation for $c=0.35$ is still far from the ground state at $t=3000$.
  The $e^{-\gamma t/c}$ is a predicted asymptotic behavior, and it is not possible to decide  whether we are observing a very long transient 
  state for small $c$'s or if the exponential decay fails. We will go back to this issue elsewhere.
  We also compare the behavior of the fully coupled system
  with the behavior  when the potential is directly defined by 
  the reduced formula $\Phi_{\mathrm {red}}(t)=-\kappa v^\top q(t)-\frac{\gamma}{c}(\Omega v)^\top p(t)$.
  The relative discrepancy decreases from about six percent at $c=2$ to less
than one percent at $c=6$, supporting the analysis.\\

 \begin{figure}[htbp]
 \centering
 \includegraphics[width=5.5cm]{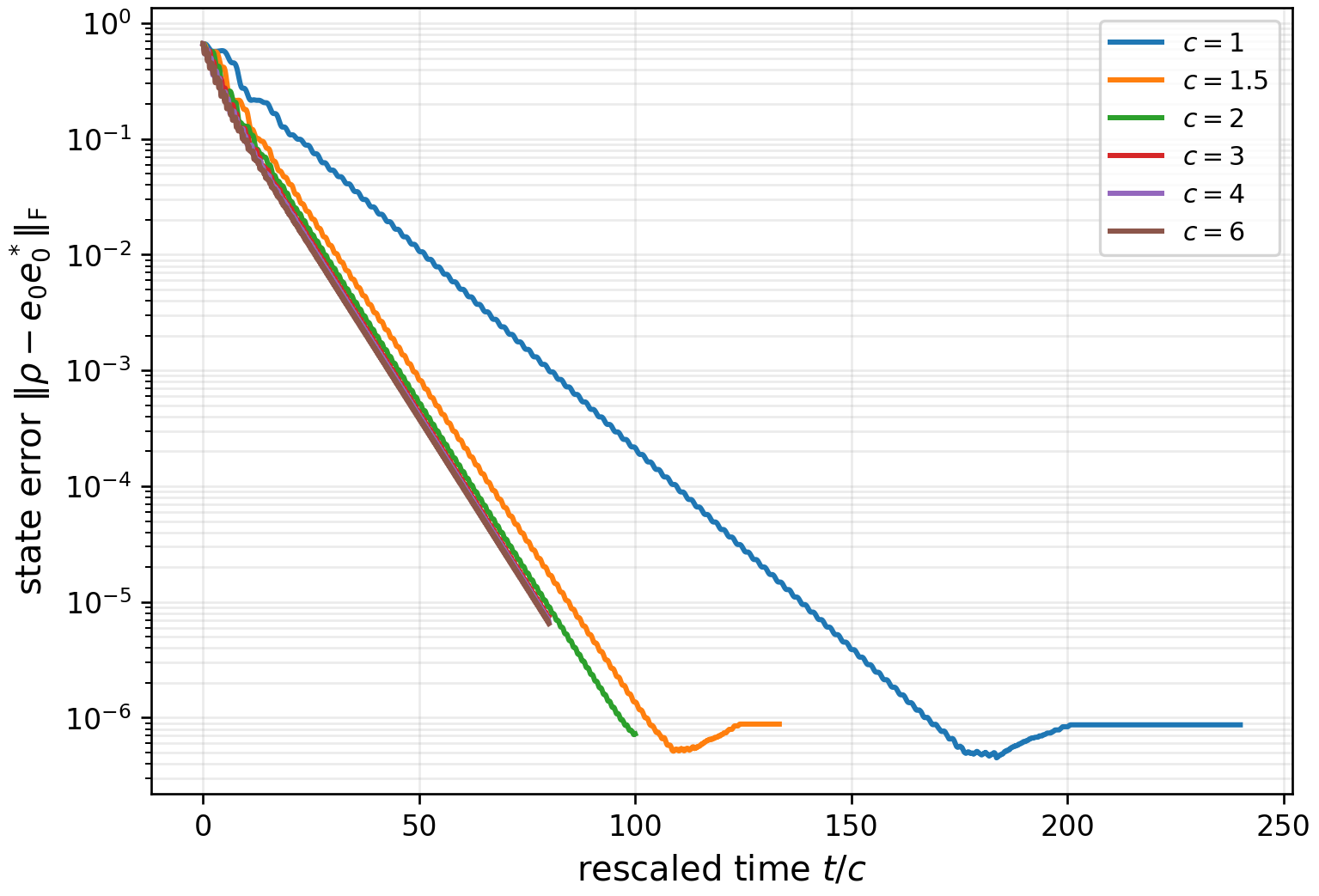}~\includegraphics[width=5.5cm]{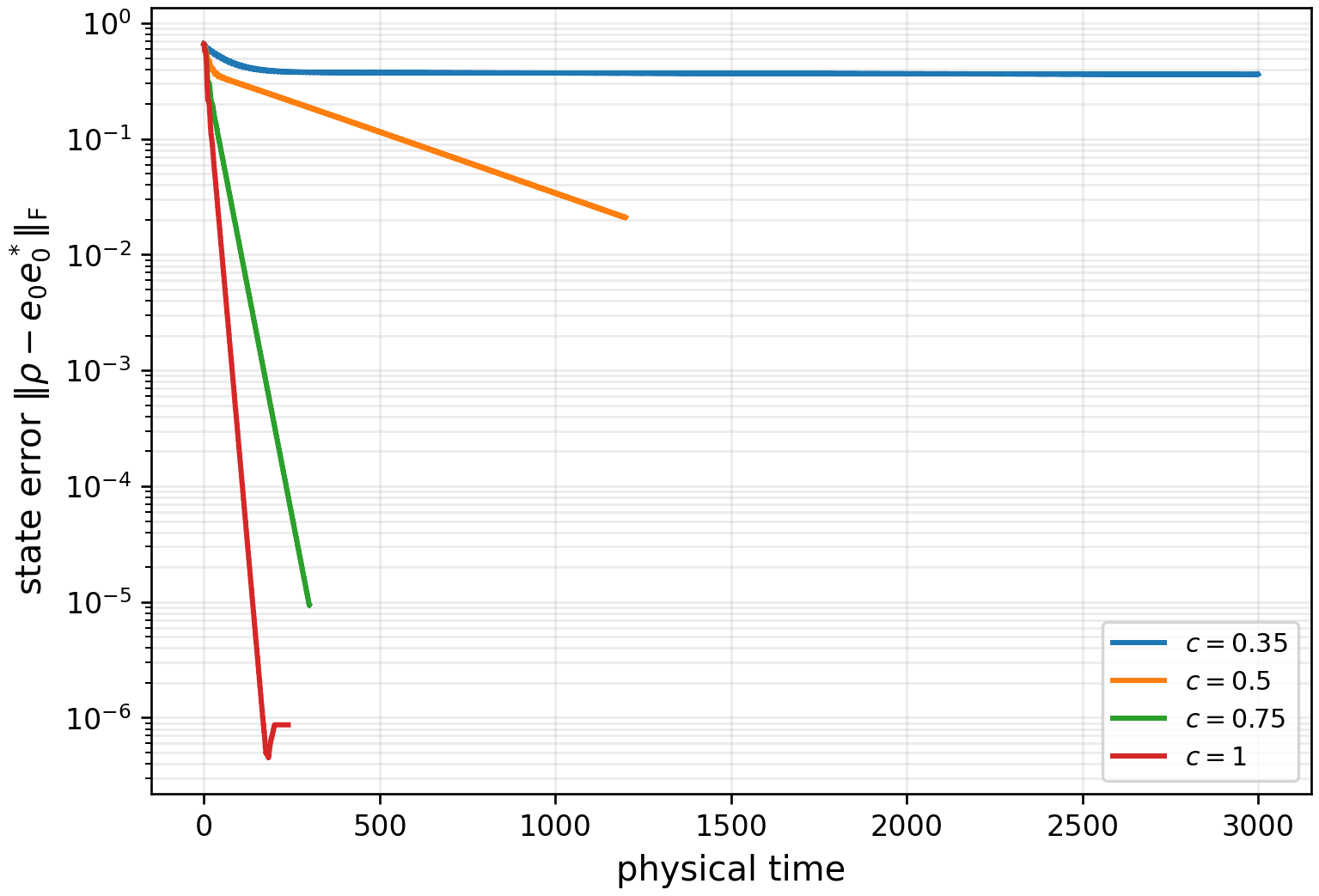}\includegraphics[width=5.5cm]{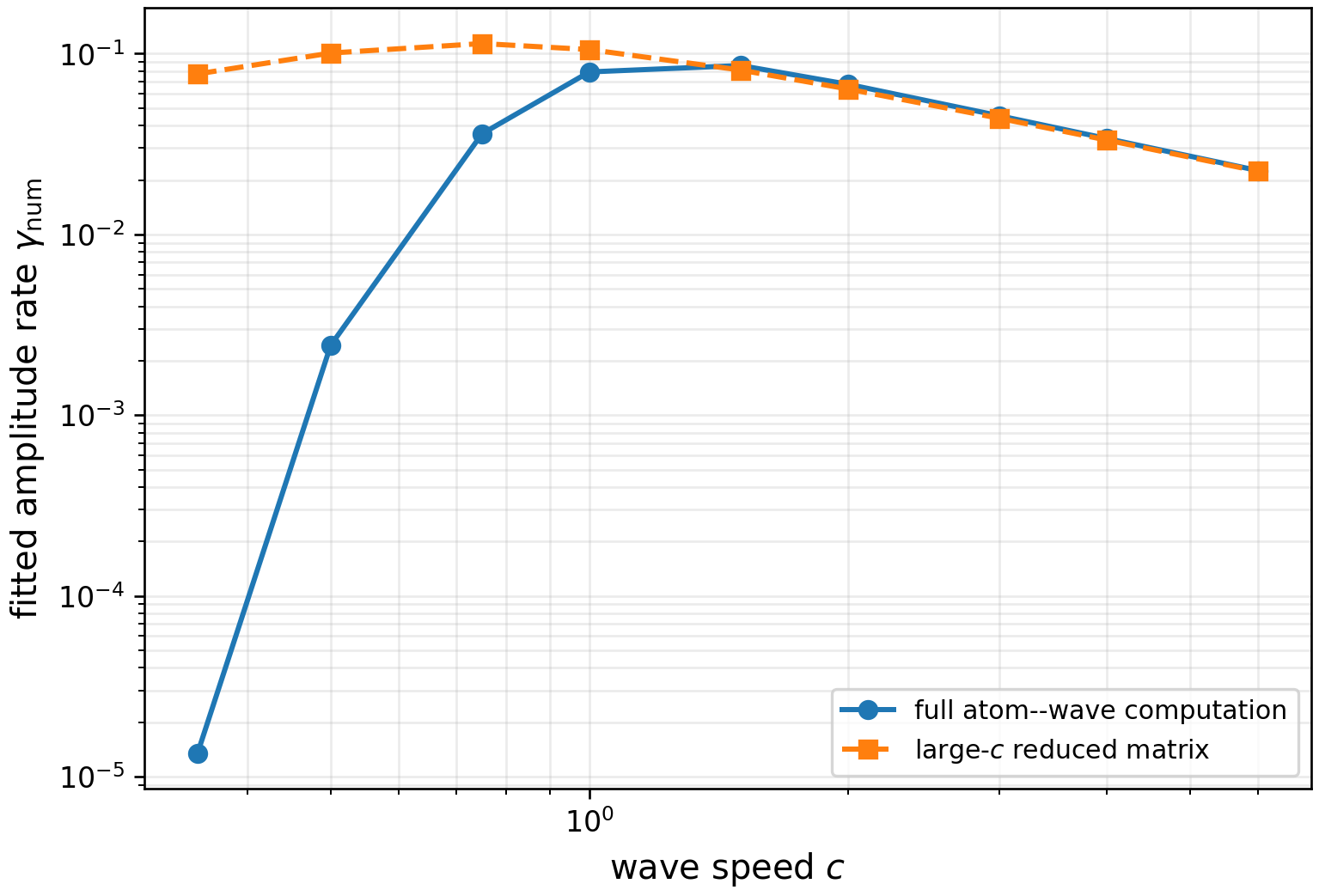}
 \caption{Pure-state computations.  Left: collapse of the ground-state
 error in the rescaled time $t/c$ in the large-$c$ regime.  Middle:
 slowdown at small $c$.  Right: fitted rate for the complete
 atom--wave system and 
 the reduced large-$c$ system  
 }
 \label{fig:V_num_c_sweep}
\end{figure}

 \begin{table}[htbp]
 \centering
 \begin{tabular}{@{}|c|c|c|c|c|c|@{}}
\hline
 $c$ & $T$ & $\gamma_{\rm num}$ & $c\gamma_{\rm num}$
     & $\gamma_{\rm red}$ &  $\|\rho(T)-e_0e_0^*\|$\\
\hline
 $0.35$ & $3000$ & $1.35\,10^{-5}$ & $4.73\,10^{-6}$ & $7.73\,10^{-2}$ & $3.61\,10^{-1}$\\
 $0.50$ & $1200$ & $2.435\,10^{-3}$ & $1.217\,10^{-3}$ & $1.008\,10^{-1}$ & $2.42\,10^{-2}$\\
 $0.75$ & $300$  & $3.603\,10^{-2}$ & $2.702\,10^{-2}$ & $1.137\,10^{-1}$ & $1.60\,10^{-5}$\\
 $1.00$ & $240$  & $7.924\,10^{-2}$ & $7.924\,10^{-2}$ & $1.054\,10^{-1}$ & $8.68\,10^{-7}$\\
 $1.50$ & $200$  & $8.598\,10^{-2}$ & $1.290\,10^{-1}$ & $8.112\,10^{-2}$ & $8.81\,10^{-7}$\\
 $2.00$ & $200$  & $6.770\,10^{-2}$ & $1.354\,10^{-1}$ & $6.376\,10^{-2}$ & $1.23\,10^{-6}$\\
 $3.00$ & $240$  & $4.535\,10^{-2}$ & $1.360\,10^{-1}$ & $4.390\,10^{-2}$ & $1.28\,10^{-5}$\\
 $4.00$ & $320$  & $3.394\,10^{-2}$ & $1.357\,10^{-1}$ & $3.329\,10^{-2}$ & $1.18\,10^{-5}$\\
 $6.00$ & $480$  & $2.256\,10^{-2}$ & $1.354\,10^{-1}$ & $2.236\,10^{-2}$ & $1.14\,10^{-5}$\\
\hline
 \end{tabular}
 \caption{Fitted amplitude rates in the pure-state case}
 \label{tab:V_num_c_sweep}
\end{table}

 \noindent
 {\bf Violation of the weak-coupling condition.}
 We  replace $v=(V_{1/2},V_{3/2})$ by $sv$ and work with $c=4$.   
 The consider the sufficient condition
\begin{equation}
 \chi_{\rm crude}(s)
 =\frac{2\kappa s^2|v|^2}{\hbar\omega_{1/2}}<1,
 \label{eq:V_num_crude_condition}
\end{equation}
whose critical multiplier is $s_{\rm crude}=1.30015$,
and the sharp condition  \[
 \chi(s)=\ds\frac{2\kappa s^2}{\hbar}
 \left(\ds\frac{V_{1/2}^2}{\omega_{1/2}}
       +\frac{V_{3/2}^2}{\omega_{3/2}}\right)<1,
 \]
with the threshold $s_*=1.41331$.
Results are given in Figure~\ref{fig:V_num_weak_sweep} and we observe a neat 
 distinction.  At $s=1.34$
and $s=1.40$, condition~\eqref{eq:V_num_crude_condition} is violated but
$\chi<1$; the numerical solution still converges to $e_0e_0^*$.  At $s=1.42$,
where $\chi=1.00949$, the solution first approaches the ground state and then
departs along the weakly unstable direction.  For larger $s$ it rapidly
settles to a nonzero-field  state.
This new state is not merely a numerical plateau.  It can be predicted by solving
the finite-dimensional self-consistency equations
\begin{equation}
 (H_0+\Phi^\infty V)a=\lambda a,
 \qquad
 \Phi^\infty=-\kappa (Va)^\top a,
 \qquad \|a\|=1.
 \label{eq:V_num_dressed_equilibrium}
\end{equation}
The numerical tail and the solution of~\eqref{eq:V_num_dressed_equilibrium}
agree as shown in Table~\ref{tab:V_num_weak_sweep}.  The sign of $\Phi^\infty$ selects
one of two symmetry-related branches; the table displays its modulus.
The simulations therefore suggest that a qualitative change occurs  according to \eqref{small}: the pure ground state  becomes unstable
and nontrivial self-consistent equilibria bifurcate from it.
\\

\begin{figure}[htbp]
 \centering
 \includegraphics[width=7cm]{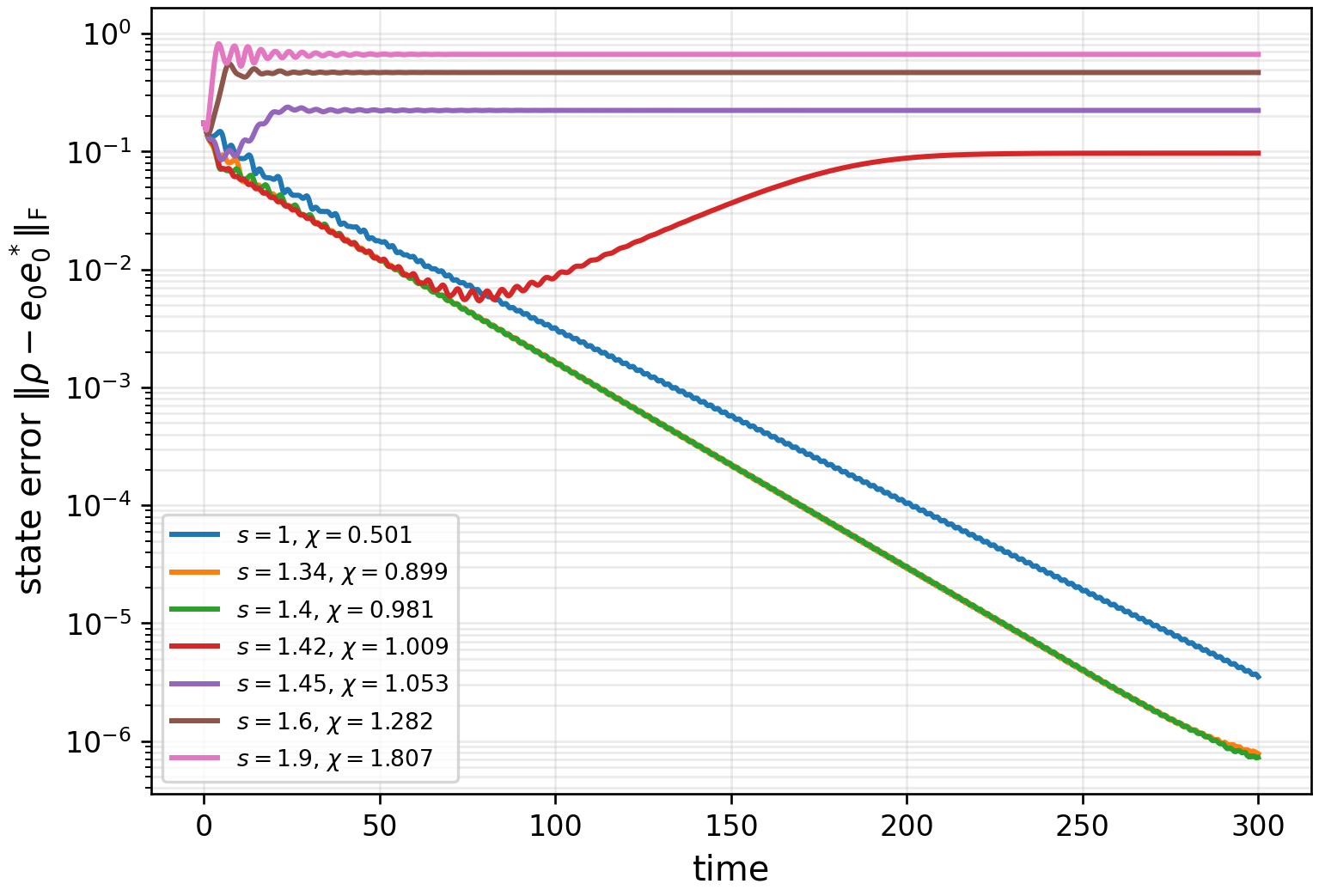}~\includegraphics[width=7cm]{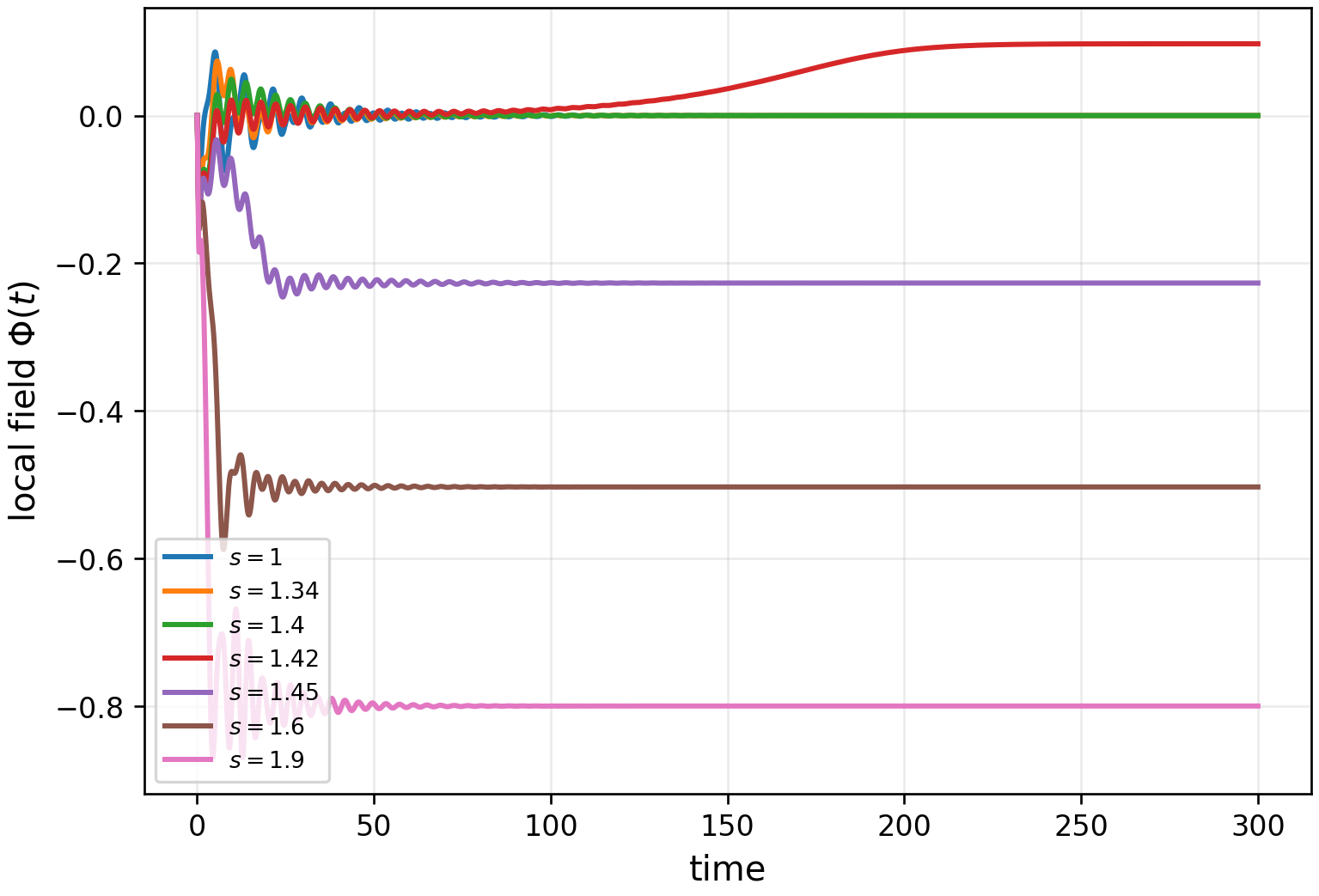}
 \includegraphics[width=7cm]{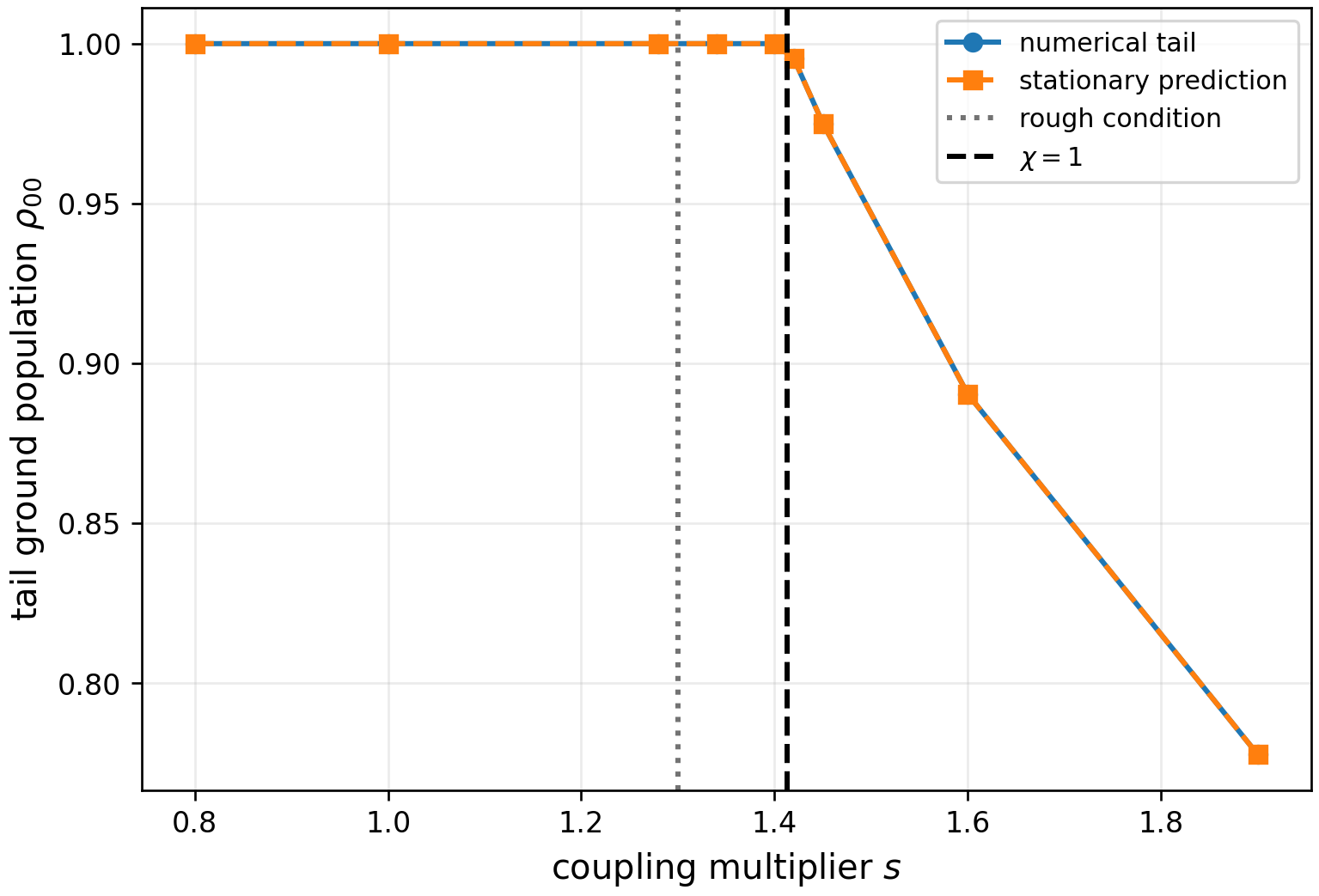}~\includegraphics[width=7cm]{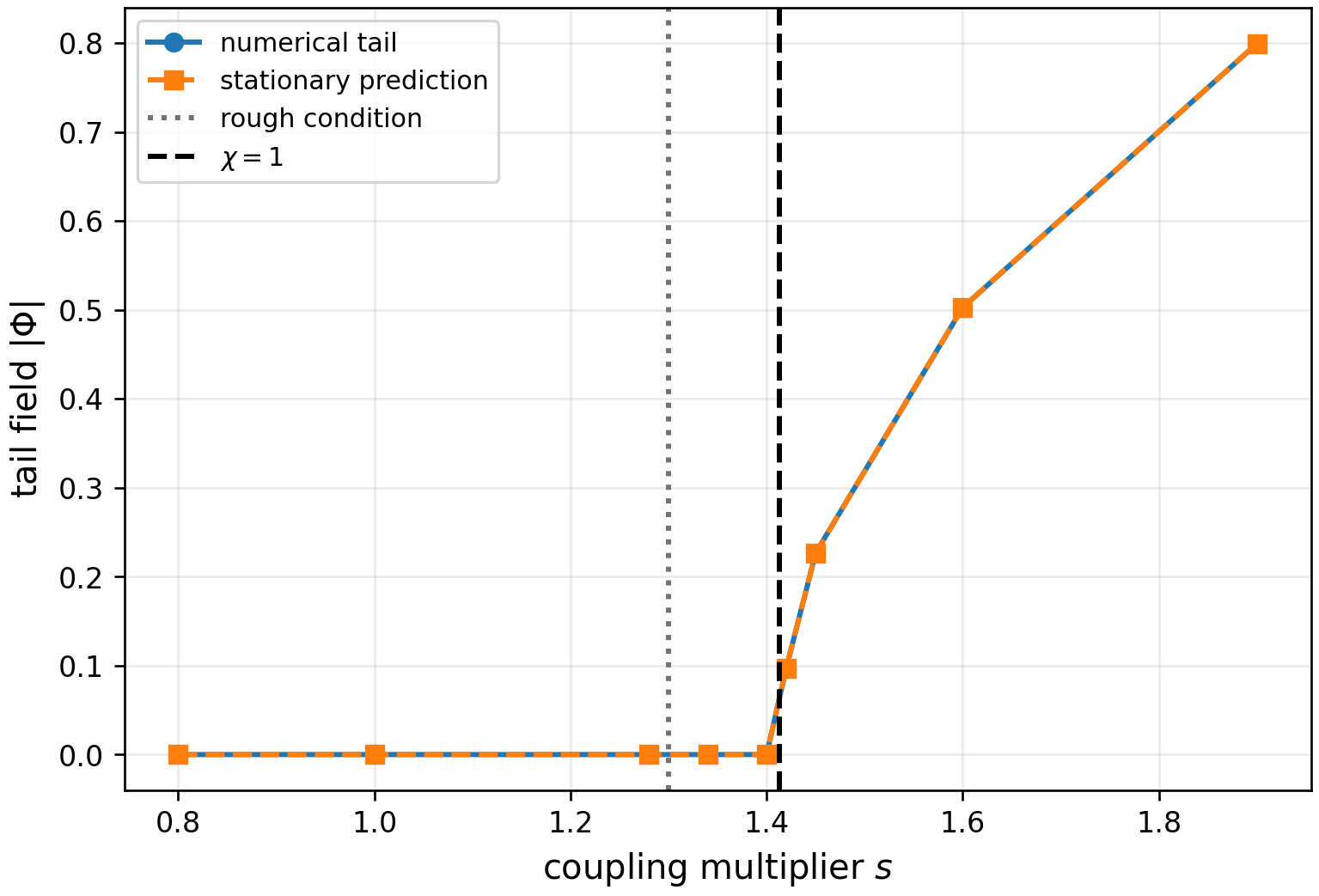}
 \caption{Coupling sweep at $c=4$.  The dotted vertical line is the threshold
~\eqref{eq:V_num_crude_condition}; the dashed line is
 the sharp value $\chi=1$.  Above the latter value, the ground state is
 replaced by a  stationary state with nonzero field.
 Here $ |\Phi|^{\mathrm{tail}}$ is the median   of $|\Phi(t)|, $
for $.9T\leq t\leq T$.}
 \label{fig:V_num_weak_sweep}
\end{figure}

\begin{table}[htbp]
 \centering
 \begin{tabular}{@{}|c|c|c|c|c|c|@{}}
 \hline
 $s$ & $\chi_{\rm crude}$ & $\chi$ & $\rho_{0}(T)$
     & $|\Phi^{\rm tail}|$ & $|\Phi|^{\rm stat}=|\kappa (Va)^\top a|$\\
 \hline
 $1.28$ & $0.9693$ & $0.8202$ & $1.000000$ & $2.66\,10^{-7}$ & $0$\\
 $1.34$ & $1.0622$ & $0.8989$ & $1.000000$ & $2.44\,10^{-7}$ & $0$\\
 $1.40$ & $1.1595$ & $0.9813$ & $1.000000$ & $2.50\,10^{-7}$ & $0$\\
 $1.42$ & $1.1929$ & $1.0095$ & $0.995302$ & $0.097218$ & $0.097228$\\
 $1.45$ & $1.2438$ & $1.0526$ & $0.975030$ & $0.226718$ & $0.226718$\\
 $1.60$ & $1.5145$ & $1.2816$ & $0.890377$ & $0.502809$ & $0.502809$\\
 $1.90$ & $2.1356$ & $1.8073$ & $0.777695$ & $0.799428$ & $0.799428$\\
 \hline
 \end{tabular}
 \caption{Weak-coupling sweep.  The stationary predictions also reproduce
 the three populations; only the ground population and field are displayed.
 Here, $ |\Phi|^{\mathrm{tail}}$ is the median   of $|\Phi(t)|, $
for $.9T\leq t\leq T$.}
 \label{tab:V_num_weak_sweep}
\end{table}

 \noindent
 {\bf Mixed states: persistence and creation of $\rho_{12}$.}
 Consider two states with the same
spectrum
\[
\sigma(\rho(0))= \{0.62,0.28,0.10\}.
\]
The first is diagonal, the second is defined by
$
 \rho(0)=U_{0.6,0.4}\rho_{\mathrm{ref}}U_{0.6,0.4}^{*},
 $
with unitary rotation
$$
U_{\theta,\varphi}
=
\begin{pmatrix}
1&0&0\\
0&\cos\theta&-e^{-i\varphi}\sin\theta\\
0&e^{i\varphi}\sin\theta&\cos\theta
\end{pmatrix},
$$
It produces a nonzero \(1\text{--}2\) coherence while preserving the spectrum of the reference density matrix.
Both have $\rho_{1/2}=\rho_{3/2}=0$, so that the field is identically zero.
For the first state, $\rho_{X}=0$; for the second, we get a crossed coherence of constant modulus
\[
 |\rho_{X}(t)|
 =(0.28-0.10)|\sin(0.6)\cos(0.6)|
 =0.08388352
\]
for every time.  Thus isospectrality cannot determine the asymptotic
trajectory of a mixed state, even before considering a radiative transient.
This exemplifies the obstruction to a full damping 
of the coherences.

Next, we test possible creation of quantum coherences. To this end, we take
\[
 \rho^\delta(0)
 =\ds\frac12 a_{1/2}a^*_{1/2}
  +\frac12 a_{3/2}^\delta a_{3/2}^{\delta *},
\]
where
\[
 a_{01}=\sqrt{0.55}\,e_0+\sqrt{0.45}\,e_1,
 \qquad
 a_{02}^\delta=\sqrt{0.55}\,e_0
                +e^{i\delta}\sqrt{0.45}\,e_2.
\]
We perform simulation with 8 different values of $\delta$.
For every $\delta$, the initial excited coherence vanishes and the spectrum is
the same:
\[
 \rho_{X}^\delta(0)=0,
 \qquad
\sigma(\rho^\delta(0))=\{0,0.225,0.775\}.
\]
Nevertheless, Figure~\ref{fig:V_num_mixed_creation} shows that a coherence of
size approximately $0.11$ is generated.  At $T=250$, $|\rho_{1/2}|^2+|\rho_{3/2}|^2$ has
decreased up to  $8.9\,10^{-5}$ and $1.37\,10^{-4}$, whereas
\(
 0.109916\leq |\rho_{X}(T)|\leq0.111646.
\)
This variation is modest for the chosen family, but it is many orders of
magnitude larger than the spectral drift of the numerical method.  The
demodulated quantity $e^{-i\omega_Xt}\rho_{12}(t)$ has a tail standard
deviation below $6\,10^{-9}$ in all eight runs.
As already observed in the proof of Lemma~\ref{lem:cesaro}, we have
\[
 e^{-i\omega_Xt}\rho_{X}(t)
 =\rho_{12}(0)
 -\frac{i}{\hbar}\int_0^t e^{-i\omega_Xs}\Phi(s)
 \bigl(V_{1/2}\rho_{1/2}(s)
       -V_{3/2}\rho_{3/2}^*(s)\bigr)\,\mathrm ds.
\]
It shows that the residual complex amplitude depends on 
 the complete radiative history ; it is not
a function of the eigenvalues of $\rho(0)$ alone; it also explains why a coherence absent initially
is created generically.

In these computations, 
$\rho_{0}$ tends to $0.775$: the ground state asymptotically carries the largest conserved eigenvalue.
The remaining \(2\times2\) excited block has the conserved eigenvalues \(0.225\) and \(0\), although its two diagonal populations may both be nonzero because of a persistent excited-state coherence.
its entries must satisfy the exact isospectral relation
\[
 (\rho_{1}-\rho_{2})^2+4|\rho_{X}|^2=(0.225)^2.
 \]
Figure~\ref{fig:V_num_mixed_creation} shows the expected decay of the $\rho_{1/2}, \rho_{3/2}$ (left), but also the creation of the coherence $\rho_X$, with  a final amplitude that depends on the initial state. 
The numerics confirms that 
off-diagonal entry can persist, can be created from zero, and depends on the
transient trajectory.  The appropriate limiting object  has the form
\[
 \begin{pmatrix}
 \rho_0^\infty&0&0\\
 0&\rho_1^\infty&r_Xe^{i\omega_Xt}\\
 0&\overline{r_X}e^{-i\omega_Xt}&\rho_2^\infty
 \end{pmatrix},
\]
and it is not a stationary asymptotic density matrix.

\begin{figure}[htbp]
 \centering
 \includegraphics[width=5.5cm]{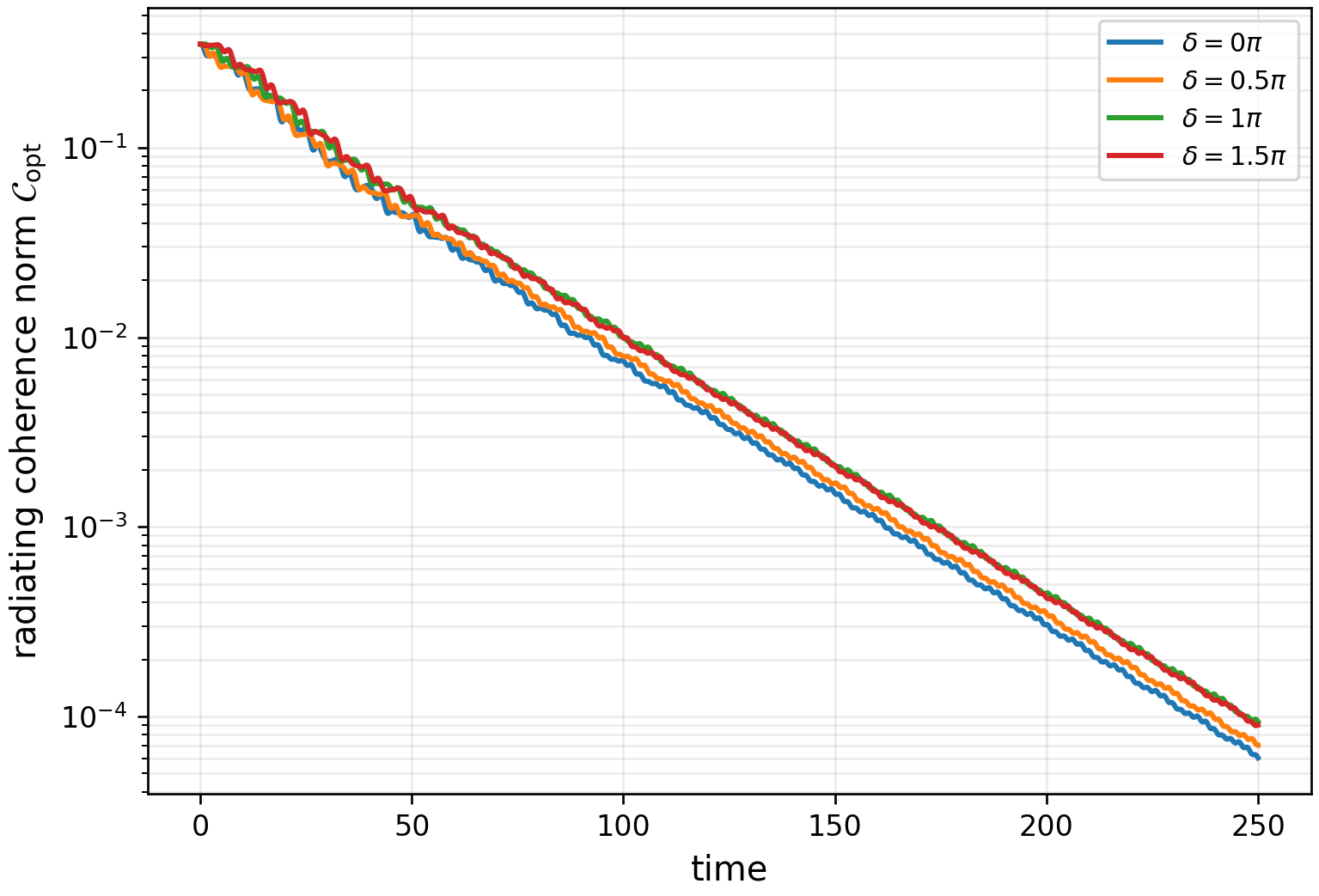}~
 \includegraphics[width=5.5cm]{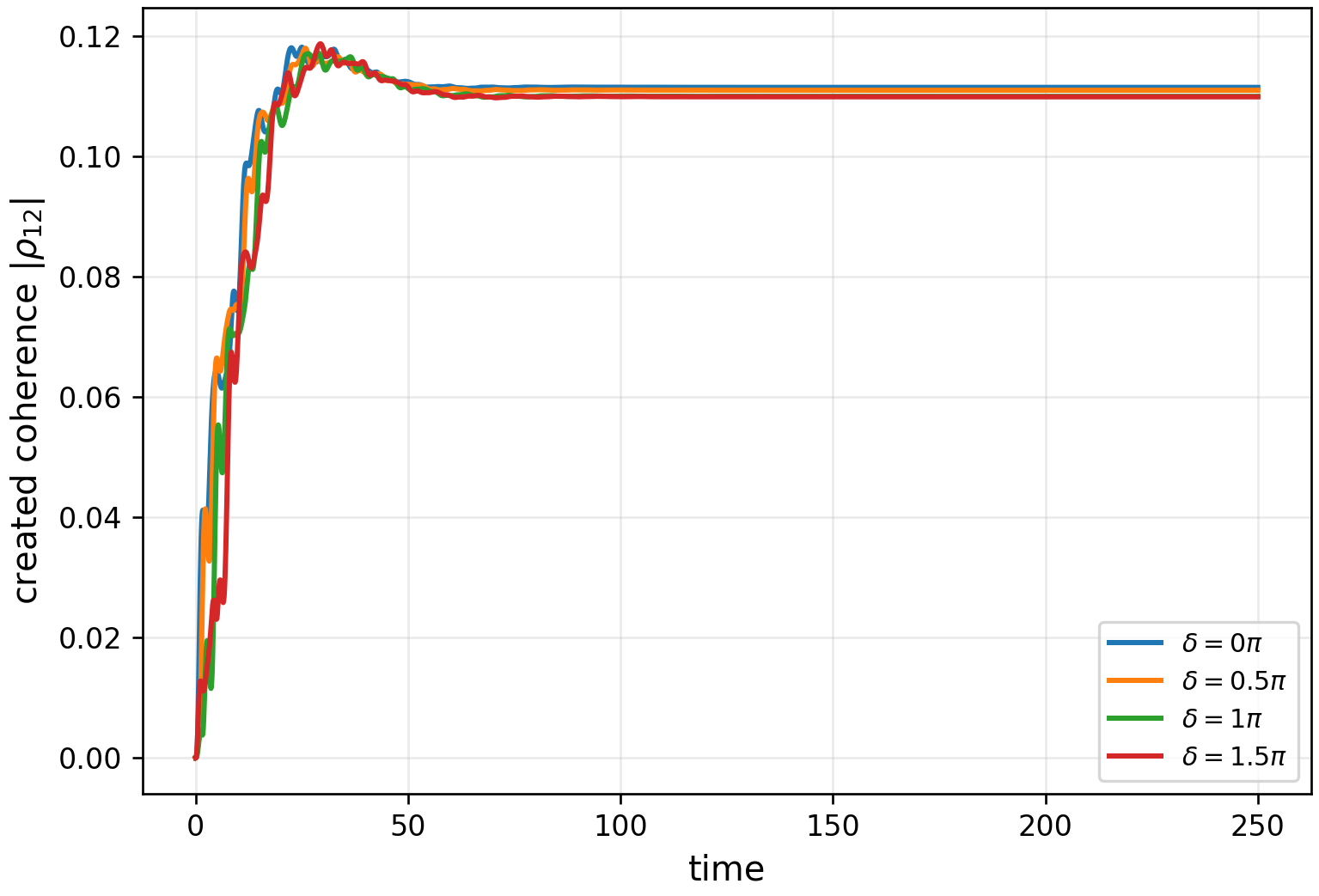}~\includegraphics[width=5.5cm]{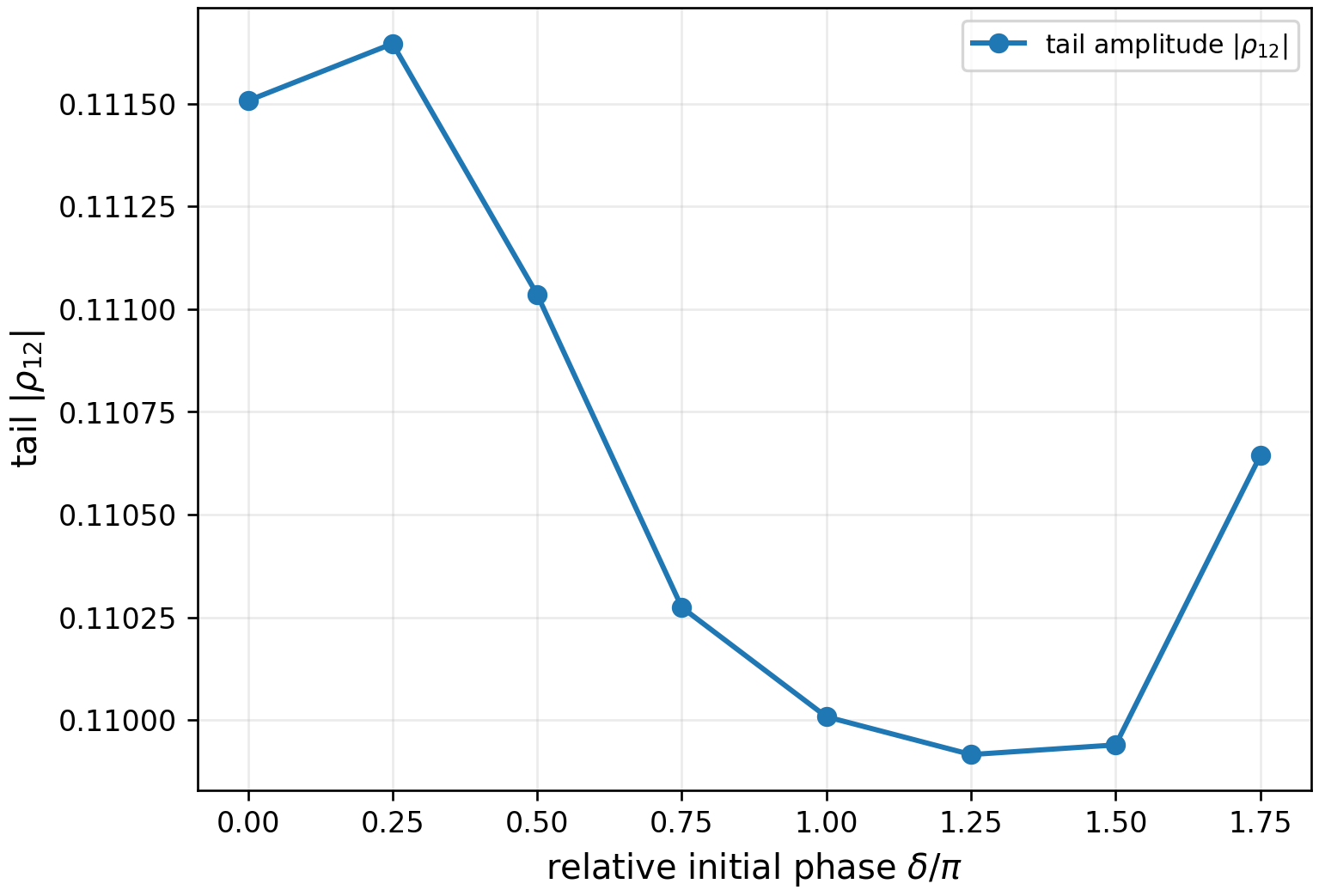}
 \caption{Mixed states at $c=3$.  Left: decay of the 
 $\rho_{1/2}, \rho_{3/2}$ coherences.
   Middle:  creation of $\rho_{X}$ from zero for four isospectral initial states.  Right: dependence
 of the residual modulus on the relative initial phase.}
 \label{fig:V_num_mixed_creation}
\end{figure}

\section*{Acknowledgements}
This work was supported by the ANR project MaDynOS (ANR-24-CE40-3535). Part of this work was carried out during a visit to the Brin Mathematics Research Center at the University of Maryland, with additional support from the CNRS International Research Project KinEq. The author warmly thanks 
Doron Levy and 
the staff of the Center for their hospitality and Antoine Mellet for many fruitful and inspiring discussions.

%
 \bibliography{NStates}
\bibliographystyle{plain}

   \end{document}